\documentclass[11pt]{amsart}
\usepackage{mathrsfs}
\usepackage{amssymb}
\usepackage{graphicx}
\usepackage[T1]{fontenc}
\usepackage[all]{xy}
\usepackage{amscd}
\usepackage{amsmath, amssymb}
\usepackage{amsthm}
\usepackage{amsfonts}
\usepackage[colorlinks,linkcolor=blue,citecolor=blue, pdfstartview=FitH]
{hyperref}
\usepackage{backref}
\usepackage{amscd}
\usepackage{easybmat}
\usepackage{mathrsfs}
\usepackage{amsfonts}
\usepackage{color}
\usepackage{pifont}
\usepackage{upgreek}
\usepackage{bm}
\usepackage{hyperref}
\usepackage{shorttoc}
\usepackage{amsmath,amstext,amsthm,a4,amssymb,amscd}
\usepackage[mathscr]{eucal}
\usepackage{mathrsfs}
\usepackage{epsf}
\usepackage{tikz}

\numberwithin{equation}{section}

\def\p{\partial}
\def\o{\overline}
\def\b{\bar}
\def\mb{\mathbb}
\def\mc{\mathcal}
\def\ms{\mathscr}
\def\mbf{\mathbf}
\def\n{\nabla}
\def\v{\varphi}

\def\wt{\widetilde}
\def\mf{\mathfrak}

\def\op{\operatorname}
\def\ra{\rightarrow}
\def\bc{\mathbb C}
\def\mr{\mathrm}
\def\br{\mathbb R}
\def\bs{\boldsymbol}

\theoremstyle{plain}
\newtheorem{thm}{Theorem}[section]
\newtheorem{lemma}[thm]{Lemma}
\newtheorem{prop}[thm]{Proposition}
\newtheorem{cor}[thm]{Corollary}
\theoremstyle{definition}
\newtheorem{rem}[thm]{Remark}

\newtheorem{ex}[thm]{Example}
\theoremstyle{definition}
\newtheorem{defn}[thm]{Definition}
\newcommand{\comment}[1]{}

\usepackage{amscd}
\usepackage{fancyhdr}
\makeatletter
\@namedef{subjclassname@2020}{2020 Mathematics Subject Classification}

\begin{document}

\title{Signature, Toledo invariant and  surface group representations in the real symplectic group}
\author{InKang Kim}
\author{Pierre Pansu}
\author{Xueyuan Wan}

\address{Inkang Kim: School of Mathematics, KIAS, Heogiro 85, Dongdaemun-gu Seoul, 02455, Republic of Korea}
\email{inkang@kias.re.kr}

\address{Pierre Pansu:
Laboratoire de Math\'ematiques d'Orsay, UMR 8628 du CNRS, Universit\'e Paris-Sud, 91405 Orsay C\'edex, France}
\email{pierre.pansu@math.u-psud.fr}

\address{Xueyuan Wan: Mathematical Science Research Center, Chongqing University of Technology, Chongqing 400054, China}
\email{xwan@cqut.edu.cn}
\begin{abstract}
	In this paper, by using Atiyah-Patodi-Singer index theorem, we obtain a formula for the signature of a flat symplectic vector bundle over a  surface with boundary, which is related to the Toledo invariant of a surface group representation in the real symplectic group and  the Rho invariant on the boundary. As an application, we obtain a Milnor-Wood type inequality for the signature. In particular, we give a new proof of the Milnor-Wood inequality for the Toledo invariant  in the case of closed surfaces and obtain some modified inequalities  for the surface with boundary. 
\end{abstract}

  \thanks{Research by Inkang Kim is partially supported by Grant NRF-2019R1A2C1083865 and KIAS Individual Grant (MG031408), and  Xueyuan Wan is supported by  NSFC (No. 12101093), Scientific Research Foundation of Chongqing University of Technology.}

\maketitle
\tableofcontents

\section*{Introduction}

Let $\Sigma$ be a closed surface, and consider a surface group representation $\rho:\pi_1(\Sigma)\to \op{Sp}(E,\Omega)$ into the real symplectic group $\op{Sp}(E,\Omega)$, where $(E,\Omega)$ is a symplectic vector space. Let $(\mc{E},\Omega)$ denote the flat symplectic vector bundle over $\Sigma$ associated to the representation $\rho$. There is a canonical quadratic form $\int_\Sigma \Omega(\cdot\cup\cdot)$ on the cohomology $\mathrm{H}^1(\Sigma,\mc{E})$. The  Meyer's signature formula \cite[\S 4.1, Page 19]{Meyer} implies that the signature of the quadratic form is given by $4\int_\Sigma c_1(\mc{E},\Omega)$,  where $c_1(\mc{E},\Omega)$ denotes the first Chern class of the symplectic vector bundle $(\mc{E},\Omega)$, which can be expressed in terms of Toledo invariant,
see e.g. \cite[Appendix A]{KP}.  For the case of manifolds with boundary, Atiyah, Patodi and Singer introduced the eta invariant and obtained an index theorem, which is known as Atiyah-Patodi-Singer index theorem, see \cite[Theorem 3.10]{APS}.  
For a  surface $\Sigma$ with boundary $\p\Sigma$, one can consider a unitary representation of the fundamental group $\rho:\pi_1(\Sigma)\to \op{U}(n)$, which gives a flat Hermitian vector bundle over $\Sigma$. Similarly, one can also define a quadratic form on the relative cohomology with coefficients in the flat bundle. From \cite[Theorem 2.2, 2.4]{APSII},  the signature of the quadratic form is exactly the eta invariant (or Rho invariant). Inspired by these results, in this paper, by using Atiyah-Patodi-Singer index theorem,  we will consider the signature of the flat symplectic vector bundle associated to the representation $\rho:\pi_1(\Sigma)\to \op{Sp}(E,\Omega)$ for the surface $\Sigma$ with boundary.

Let $\Sigma$ be a connected oriented surface with smooth boundary $\p\Sigma$, and $g_\Sigma$ be a Riemannian metric on $\Sigma$. Suppose that on the collar neighborhood $ \p\Sigma\times [0,1]\subset \Sigma$  of $\p\Sigma$, the metric has a product form. Let $\rho:\pi_1(\Sigma)\to \op{Sp}(E,\Omega)$ be a surface group representation, which gives a flat symplectic vector bundle $(\mc{E},\Omega)$ over $\Sigma$. Consider the image of twisted singular cohomology  in the absolute cohomology  $\widehat{\mathrm{H}}^1(\Sigma,\mc{E})=\op{Im}(\mathrm{H}^1(\Sigma,\p \Sigma,\mc{E})\rightarrow \mathrm{H}^1(\Sigma,\mc{E}))$, and  a canonical quadratic form $Q_{\mb{R}}(\cdot,\cdot)=\int_\Sigma \Omega(\cdot\cup\cdot)$ on $\widehat{\mathrm{H}}^1(\Sigma,\mc{E})$, which is non-degenerate. We denote by $\op{sign}(\mc{E},\Omega)$ the signature of the quadratic form.
For any $\mbf{J}\in\mc{J}(\mc{E},\Omega)$, the set of compatible complex structures,
 the operator $A_{\mbf{J}}=\mbf{J}\frac{d}{dx}$ is a $\mb{C}$-linear formally self-adjoint elliptic first order differential operator  in the space $A^0(\p\Sigma,\mc{E}_{\mb{C}}|_{\p\Sigma})$, see Proposition \ref{prop2.0}. Hence it  has a discrete spectrum with real eigenvalues, and  the eta invariant $\eta(A_{\mbf{J}})$ is well-defined, see Section \ref{eta}. Let $\n$ be any peripheral connection on $(\mc{E},\Omega,\mbf{J})$, see Definition \ref{connection}, and 
  let $c_1(\mc{E},\Omega, \mbf{J})$ denote the first Chern class of the flat symplectic vector bundle  in the de Rham cohomology with compact support, which is defined as the first Chern class associated to the peripheral connection $\n$, see \eqref{FCC} for its definition. By using Atiyah-Patodi-Singer index theorem \cite[Theorem 3.10]{APS}, we obtain 
\begin{thm}\label{thm0.1}
The signature of the flat symplectic vector bundle $(\mc{E},\Omega)$ is 
\begin{equation}\label{0.1}
  	\op{sign}(\mc{E},\Omega)=4\int_\Sigma c_1(\mc{E},\Omega, \mbf{J})+\eta(A_{\mbf{J}}).
\end{equation}
\end{thm}

For the case of closed surfaces, the above theorem was obtained by Meyer  \cite[\S 4.1, Page 19]{Meyer}, and by Lusztig \cite[Section 2]{Lus} considering the signature of a flat $\op{U}(p,q)$-Hermitian vector bundle. For the case of the surfaces with boundary, and  a surface group representation in $\op{U}(p,q)$, the above theorem was proved by Atiyah \cite[(3.1)]{Atiyah} under the assumption that the representation on each component of the boundary is elliptic. By using the formula of signature for elliptic case,  Atiyah \cite[Theorem 2.13]{Atiyah} proved that the  signature can be expressed  as another kind of formula in terms of the relative Chern class of a certain line bundle for the general case, see \cite[Section 3]{Atiyah} for the proof. In our next paper \cite{KPW}, we will continue to consider  the signature, Toledo invariant, Rho invariant and Milnor-Wood type inequality associated with the surface group representations in the $\op{U}(p,q)$-group.

Next, we will show that the first term $4\int_\Sigma c_1(\mc{E},\Omega, \mbf{J})$ is related to the Toledo invariant $\op{T}(\Sigma,\rho)$. For a closed surface $\Sigma$, the Toledo invariant was defined in \cite{DT, Toledo} by considering  a surface group representation $\rho: \pi_{1}(\Sigma) \rightarrow$ $\operatorname{PSU}(1, n)$ of its fundamental group   in the group of motions of complex hyperbolic $n$-space. In \cite{BIW},  Burger, Iozzi and  Wienhard extended the definition of Toledo invariant to the surface with boundary and  obtained the  Milnor-Wood inequality $|\op{T}(\Sigma,\rho)|\leq \op{rank}(\ms{X})|\chi(\Sigma)|$ by using the methods of bounded cohomology, see  also Section \ref{Toledo} for the definition of Toledo invaraint. More precisely, consider a representation $\rho:\pi_1(\Sigma)\to G$ into a Lie group $G$ which  is of type $(\mathrm{RH})$, so that the associated symmetric space $\mathscr{X}=G/K$
is a Hermitian symmetric space of noncompact type. Let $\omega_\mathscr{X}$ denote the (unique) K\"ahler form such that the minimal holomorphic sectional curvature of associated Hermitian metric is $-1$. The K\"ahler form $\omega_\mathscr{X}$  gives a bounded K\"ahler class $\kappa^b_G\in \mathrm{H}^2_{c,b}(G,\mb{R})$. Consider the pullback in bounded cohomology, $\rho_b^{*}\left(\kappa_{G}^{\mathrm{b}}\right) \in \mathrm{H}_{\mathrm{b}}^{2}\left(\pi_{1}(\Sigma), \mathbb{R}\right) \cong \mathrm{H}_{\mathrm{b}}^{2}(\Sigma, \mathbb{R})$. The canonical $\operatorname{map} j_{\partial \Sigma}: \mathrm{H}_{\mathrm{b}}^{2}(\Sigma, \partial \Sigma, \mathbb{R}) \rightarrow \mathrm{H}_{\mathrm{b}}^{2}(\Sigma, \mathbb{R})$  is an isomorphism, and the Toledo invariant is  defined as
$$
\mathrm{T}(\Sigma, \rho)=\left\langle j_{\partial \Sigma}^{-1} \rho_b^{*}\left(\kappa_{G}^{\mathrm{b}}\right),[\Sigma, \partial \Sigma]\right\rangle,
$$
where $j_{\partial \Sigma}^{-1} \rho_b^{*}\left(\kappa_{G}^{\mathrm{b}}\right)$ is considered as an ordinary relative class and $[\Sigma, \partial \Sigma]$ is the relative fundamental class. If $\p\Sigma=\emptyset$, then the Toledo invariant can be given by $\mathrm{T}(\Sigma, \rho)=\int_\Sigma f^*\omega_{\mathscr{X}}$, where $f:\wt{\Sigma}\to \mathscr{X}$ denotes any $\rho$-equivariant smooth map. For a manifold with  cusps,  
there are also several equivalent definitions for volume invariant, which is a natural generalization of Toledo invariant for higher dimensional manifolds,
 see e.g. \cite{Dun, FK, Kim, KK, KM}.

 In our case, $G=\op{Sp}(E,\Omega)$ and $K$ is the maximal compact subgroup of $G$, which is isomorphic to the unitary group. Then the associated symmetric space $\mathscr{X}=G/K$ can be identified with the bounded symmetric domain $\op{D}^{\op{III}}_n$ of type $\op{III}$,   where $\dim E=2n$. It is also isomorphic to the space $\mc{J}(E,\Omega)$ of all compatible complex structures on $(E,\Omega)$.  Let $\omega_{\op{D}^{\op{III}}_n}$ denote the K\"ahler form with  holomorphic sectional curvature in $[-1,-1/n]$, see Appendix \ref{App2}. For any $\mbf{J}\in \mc{J}(\mc{E},\Omega)$, it is equivalent to  a $\rho$-equivariant  map $\wt{\mbf{J}}:\wt{\Sigma}\to \op{D}^{\op{III}}_n$, see \eqref{tJ}. The form $\wt{\mbf{J}}^*\omega_{\op{D}^{\op{III}}_n}$ is a $\rho$-equivariant  form on $\wt{\Sigma}$ and  descends to a form on $\Sigma$. For each $L\in G$, there exist a $L$-invariant one form $\alpha$ with $d\alpha=\omega_{\op{D}^{\op{III}}_n}$, see \eqref{Fixpoint2} and Section \ref{IKP}. Let $\chi_i\in [0,1]$ be a smooth cut-off function which is equal to {1} near the boundary component $c_i$ and vanishes outside a small collar neighborhood of $c_i$. Then one can define the following de Rham cohomology class with compact support
   \begin{align*}
\left[\rho^*\omega_{\op{D}^{\op{III}}_n}\right]_c=\left[\wt{\mbf{J}}^*	\omega_{\op{D}^{\op{III}}_n}-d(\sum_{i=1}^q\chi_i\wt{\mbf{J}}^*\alpha_i)\right]_c,	
 \end{align*}
where $\alpha_i$ is a $\rho(c_i)$-invariant one form with $d\alpha_i=\omega_{\op{D}^{\op{III}}_n}$, $q$ is the number of components of $\p\Sigma$. Following \cite[Proposition-definition 4.1]{KM}, the class $\left[\rho^*\omega_{\op{D}^{\op{III}}_n}\right]_c$ is independent of $\mbf{J}$ and depends only on the conjugacy class of $\rho$. The Rho invariant is defined by 
\begin{align*}
\boldsymbol{\rho}(\p\Sigma)=\frac{1}{\pi}\sum_{i=1}^q\int_{c_i}\wt{\mbf{J}}^*\alpha_i+\eta(A_{\mbf{J}}),	
\end{align*}
which is a natural generalization of Atiyah-Patodi-Singer Rho invariant for unitary representations, see Remark  \ref{remRho}.  Let $\mc{J}_o(\mc{E},\Omega)$ denote the space of the compatible complex structure $\mbf{J}$, which is the  pullback of a compatible complex structure on $\mc{E}|_{\p\Sigma}$ when restricted to a small collar neighborhood of $\p\Sigma$, see \eqref{J0space}.
If $\mbf{J}\in \mc{J}_o(\mc{E},\Omega)$, then
  the form $\wt{\mbf{J}}^*\omega_{\op{D}^{\op{III}}_n}$ has compact support on $\Sigma_o:=\Sigma\backslash\p\Sigma$. 
\begin{thm}\label{thm0.2}
For any $\mbf{J}\in \mc{J}_o(\mc{E},\Omega)$,  the Toledo invariant $\op{T}(\Sigma,\rho)$ satisfies
	\begin{equation}
 \op{T}(\Sigma,\rho)= \frac{1}{2\pi}\int_\Sigma\left[\rho^*\omega_{\op{D}^{\op{III}}_n}\right]_c=2\int_\Sigma c_1(\mc{E},\Omega, \mbf{J})-\frac{1}{2\pi}\sum_{i=1}^q\int_{c_i}\wt{\mbf{J}}^*\alpha_i.
\end{equation}
Hence the signature can be given by the following formula:
\begin{equation}\label{sign0}
  	\op{sign}(\mc{E},\Omega)=2 \op{T}(\Sigma,\rho)+\bs{\rho}(\p\Sigma).
\end{equation}
\end{thm}
There is a bound for the Toledo invariant, which is known as Milnor-Wood inequality \cite{Milnor, Wood}. It is  provided as an obstruction for a circle bundle to admit a flat structure, see also \cite{Fri, Gold}. This inequality and the maximal representations were widely studied, see e.g. \cite{BIW, BILW, DT, Dup, GGM, KM, KM1, Toledo, Tur}.
Here we will deduce a Milnor-Wood type inequality for  the signature by using the formula  \eqref{sign0}. 
\begin{thm}\label{thm0.3}
The signature satisfies the following Milnor-Wood type inequality:
	\begin{equation*}
  |\op{sign}(\mc{E},\Omega)|\leq \dim E \cdot|\chi(\Sigma)|.
\end{equation*}
\end{thm}
In particular, if $\Sigma$ is closed, then $\op{sign}(\mc{E},\Omega)=2 \op{T}(\Sigma,\rho)$ and we obtain the Milnor-Wood inequality for the Toledo invariant in the case of the surface group representations in the real symplectic group.
\begin{cor}[{Turaev \cite{Tur}}]\label{cor0.4} The Toledo invariant satisfies
	\begin{equation*}
  |\op{T}(\Sigma,\rho)|\leq \frac{\dim E}{2}|\chi(\Sigma)|.
\end{equation*}
\end{cor}
For the case of $\dim E=2$,  $\op{Sp}(E,\Omega)\cong \op{SL}(2,\mb{R})$, we obtain the following modified Milnor-Wood inequalities \eqref{TI0}.
\begin{prop}\label{propT0}
For any representation $\rho:\pi_1(\Sigma)\ra \op{SL}(2,\mb{R})$, one has 
\begin{equation}\label{TI0}
  -|\chi(\Sigma)|-1 + \sum_{\rho(c_k)\text{ is elliptic}}  \frac{\theta_{k}}{\pi}\leq\op{T}(\Sigma,\rho) \leq |\chi(\Sigma)|+1 - \sum_{\rho(c_k)\text{ is elliptic}}  \left(1-\frac{\theta_{k}}{\pi}\right),
\end{equation}
where $\theta_k\in (0,\pi)$ such that $[R(\theta_k)]$ is conjugate to $[\rho(c_k)]\in \op{PSL}(2,\mb{R})=\op{SL}(2,\mb{R})/\{\pm I\}$, and $[\bullet]$ denotes the class in $\op{PSL}(2,\mb{R})$, 
 $R(\theta_k)=\left(\begin{matrix}
\cos\theta_k & -\sin\theta_k\\
\sin\theta_k & \cos\theta_k	
\end{matrix}
\right)$.
\end{prop}
We also consider the representations in the real symplectic group of higher rank, and obtain some modified inequalities   for Toledo invariant, see Section \ref{Class} for more details.

Now we briefly explain the proof of the above results. 

By the complexification, we consider the complex vector bundle $\mc{E}_{\mb{C}}=\mc{E}\otimes \mb{C}$ which is equipped with the symplectic form $\Omega$.  The quadratic form $Q_{\mb{C}}(\cdot,\cdot)=\int_\Sigma \Omega(\cdot \cup\cdot)$ can be defined on $\widehat{\mathrm{H}}^1(\Sigma,\mc{E}_{\mb{C}})$, which is also non-degenerate.  We denote by $\op{sign}(\mc{E}_{\mb{C}},\Omega)$ the signature of the quadratic form on $\widehat{\mathrm{H}}^1(\Sigma,\mc{E}_{\mb{C}})$, which equals  $\op{sign}(\mc{E},\Omega)$ for the real symplectic vector bundle $(\mc{E},\Omega)$, see Proposition \ref{prop3.0}. Let $g_\Sigma$ be a Riemannian metric on $\Sigma$ which has the product form $g_\Sigma=du^2+g_{\p\Sigma}$ on $\p\Sigma\times [0,1]$. The Hodge $*$ operator  is defined by \eqref{star}. For any complex structure $\mbf{J}\in \mc{J}(\mc{E},\Omega)$, one can define the Hermitian inner product $(\cdot,\cdot)=2\Omega(\cdot,\mbf{J}\o{\cdot})$ and the global $\mathrm{L}^2$-inner product $\langle\cdot,\cdot\rangle$. The operator $*\mbf{J}$ satisfies $(*\mbf{J})^2=\op{Id}$ when acting on the space  $\wedge^1T^*\Sigma\otimes \mc{E}_{\mb{C}}$. Let $\wedge^\pm$ denote the $\pm 1$-eigenspaces of $*\mbf{J}$, and denote by $\pi^{\pm}=\frac{1\pm *\mbf{J}}{2}:\wedge^1T^*\Sigma\otimes \mc{E}_{\mb{C}}\to \wedge^{\pm}$ the natural projections. Set $d^{\pm}=\pi^{\pm}\circ d$, where $d$ is the canonical flat connection on the space $A^*(\Sigma,\mc{E}_{\mb{C}})$. The   operators $d^{\pm}$ have  the  form 
$d^{\pm}=\sigma^{\pm}(\frac{\p}{\p u}+A^{\pm}_{\mbf{J}}),$
where $\sigma^{\pm}:\mc{E}_{\mb{C}}\to\wedge^\pm$ are the bundle isomorphisms, and $A^{\pm}_{\mbf{J}}=\pm A_{\mbf{J}}$ are the first order  elliptic formally self-adjoint operators on the boundary, see Corollary \ref{cor4.2}.  Let  
 $\widehat{\Sigma}=\Sigma\cup((-\infty,0]\times\p\Sigma)$
 be the complete manifold obtained form $\Sigma$ by gluing the negative half-cylinder $(-\infty,0]\times \p\Sigma$ to the boundary of $\Sigma$. The vector bundle $\mc{E}$, the complex structure $\mbf{J}$ and the canonical flat connection $d$  can be extended naturally to $\widehat{\Sigma}$. Denote by $\mathscr{H}^*(\widehat{\Sigma},\mc{E}_{\mb{C}})$ the space of harmonic $\mathrm{L}^2$-forms on $\widehat{\Sigma}$, which is isomorphic to $\widehat{\mathrm{H}}^*(\Sigma,\mc{E}_{\mb{C}})$, see \eqref{L2harmonic}. Moreover, it is the direct sum of the two subspaces $\op{Ker}(d^+)^*\cap \mathrm{L}^2(\widehat{\Sigma},\wedge^+)$ and  $\op{Ker}(d^-)^*\cap \mathrm{L}^2(\widehat{\Sigma},\wedge^-)$, which correspond to the positive and negative definite subspaces of quadratic form $Q_{\mb{C}}$ respectively, see Proposition \ref{prop3}. The signature is then given by 
 \begin{align*}
		\op{sign}(\mc{E},\Omega)&=\mathrm{L}^2\op{Index}(d^{-})-\mathrm{L}^2\op{Index}(d^{+}),
	\end{align*}
where $\mathrm{L}^2\op{Index}(d^{\pm})$ denote the $\mathrm{L}^2$-indices  of the operators $d^{\pm}$, see \eqref{L2index}.  Let $P_{\pm}$ denote the orthogonal projections of $\mathrm{L}^2(\p\Sigma,\mc{E}_{\mb{C}})$ onto the subspace spanned by all eigenfunctions of $A^{\pm}_{\mbf{J}}$ with eigenvalues $\lambda>0$, and $A^0(\Sigma,\mc{E}_{\mb{C}};P_\pm)$ be the subspace of $A^0(\Sigma,\mc{E}_{\mb{C}})$ consisting of all sections $\v$ which satisfy the boundary conditions
$P_\pm(\v|_{\p\Sigma})=0.$
Denote by 
$d^{\pm}_{P}:A^0(\Sigma,\mc{E}_{\mb{C}};P_\pm)\to A^0(\Sigma,\wedge^{\pm})$
the restriction of $d^{\pm}$. The $\mathrm{L}^2$-index of $d^{\pm}$ can be expressed as the sum of $\op{Index}(d^-_P)$ and $h_{\infty}(\wedge^\pm)$,
 where $h_{\infty}(\wedge^\pm)$ denote the dimension of the subspace of $ \operatorname{Ker} \sigma^{\pm}A_{\mbf{J}}^{\pm}(\sigma^\pm)^{-1}$  consisting of limiting values of extended $\mathrm{L}^2$-sections $a$ of $\wedge^{\pm}$ satisfying $(d^{\pm})^*a=0$. Hence 
	\begin{align*}
		\op{sign}(\mc{E},\Omega)
		=\op{Index}(d^-_P)-\op{Index}(d^+_P)+h_{\infty}(\wedge^-)-h_{\infty}(\wedge^+).
	\end{align*}
 By Atiyah-Patodi-Singer index theorem \cite[Theorem 3.10]{APS}, $d^{\pm}_P$ are Fredholm operators and 
\begin{align*}
\op{Index}(d^{\pm}_P)=\int_{\Sigma}\alpha_{{\pm}}(z)d\mu_g-\frac{\eta(A_{\mbf{J}}^{\pm})+\dim\op{Ker}A^{\pm}_{\mbf{J}}}{2},	
\end{align*}
where $\int_{\Sigma}\alpha_{{\pm}}(z)d\mu_g$ are the Atiyah-Singer integrands. Hence
\begin{multline*}
\op{sign}(\mc{E},\Omega)=\int_{\Sigma}\alpha_-(z)d\mu_g-\int_{\Sigma}\alpha_+(z)d\mu_g+h_{\infty}(\wedge^-)-h_{\infty}(\wedge^+)+\eta(A_{\mbf{J}}).
\end{multline*}
Following \cite{APS} we consider the double $\Sigma\cup_{\p\Sigma}\Sigma$ of $\Sigma$, and consider  the $\mb{Z}_2$-graded vector bundle  $\mc{F}=\mc{F}^+\oplus\mc{F^-}$ over the double $\Sigma\cup_{\p\Sigma}\Sigma$, where 
$\mc{F}^+=\mc{E}_{\mb{C}}$ and $\mc{F}^-=\wedge^-$. Then the term, Atiyah-Singer integrand  can be given by 
\begin{align*}
\int_{\Sigma}\alpha_-(z)d\mu_g =\lim_{t\to 0}\int_{\Sigma}\op{Str}\langle z|e^{-tD^2}|z\rangle d\mu_g,
\end{align*}
where $D$ is a Dirac operator on $\mc{F}$, see \eqref{Dirac operator}. For any flat symplectic vector bundle $(\mc{E},\Omega)$ with $\mbf{J}\in\mc{J}(\mc{E},\Omega)$, we define the  {\it peripheral connection} to be a connection which commutes with $
\mbf{J}$, preserves the symplectic form, and depends only $x$ on a small collar neighborhood of $\p\Sigma$, see Definition \ref{connection}. For each peripheral connection, there is a natural Dirac operator $D^{\mc{F}}$  on $\mc{F}$, which is associated with a Clifford connection $\n^\mc{F}$. The Duhamel's formula gives
$$\lim_{t\to 0}\int_{\Sigma}\op{Str}\langle z|e^{-tD^2}|z\rangle d\mu_g=\lim_{t\to 0}\int_{\Sigma}\op{Str}\langle z|e^{-t(D^{\mc{F}})^2}|z\rangle d\mu_g,$$
see Lemma \ref{lemma4.7}. By the local index theorem, see e.g. \cite[Theorem 8.34]{Melrose}, we obtain 
\begin{align*}
	\lim_{t\to 0}\int_{\Sigma}\op{Str}\langle z|e^{-t(D^{\mc{F}})^2}|z\rangle d\mu_g=2\int_{\Sigma}c_1(\mc{E},\Omega, \mbf{J})+\frac{\dim E}{2}\chi(\Sigma),
\end{align*}
see Proposition \ref{prop410}. Similarly, we can calculate the term $\int_{\Sigma}\alpha_{{+}}(z)d\mu_g$. Following a similar  argument in \cite{APS}, we obtain 
$h_{\infty}(\wedge^-)=h_{\infty}(\wedge^+)$, see \eqref{hpm}. Combining with the above equalities, Theorem \ref{thm0.1} is proved.

For the trivial symplectic vector bundle $(F,\Omega)=\op{D}^{\op{III}}_n\times (\mb{R}^{2n},\Omega)$ over the bounded symmetric domain $\op{D}^{\op{III}}_n$ of type $\op{III}$, there is a canonical complex structure $\mbf{J}_F$ on $F$, see \eqref{complexstructureF}. Moreover, we can define a complex connection $\n^F$ on $(F,\mbf{J}_F)$ such that the first Chern form of the connection is  $\frac{1}{4\pi}\omega_{\op{D}^{\op{III}}_n}$, see \eqref{Chern forms 1}. For any representation $\rho:\pi_1(\Sigma)\to \op{Sp}(E,\Omega)$ and any $\mbf{J}\in \mc{J}_o(\mc{E},\Omega)$, it gives a $\rho$-equivariant map $\wt{\mbf{J}}: \wt{\Sigma}\to \op{D}_n^{\op{III}}$ by using the  identification $\mc{J}(E,\Omega)\cong \op{D}_n^{\op{III}}$, which is also equivalent to the smooth map $\mbf{J}:\Sigma\to \wt{\Sigma}\times_\rho \op{D}_n^{\op{III}}$. For the vector bundle $F_\rho=\wt{\Sigma}\times_\rho F$ over  $\wt{\Sigma}\times_\rho \op{D}_n^{\op{III}}$,  the identification  $(E,\Omega)\cong (\mb{R}^{2n},\Omega)$ induces a complex linear symplectic isomorphism between $(\mbf{J}^*F_\rho,\mbf{J}^*\mbf{J}_{F_\rho},\mbf{J}^*\Omega)$ and $(\mc{E},\mbf{J},\Omega)$, where $\mbf{J}^*\Omega$ denotes the induced symplectic form on $F_\rho$. The connection $\n^F$ induces a natural connection $\n^{F_\rho}$ on $F_\rho$, and by pullback, the connection $\mbf{J}^*\n^{F_\rho}$ can be proved to be a peripheral connection, see Proposition \ref{proppullbackconnection}. On the other hand, the invariant K\"ahler form $\omega_{\op{D}^{\op{III}}_n}$ is also well-defined on $\wt{\Sigma}\times_\rho \op{D}^{\op{III}}_n$, which is just the curvature of the connection $\n^{F_\rho}$ up to a factor.  The pullback two form $\wt{\mbf{J}}^*\omega_{\op{D}^{\op{III}}_n}$ is a $\rho$-equivariant two form on $\wt{\Sigma}$, so it descends  to a two form on $\Sigma$, which is just $\mbf{J}^*\omega_{\op{D}^{\op{III}}_n}$.
 Note that $c_1(\mc{E},\Omega, \mbf{J})$ is independent of the peripheral connection, thus
 \begin{equation*}
  2\int_\Sigma c_1(\mc{E},\Omega, \mbf{J})=\frac{1}{2\pi}\int_\Sigma \wt{\mbf{J}}^*\omega_{\op{D}^{\op{III}}_n},
 \end{equation*}
see Proposition \ref{propchern}. By considering the specific corresponding between the groups of  bounded cohomology and de Rham cohomolgy, we obtain
$$\op{T}(\Sigma,\rho)=\frac{1}{2\pi}\int_\Sigma \wt{\mbf{J}}^*	\omega_{\op{D}^{\op{III}}_n}-\frac{1}{2\pi}\sum_{i=1}^q\int_{c_i}\wt{\mbf{J}}^*\alpha_i,$$
and Theorem \ref{thm0.2} is proved. 

The signature also can be given by 
\begin{multline*}
 \pm\op{sign}(\mc{E},\Omega)	=-\dim E \cdot\chi(\Sigma)-\dim \mathrm{H}^0(\p\Sigma,\mc{E})\\
+2\dim \mathrm{H}^0(\Sigma,\mc{E})-2\dim_{\mb{C}}\op{Ker}(d^\mp)^*\cap \mathrm{L}^2(\widehat{\Sigma},\wedge^{\mp}),
\end{multline*}
see \eqref{5.6}. Since $\dim \mathrm{H}^0(\p\Sigma,\mc{E})\geq q\dim \mathrm{H}^0(\Sigma,\mc{E})$ and so if $q\geq 2$, we conclude $|\op{sign}(\mc{E},\Omega)|\leq \dim E |\chi(\Sigma)|$. For a general $q\geq 1$,  and for any representation $\rho:\pi_1(\Sigma)\to \op{Sp}(E,\Omega)$, by a small perturbation with the representation being fixed on the boundary, we obtain a smooth family $\{\rho_\epsilon\}_{\epsilon>0}$ such that the associated signature is invariant and $\rho_{\epsilon}\to\rho$ as $\epsilon\to 0$. Moreover $\dim \mathrm{H}^0(\Sigma,\mc{E}_\epsilon)=0$ for any $\epsilon>0$, see Lemma \ref{lemma3}. Therefore, we prove Theorem \ref{thm0.3}. 

We also consider the case of $n=1$, i.e. $\rho:\pi_1(\Sigma)\to \op{SL}(2,\mb{R})$. In this case, each element $L$ has the form $L=\pm \exp(2\pi B)$ for some $B\in \mf{s}\mf{l}(2,\mb{R})$, we can define a canonical compatible complex structure by $\mbf{J}(x)=\exp(-xB)J\exp(xB)$ for each boundary. Then the eta invariant of $A_{\mbf{J}}$ and Rho invariant can be calculated specifically, see Appendix \ref{Appeta}.  By reversing the maximal even representations on boundary  which are conjugate to the rotation $R(\theta), \theta\in (\pi,2\pi)$ or are parabolic with eigenvalues only $1$, we obtain a new representation $\rho_1$,
 and the Toledo invariant is invariant, i.e. $\op{T}(\Sigma,\rho)=\op{T}(\Sigma,\rho_1)$. By using  \eqref{sign0} and Theorem \ref{thm0.3} to $\rho_1$, Proposition \ref{propT0} is proved, see Corollary \ref{MWH}, Proposition \ref{propmod2} and Proposition \ref{propmod3}.

 This article is organized as follows. In Section \ref{Sym}, we will consider the symplectic vector bundle and eta invariant over a circle $S^1$. In Section \ref{Sig}, we will define the signature for flat symplectic vector bundles and prove Theorem \ref{thm0.1} using Atiyah-Patodi-Singer index theorem.  In Section \ref{Tol}, the first Chern class in the formula \eqref{0.1} of signature is proved to be related to Toledo invariant and Theorem \ref{thm0.2} is proved. In Section \ref{AMW}, we will prove a Milnor-Wood type inequality for signature and obtain Theorem \ref{thm0.3}, Corollary  \ref{cor0.4}. We also give a classification of elements in $\op{Sp}(2n,\mathbb R)$ and obtain some modified inequalities, Proposition \ref{propT0} is proved in this section. In Appendix, we will calculate the eta invariant and Rho invariant for two dimensional symplectic vector spaces, and calculate the curvature of the bounded symmetric domain of type $\op{III}$.

\section{Symplectic vector bundle and eta invariant over a circle}\label{Sym}

In this section, we will consider the flat symplectic vector bundle $\mc{E}$ associated with a representation $\rho$ of the fundamental group of a circle $S^1$ into the real symplectic group  $\op{Sp}(E,\Omega)$, and  define a first order elliptic self-adjoint differential operator $A_{\mbf{J}}$, and we will recall the definition of eta invariant $\eta(A_{\mbf{J}})$ for the operator, one can refer to \cite{APS, Muller} for the eta invariant.

\subsection{Symplectic vector bundle}
Let $(E,\Omega)$ be a real symplectic vector space, where $\Omega$ is a symplectic form. For any symplectic linear transformation $L\in \op{Sp}(E,\Omega)$, i.e. $L$ preserves the symplectic form $\Omega$, and consider the representation $\rho(\gamma_0)=L$, where $\gamma_0$ denotes the generator of $\pi_1(S^1)$, which is given by $\gamma_0(x)=e^{ix},0\leq x\leq 2\pi$, then it defines a flat vector bundle
$$\mc{E}=\mb{R}\times_{\rho}E=(\mb{R}\times E)/\sim$$
over $S^1$, where $(x_1,e_1)\sim (x_2,e_2)$ if $x_2=x_1+2\pi k,k\in \mb{Z}$ and $e_2=L^{-k}(e_1)$. Each global section of $\mc{E}$ is equivalent to a map $s:\mb{R}\to E$ satisfies the $\rho$-equivariant condition $s(x+2\pi)=L^{-1}s(x)$.  

 A complex structure on a real vector space $E$ is an automorphism $J: E \rightarrow E$ such that $J^{2}=-$ Id. A complex structure $J$ on a real symplectic vector space $(E, \Omega)$ is called {\it compatible with $\Omega$} (or {\it $\Omega$-compatible}) if
\(\Omega(\cdot, J \cdot)
\)
defines a positive definite inner product. 
We denote by $\mc{J}(E,\Omega)$ the space of all $\Omega$-compatible complex structures on $(E,\Omega)$.
In particular, one has $\mc{J}(E,\Omega)\subset \op{Sp}(E,\Omega)$.

Denote by $\mc{J}(\mc{E},\Omega)=\wt{\Sigma}\times_\rho \mc{J}(E,\Omega)$ the space of all $\rho$-equivariant complex structures in $\mc{J}(E,\Omega)$.
{For any symplectic transformation $L\in \op{Sp}(E,\Omega)$ which can be written as  $L=\pm\exp(2\pi B)$, where 
 $B\in\mf{s}\mf{p}(E,\Omega)$, i.e. $B^\top\Omega+\Omega B=0$, we can find a canonical complex structure $\mbf{J}\in \mc{J}(\mc{E},\Omega)$ for any given $J\in \mc{J}(E,\Omega)$.
 In fact, for any $J\in\mc{J}(E,\Omega)$, we  define }
 $$\mbf{J}(x)=\exp(-xB)J\exp(xB)\in \mc{J}(E,\Omega),$$
which satisfies $\mbf{J}(x+2\pi)=L^{-1}\mbf{J}(x)L$ and thus gives a complex structure $\mbf{J}$ on the flat vector bundle $\mc{E}$.  Hence $\mbf{J}\in \mc{J}(\mc{E},\Omega)$. { But in general, $L$ cannot be written as $L=\pm\exp(2\pi B)$ except for $\op{Sp}(2,\mb{R})$, hence one needs to choose another complex structure on $\mc{E}$.}

There exists a canonical flat connection $d$ on $\mc{E}$, which is induced from the trivial vector bundle $\mb{R}\times E\to \mb{R}$. The holonomy representation of the flat connection $d$ is just the representation $\rho$.
 Denote by $A^{0}(S^1,\mc{E})$ the space of all smooth sections of $\mc{E}$, which can be identified with the space $A^0(\mb{R},E)^L$ of all $\rho$-equivariant smooth maps $s:\mb{R}\to E$. There is a standard $\mathrm{L}^2$-metric on the space $A^{0}(S^1,\mc{E})\cong A^0(\mb{R},E)^L$ with respect to inner product $\Omega(\cdot,\mbf{J}\cdot)$ and the metric $dx\otimes dx$ on $S^1$, i.e. 
$
 	\int_{S^1}\Omega(\cdot,\mbf{J}\cdot)dx.
$

\subsection{First order differential operator}
Let $(E,\Omega)$ be a real symplectic vector space.
By $\mb{C}$-linear extension, $\Omega$ is also a symplectic form on the complex vector space $E_{\mb{C}}=E\otimes\mb{C}$. For any $L\in \op{Sp}(E,\Omega)$, it can be viewed as an element in $\op{Sp}(E_{\mb{C}},\Omega)$ by $\mb{C}$-linear extension. Since
$$\mc{E}_{\mb{C}}=\mc{E}\otimes\mb{C}=\mb{R}\times_{\rho}E_{\mb{C}},$$
so the space $A^0(S^1,\mc{E}_{\mb{C}})$ of all smooth sections of $\mc{E}_{\mb{C}}$ can be identified with the space $A^0(\mb{R},E_{\mb{C}})^L$ of all smooth $\rho$-equivariant maps from $\mb{R}$ to $E_{\mb{C}}$. For any $\mbf{J}\in \mc{J}(\mc{E},\Omega)$, 
 we can also extend it to a $\mb{C}$-linear transformation of $E_{\mb{C}}$. 

Consider the following $\mb{C}$-linear first order 
differential operator
\begin{align}\label{ABJ1}
A_{\mbf{J}}:=\mbf{J}\frac{d}{dx}
\end{align}
which acts on the space  $A^0(\mb{R},E_{\mb{C}})^L\cong A^0(S^1,\mc{E}_{\mb{C}})$. 
Denote
\begin{align}\label{Hermitian metric}
H(e_1,e_2)=2\Omega(e_1,\mbf{J}e_2).
\end{align}
for any $e_1,e_2\in \mc{E}_{\mb{C}}|_{x}$, $x\in [0,2\pi]$. One can check that $H$ is a Hermitian metric on $\mc{E}_{\mb{C}}$, the Hermitian inner product is denoted by $$(e_1,e_2):=H(e_1,\o{e_2}).$$ The global $\mathrm{L}^2$-inner product on $A^0(S^1,\mc{E}_{\mb{C}})$ is defined as 
\begin{align}\label{L2 metric}
	\langle\cdot,\cdot\rangle=\int_{S^1}(\cdot,\cdot)dx=2\int_{S^1}\Omega(\cdot,\mbf{J}\o{\cdot})dx.
\end{align}

\begin{prop}\label{prop2.0}
$A_{\mbf{J}}$ is a $\mb{C}$-linear formally self-adjoint elliptic first order differential operator  in the space $A^0(S^1,\mc{E}_{\mb{C}})$.
\end{prop}
\begin{proof}
It is obvious that $A_{\mbf{J}}$ is $\mb{C}$-linear, first order and elliptic, so we just need to prove $A_{\mbf{J}}$ is formally self-adjoint. For any $s_1,s_2\in  A^0(S^1,\mc{E}_{\mb{C}})$, one has 
\begin{align*}
	&\quad \langle A_{\mbf{J}}s_1,s_2\rangle-\langle s_1,A_{\mbf{J}}s_2\rangle\\
	&=2\int_{S^1}\left(\Omega\left({\frac{d}{dx}}s_1,\o{s_2}\right)+\Omega\left(s_1,{\frac{d}{dx}}\o{s_2}\right)\right)dx\\
	&=2\int_{S^1}d\Omega(s_1,\o{s_2})=0,
\end{align*}
which completes the proof.
\end{proof}
\begin{rem}
The operator $A_{\mbf{J}}$ has a natural extension in the Hilbert space $\mathrm{L}^2(S^1,\mc{E}_{\mb{C}})$, we also denote it by $A_{\mbf{J}}$, see e.g.  \cite[Definition 7.1 in Appendix]{Kodaira}. From Proposition \ref{prop2.0}, $A_{\mbf{J}}$ is formally self-adjoint and elliptic, so $A_{\mbf{J}}$ is self-adjoint in the Hilbert space $\mathrm{L}^2(S^1,\mc{E}_{\mb{C}})$, see e.g.  \cite[Theorem 7.2 in Appendix]{Kodaira}.   
\end{rem}

\subsection{Eta invariant}\label{eta}
For every elliptic self-adjoint differential operator $A$, which acts on a Hermitian vector bundle over a closed manifold, the operator $A$ has a discrete spectrum with real eigenvalues. 
 Let $\lambda_j$ run over the eigenvalues of $A$, then the eta function of $A$ is defined as 
$$\eta_A(s)=\sum_{\lambda_j\neq 0}\frac{\op{sign}\lambda_j}{|\lambda_j|^s},$$
where $s\in \mb{C}$. The eta function admits a meromorphic continuation to the whole complex plane and is holomorphic at $s=0$. The special value $\eta_A(0)$ is then called the {\it eta invariant} of the operator $A$, and we denote the eta invariant by
\begin{align}\label{eta invariant}
\eta(A)=\eta_A(0).
\end{align}
The bundle $\mc{E}_{\mb{C}}$ is a Hermitian vector bundle over $S^1$ with the Hermitian metric \eqref{Hermitian metric}, and the operator $A_{\mbf{J}}$ is an elliptic operator which is formally self-adjoint with respect to the inner product \eqref{L2 metric}, then $A_{\mbf{J}}$ has discrete spectrum consisting of real eigenvalues $\lambda$ of finite multiplicity, and the eta invariant $\eta(A_{\mbf{J}})$ of $A_{\mbf{J}}$  is defined by \eqref{eta invariant}.
\begin{ex}
	For any representation $\rho:\pi_1(S^1)\to \op{Sp}(2,\mb{R})=\op{SL}(2,\mb{R})$, and consider the operator $$A_{\mbf{J}}=\mbf{J}\frac{d}{dx},$$
	where $\mbf{J}:=\exp(-xB)J\exp(xB)$, and  $L=\pm\exp(2\pi B)\in \op{Sp}(2,\mb{R})$ denotes the representation of the generator of $\pi_1(S^1)$, then the eta invariant $\eta(A_{\mbf{J}})$ is calculated in Appendix \ref{Appeta}, see \eqref{etadim2}.
\end{ex}

\section{Signature of flat symplectic vector bundles}\label{Sig}

In this section, we will define the signature of a flat symplectic vector bundle, and show it can be expressed as the difference of two $\mathrm{L}^2$-indices. 

Let $\Sigma$ be a connected oriented surface with smooth boundary $\p\Sigma$,  each component of $\p\Sigma$ is homeomorphic to $S^1$, let $\iota:\p\Sigma\to\Sigma$ denote the natural inclusion. Let $(E,\Omega)$ be a real symplectic vector space, and $\rho:\pi_1(\Sigma)\to\op{Sp}(E,\Omega)$ be a representation from the fundamental group $\pi_1(\Sigma)$ of $\Sigma$ into the real symplectic group $\op{Sp}(E,\Omega)$. The representation $\rho$ gives a flat vector bundle $\mc{E}=\wt{\Sigma}\times_{\rho}E$ over $\Sigma$.
Any element of $A^*(\Sigma,\mc{E})$ can be viewed as a $\rho$-equivariant element in $A^*(\wt{\Sigma},\mb{R})\otimes E$, where $\rho$-equivariant means $(\gamma^{-1})^*\omega\otimes \rho(\gamma)v=\omega\otimes v$ for $\omega\in A^*(\wt{\Sigma},\mb{R})$ and $v\in E$. There exists a {\it canonical flat connection} $d$ on the flat bundle $\mc{E}$, which is defined  by $d(\omega\otimes v):=d\omega\otimes v$. One can also refer  to \cite[Section 1.1]{BM} for the representations, flat bundles and the canonical flat connection.

\subsection{Definition of signature}
Let $\mathrm{H}^*(\Sigma,\mc{E})$ (resp. $\mathrm{H}^*(\Sigma,\p\Sigma,\mc{E})$)  denote the (resp. relative) twisted  singular cohomology, one can refer to \cite[Chapter 5]{DK} for its definitions. 
Set
\begin{align*}
	\widehat{\mathrm{H}}^1(\Sigma,\mc{E})&:=\op{Im}(\mathrm{H}^1(\Sigma,\p \Sigma,\mc{E})\rightarrow \mathrm{H}^1(\Sigma,\mc{E})).
\end{align*}
 There exists a natural quadratic form 
$$Q_{\mb{R}}:\widehat{\mathrm{H}}^1_{\text{}}(\Sigma,\mc{E})\times \widehat{\mathrm{H}}^1(\Sigma,\mc{E})\to \mb{R}$$
$$Q_{\mb{R}}([a],[b])=\int_{\Sigma}\Omega([a]\cup [b]).$$
By the same argument as in \cite[Page 65]{APS}, the form $Q_{\mb{R}}$ is non-degenerate due to Poincar\'e duality.
Moreover, $Q_{\mb{R}}$ is symmetric, i.e.  $Q_{\mb{R}}([a],[b])=Q_{\mb{R}}([b],[a])$. If $\widehat{\mathrm{H}}^1(\Sigma,\mc{E})=\mathscr{H}^+\oplus\mathscr{H}^-$ such that $Q_{\mb{R}}$ is positive definite on $\mathscr{H}^+$ and negative definite on $\mathscr{H}^-$, then the signature of  $(\mc{E},\Omega)$ is defined as the signature of the symmetric bilinear form $Q_{\mb{R}}$, that is,
$$\op{sign}(\mc{E},\Omega):=\op{sign}(Q_{\mb{R}})=\dim\mathscr{H}^+-\dim\mathscr{H}^-.$$
 By $\mb{C}$-linear extension, $\Omega$ can be viewed as a symplectic form on $\mc{E}_{\mb{C}}$. For this symplectic vector bundle $(\mc{E}_{\mb{C}},\Omega)$, we obtain a complex quadratic form 
$$Q_{\mb{C}}:\widehat{\mathrm{H}}^1(\Sigma,\mc{E}_{\mb{C}})\times \widehat{\mathrm{H}}^1(\Sigma,\mc{E}_{\mb{C}})\to \mb{C}$$
$$Q_{\mb{C}}([a],[b])=\int_{\Sigma}\Omega([a]\cup \o{[b]}),$$
where $\widehat{\mathrm{H}}^1(\Sigma,\mc{E}_{\mb{C}}):=\op{Im}(\mathrm{H}^1(\Sigma,\p \Sigma,\mc{E}_{\mb{C}})\to \mathrm{H}^1(\Sigma,\mc{E}_{\mb{C}}))$. For any $[a],[b]\in \widehat{\mathrm{H}}^1(\Sigma,\mc{E}_{\mb{C}})$, one has $Q_{\mb{C}}([a],[b])=\o{Q_{\mb{C}}([b],[a])}$, which means $Q_{\mb{C}}$ is a Hermitian form.
Naturally, one can define the signature $\op{sign}(\mc{E}_{\mb{C}},\Omega)$ of $(\mc{E}_{\mb{C}},\Omega)$. The two signatures $\op{sign}(\mc{E},\Omega)$ and $\op{sign}(\mc{E}_{\mb{C}},\Omega)$ can be proved to be equal. 
\begin{prop}\label{prop3.0}
$\op{sign}(\mc{E}_{\mb{C}},\Omega)=\op{sign}(\mc{E}_{\mb{R}},\Omega)$.
\end{prop}
\begin{proof}
	Since $\widehat{\mathrm{H}}^1(\Sigma,\mc{E})=\mathscr{H}^+\oplus\mathscr{H}^-$ and $\widehat{\mathrm{H}}^1(\Sigma,\mc{E}_{\mb{C}})=\widehat{\mathrm{H}}^1(\Sigma,\mc{E})\otimes\mb{C}$, so 
	$$\widehat{\mathrm{H}}^1(\Sigma,\mc{E}_{\mb{C}})=\mathscr{H}^+_{\mb{C}}\oplus\mathscr{H}^-_{\mb{C}},$$
	where $\mathscr{H}^{\pm}_{\mb{C}}:=\mathscr{H}^{\pm}\otimes\mb{C}$, and $\dim_{\mb{C}}\mathscr{H}^{\pm}_{\mb{C}}=\dim_{\mb{R}}\mathscr{H}^{\pm}$. Hence 
	$$\op{sign}(\mc{E}_{\mb{C}},\Omega)=\dim_{\mb{C}}\mathscr{H}^+_{\mb{C}}-\dim_{\mb{C}}\mathscr{H}^-_{\mb{C}}=\dim_{\mb{R}}\mathscr{H}^+-\dim_{\mb{R}}\mathscr{H}^-=\op{sign}(\mc{E}_{\mb{R}},\Omega).$$
	The proof is complete.
\end{proof}

\subsection{Relation to the indices of operators}
Suppose that on the collar neighborhood $I \times \p\Sigma \subset \Sigma$ of $\p\Sigma$, $I=[0,1]$, the Riemannian metric of $\Sigma$ is equal to the product metric $g_\Sigma=d u^{2}+g_{\p\Sigma}$.
 Let  
 $$\widehat{\Sigma}=\Sigma\cup((-\infty,0]\times\p\Sigma)$$ 
 be the complete manifold obtained from $\Sigma$ by gluing the negative half-cylinder $(-\infty,0]\times \p\Sigma$ to the boundary of $\Sigma$.
 \begin{center}
\begin{tikzpicture}[x=1cm,y=1cm]

\begin{scope}[shift={(2,0)}, thick]
\clip(-1.8,-2)rectangle(3,2);
\draw (0,0) circle [x radius=2, y radius=1];
\end{scope}

\draw[shift={(2,0)}, yscale=cos(70), thick] (-1,0) arc (-180:0:1);
\draw[shift={(2,0)}, yscale=cos(70), thick] (-0.7,-0.6) arc (180:0:0.7);

\draw[thick] (-4,-0.3) -- (0,-0.3);
\draw[thick] (-4,0.3) -- (0,0.3);

\draw [xscale=cos(70), dashed, thick] (0,-0.3) arc (-90:90:0.3);
\draw [xscale=cos(70), thick] (0,0.3) arc (90:270:0.3);

\draw [shift={(-0.7,0)}, xscale=cos(70), dashed, thick] (0,-0.3) arc (-90:90:0.3);
\draw [shift={(-0.7,0)}, xscale=cos(70), thick] (0,0.3) arc (90:270:0.3);

\draw [thick] (0,0.3) .. controls (0.1,0.3) and (0.1,0.34) .. (0.21,0.446);
\draw [thick] (0,-0.3) .. controls (0.1,-0.3) and (0.1,-0.34) .. (0.21,-0.446);

\draw (-4,-0.5) node{{\tiny $-\infty$}};
\draw (-0.7,-0.5) node{{\tiny $0$}};
\draw (-0.35,-0.6) node{{\tiny $I$}};
\draw (-1.1,0) node{{\tiny $\partial \Sigma$}};
\draw (0,-0.5) node{{\tiny $1$}};
\draw (2.7, 0.5) node{{\tiny $\Sigma$}};
\draw (0,-1.3) node{{$\widehat{\Sigma}$}};

\end{tikzpicture}
\end{center}

 %\begin{center}
%\includegraphics[height=5cm]{HatSigma}
%\end{center}
 For any $a,b\in A^{*}(\Sigma,\mc{E}_{\mb{C}})$, the Hodge $*$ operator is defined as 
\begin{align}\label{star}
\Omega(a\wedge *b)=	g_\Sigma(\Omega(a,b))\cdot\op{Vol}_g,
\end{align}
where the volume element  $\op{Vol}_g=\sqrt{\det(g_{ij})}dx^1\wedge  dx^2$. One can check that  $*^2a=(-1)^{|a|}a$. Let $\mbf{J}\in \mc{J}(\mc{E},\Omega)$ be a compatible complex structure  on $\mc{E}$ and also denote its complex linear extension to  $\mc{E}_{\mb{C}}$. The following pairing 
\begin{align}\label{Hermitian}(\cdot,\cdot)=2\Omega(\cdot,\mbf{J}\o{\cdot})\end{align}
is a Hermitian inner product  on $\mc{E}_{\mb{C}}$. With respect to $(\cdot,\cdot)$ and $g_\Sigma$, there exists a global $\mathrm{L}^2$-inner product $\langle\cdot,\cdot\rangle$ on the space $A^*(\Sigma,\mc{E}_{\mb{C}})$. Denote by $d^*$ the adjoint operator of $d$ with respect to $\langle\cdot,\cdot\rangle$. Then 
\begin{equation}
  d^*=\mbf{J}*d*\mbf{J}.
\end{equation}
Moreover, one has $*\mbf{J}=\mbf{J}*$, and  $(*\mbf{J})^2$ is the identity when acting on $\mc{E}_{\mb{C}}$-valued one forms, that is,  $*\mbf{J}$ is an involution on the space  $\wedge^1T^*\Sigma\otimes \mc{E}_{\mb{C}}$. Let 
$$\pi^{\pm}:=\frac{1\pm *\mbf{J}}{2}:\wedge^1T^*\Sigma\otimes \mc{E}_{\mb{C}}\to \wedge^{\pm}$$
denote the natural projections onto the $\pm 1$-eigenspaces of $*\mbf{J}$, and set
$$d^{\pm}=\pi^{\pm}\circ d.$$
From Corollary \ref{cor4.2}, the operators $d^{\pm}$ have  the  form 
$$d^{\pm}=\sigma^{\pm}\left(\frac{\p}{\p u}+A^{\pm}_{\mbf{J}}\right),$$
where $\sigma^{\pm}:\mc{E}_{\mb{C}}\to\wedge^\pm$ are the bundle isomorphisms, and $A^{\pm}_{\mbf{J}}$ are the first order  elliptic formally self-adjoint operators on boundary. 

For any $a\in \wedge^+$ and $b\in\wedge^-$, one has
\begin{align*}
\Omega(a\wedge *\mbf{J}\o{b})=-\Omega(a\wedge\o{b})=-\Omega(\mbf{J}*a\wedge\o{b})=\Omega(*a\wedge\mbf{J}\o{b})=-\Omega(a\wedge *\mbf{J}\o{b}),	
\end{align*}
which follows that $\Omega(a\wedge *\mbf{J}\o{b})=0$, thus the Hermitian inner product of $a$ and $b$ satisfies 
$$(a,b)\cdot\op{Vol}_g=2g_\Sigma(\Omega(a,\mbf{J}\b{b}))\cdot \op{Vol}_g=2\Omega(a\wedge *\mbf{J}\b{b})=0,$$
that is, the two subspaces $\wedge^+$ and $\wedge^-$ are orthogonal to each other. Hence $(d^{\pm})^*=d^*$ on $\wedge^{\pm}$.

Note that  \cite[Proposition 4.9]{APS}  works as well for the cohomology with local coefficients, that is
\begin{align}\label{L2harmonic}
	\widehat{\mathrm{H}}^*(\Sigma,\mc{E}_{\mb{C}})\cong \mathscr{H}^*(\widehat{\Sigma},\mc{E}_{\mb{C}}):=\{\phi\in \mathrm{L}^2(\widehat{\Sigma},\mc{E}_{\mb{C}}): (d+d^*)\phi=0\}.
\end{align}
By identified with the above two cohomology groups, then $Q_{\mb{C}}$ is also a Hermitian quadratic form on $\mathscr{H}^1(\widehat{\Sigma},\mc{E}_{\mb{C}})$. We can extend the bundle $\mc{E}_{\mb{C}}$ and the connection $d$ to $\widehat{\Sigma}$, we still denote them by $\mc{E}_{\mb{C}}$ and $d$, respectively.
\begin{prop}\label{prop3}
It holds
$$\mathscr{H}^1(\widehat{\Sigma},\mc{E}_{\mb{C}})=\op{Ker}(d^+)^*\cap \mathrm{L}^2(\widehat{\Sigma},\wedge^+)\oplus \op{Ker}(d^-)^*\cap \mathrm{L}^2(\widehat{\Sigma},\wedge^-)$$ such that $Q_{\mb{C}}$ is positive definite on $\op{Ker}(d^+)^*\cap \mathrm{L}^2(\widehat{\Sigma},\wedge^+)$ and negative definite on $\op{Ker}(d^-)^*\cap \mathrm{L}^2(\widehat{\Sigma},\wedge^-)$, the decomposition is orthogonal with respect to $Q_{\mb{C}}$.
\end{prop}
\begin{proof}
	Firstly, we show $\op{Ker}(d^\pm)^*\cap \mathrm{L}^2(\widehat{\Sigma},\wedge^\pm)\subset \mathscr{H}^1(\widehat{\Sigma},\mc{E}_{\mb{C}})$. If $a\in \op{Ker}(d^\pm)^*\cap \mathrm{L}^2(\widehat{\Sigma},\wedge^\pm)$, then $d^* a=0$ and $a\in \mathrm{L}^2(\widehat{\Sigma},\wedge^{\pm})$. From \cite[Proposition 3.11]{APS}, the $\mathrm{L}^2$ sections in $\op{Ker}(d^{\pm})^*\cap \mathrm{L}^2(\widehat{\Sigma},\wedge^{\pm})$ are exponentially decaying as $t\to-\infty$, so 
	$$d^*a=*\mbf{J}d*\mbf{J}a=\pm*\mbf{J}da,$$
and $d^*a=0$ implies that $d a=0$, which means $a$ is harmonic. Hence $\op{Ker}(d^\pm)^*\cap \mathrm{L}^2(\widehat{\Sigma},\wedge^\pm)\subset \mathscr{H}^1(\widehat{\Sigma},\mc{E}_{\mb{C}})$. On the other hand, for any $b\in \op{Ker}(d^{\pm})^*\cap \mathrm{L}^2(\widehat{\Sigma},\wedge^\pm)$, one has
	\begin{align*}
	Q_{\mb{C}}(a,b)=\int_{\widehat{\Sigma}}\Omega(a\wedge \b{b})=\pm\int_{\widehat{\Sigma}}\Omega(a\wedge *\mbf{J}\b{b})=0,	
	\end{align*}
	since $\wedge^+$ and $\wedge^-$ are orthogonal each other. The proof is complete.
\end{proof}
The $\mathrm{L}^2$-index of $d^{\pm}$ is well-defined and is given by 
\begin{align}\label{L2index}
\mathrm{L}^2\op{Index}(d^{\pm}):=\dim \op{Ker}(d^{\pm})\cap \mathrm{L}^2(\widehat{\Sigma},\mc{E}_{\mb{C}})-\dim\op{Ker}(d^\pm)^*\cap \mathrm{L}^2(\widehat{\Sigma},\wedge^\pm).
\end{align}
Note that 
\begin{align}\label{H0}
\op{Ker}(d^{\pm})\cap \mathrm{L}^2(\widehat{\Sigma},\mc{E}_{\mb{C}})=\mathscr{H}^0(\widehat{\Sigma},\mc{E}_{\mb{C}}).
\end{align}
 In fact, if $d^-a=0$, then $d a=d^+a$, and so 
$d a=*\mbf{J}d a$, which follows that $(d^+)^*d^+a=d^*d a=0$. From \cite[Proposition 3.15]{APS}, the $\mathrm{L}^2$-solutions of $d^+$ and $(d^+)^*d^+$ are coincide, so $d a=d^+a=0$. Therefore,
\begin{align*}
\mathrm{L}^2\op{Index}(d^{\pm}):=\dim\mathscr{H}^0(\widehat{\Sigma},\mc{E}_{\mb{C}})-\dim\op{Ker}(d^\pm)^*\cap \mathrm{L}^2(\widehat{\Sigma},\wedge^\pm).
\end{align*}
By \eqref{dP}, one can define the operators $d^{\pm}_P$ by the restriction of $d^{\pm}$.
From \cite[Proposition 3.11]{APS}, $\op{Ker}d_{P}^{\pm}$ is isomorphic to the space of $\mathrm{L}^2$-solutions of $d^{\pm} \varphi=0$ on $\widehat{\Sigma}$, and $\operatorname{Ker} (d^\pm_P)^{*}$ is isomorphic to the space of extended $\mathrm{L}^2-$solutions of $(d^{\pm})^* \varphi=0$ on $\widehat{\Sigma},$ where extended solution means that on $(-\infty, 0] \times \p\Sigma$, $\varphi$ can be written as
$
\varphi=\phi+\psi
$
with $\phi \in \operatorname{Ker} \sigma^{\pm}A_{\mbf{J}}^{\pm}(\sigma^\pm)^{-1}$ and $\psi \in \mathrm{L}^2$, see Section \ref{TAPST}. The section $\phi$ is called the limiting value of the extended solution $\v$.

We denote by $h_{\infty}(\wedge^{\pm})$  the dimension of the subspace of $\op{Ker}(A^{\pm}_{\mbf{J}})$ consisting of limiting values of extended $\mathrm{L}^2$-sections $a$ of $\wedge^{\pm}$ satisfying $(d^{\pm})^*a=0$.
 Therefore,
\begin{prop}\label{prop3.2}
	The signature of $(\mc{E},\Omega)$ can be given by
	\begin{align*}
		\op{sign}(\mc{E},\Omega)&=\mathrm{L}^2\op{Index}(d^{-})-\mathrm{L}^2\op{Index}(d^{+})\\
		&=\op{Index}(d^-_P)-\op{Index}(d^+_P)+h_{\infty}(\wedge^-)-h_{\infty}(\wedge^+).
	\end{align*}
\end{prop}
 \begin{proof}
 By the definition of signature, Proposition \ref{prop3} and \cite[Corollary 3.14]{APS}, one has
 \begin{align*}
 	\op{sign}(\mc{E}_{\mb{C}},\Omega)&=\dim\op{Ker}(d^+)^*\cap \mathrm{L}^2(\widehat{\Sigma},\wedge^+)-\dim\op{Ker}(d^-)^*\cap \mathrm{L}^2(\widehat{\Sigma},\wedge^-)\\
 	&=\mathrm{L}^2\op{Index}(d^-)-\mathrm{L}^2\op{Index}(d^+)\\
 	&=\op{Index}(d^-_P)-\op{Index}(d^+_P)+h_{\infty}(\wedge^-)-h_{\infty}(\wedge^+).
 \end{align*}
 By Proposition \ref{prop3.0}, the proof is complete.
 \end{proof}

\subsection{A formula of signature}

In this subsection, by using Atiyah-Patodi-Singer index theorem, we will give a formula of signature.

\subsubsection{The Atiyah-Patodi-Singer index theorem}\label{TAPST}

Let $\Sigma$ be a connected oriented surface with smooth boundary $\p\Sigma$. Let $g_\Sigma$ be a Riemannian metric on $\Sigma$ such that $g_\Sigma=du^2+g_{\p\Sigma}$ on the collar neighborhood $\p\Sigma\times I\subset\Sigma$ of $\p\Sigma$, $I=[0,1]$. The bundles $\mc{E}_{\mb{C}}$ and $\wedge^{\pm}$ are Hermitian vector bundles, and 
$$d^{\pm}:A^0(\Sigma,\mc{E}_{\mb{C}})\to A^0(\Sigma,\wedge^\pm)$$
are two first order elliptic differential operators. On $\p\Sigma\times I$, the volume element is given by
$$\op{Vol}_g=\frac{dx\wedge du}{|dx|}.$$
From the definition of $*$ \eqref{star}, one has 
$$*(dx\otimes e)=|dx|du\otimes e,\quad*(du\otimes e)=-\frac{1}{|dx|}dx\otimes e$$
for any $e\in \mc{E}_{\mb{C}}$. Thus 
\begin{align*}
*\left(\frac{dx}{|dx|}+idu\right)=-i\left(\frac{dx}{|dx|}+idu\right),\quad 	*\left(\frac{dx}{|dx|}-idu\right)=i\left(\frac{dx}{|dx|}-idu\right).
\end{align*}
Let $\mc{E}_{\mb{C}}=\mc{E}_{\mb{C}}^+\oplus \mc{E}_{\mb{C}}^-$ be the decomposition of $
\mc{E}_{\mb{C}}$ into $\pm i$-eigenspaces of $\mbf{J}$. Then 
\begin{align*}
\wedge^+=\mb{C}\left\{\frac{dx}{|dx|}-idu\right\}\otimes\mc{E}^-_{\mb{C}}\oplus\mb{C}\left\{\frac{dx}{|dx|}+idu\right\}\otimes\mc{E}_{\mb{C}}^+	
\end{align*}
and 
\begin{align*}
\wedge^-=\mb{C}\left\{\frac{dx}{|dx|}-idu\right\}\otimes\mc{E}^+_{\mb{C}}\oplus\mb{C}\left\{\frac{dx}{|dx|}+idu\right\}\otimes\mc{E}_{\mb{C}}^-.	
\end{align*}
Note that $\wedge^{\pm}\cong\mc{E}_{\mb{C}}$ and the isomorphisms are given by 
$$\sigma^+:\mc{E}_{\mb{C}}\to\wedge^+, \sigma^+(e)=-\frac{i}{2}\left(\frac{dx}{|dx|}+idu\right)\otimes e^++\frac{i}{2}\left(\frac{dx}{|dx|}-idu\right)\otimes e^-$$
and 
$$\sigma^-:\mc{E}_{\mb{C}}\to\wedge^-, \sigma^-(e)=\frac{i}{2}\left(\frac{dx}{|dx|}-idu\right)\otimes e^+-\frac{i}{2}\left(\frac{dx}{|dx|}+idu\right)\otimes e^-.$$
The maps $\pi^{\pm}:\wedge^1T^*\Sigma\otimes\mc{E}_{\mb{C}}\to\wedge^\pm$ can be expressed as
$$\pi^+\left(\frac{dx}{|dx|}\otimes e\right)=\frac{1}{2}\left(\frac{dx}{|dx|}+idu\right)\otimes e^++\frac{1}{2}\left(\frac{dx}{|dx|}-idu\right)\otimes e^-=\sigma^+(i(e^+-e^-)),$$
$$\pi^-\left(\frac{dx}{|dx|}\otimes e\right)=\frac{1}{2}\left(\frac{dx}{|dx|}-idu\right)\otimes e^++\frac{1}{2}\left(\frac{dx}{|dx|}+idu\right)\otimes e^-=\sigma^-(-i(e^+-e^-)),$$
$$\pi^+(du\otimes e)=-\frac{i}{2}\left(\frac{dx}{|dx|}+idu\right)\otimes e^++\frac{i}{2}\left(\frac{dx}{|dx|}-idu\right)\otimes e^-=\sigma^+(e),$$
$$\pi^-(du\otimes e)=\frac{i}{2}\left(\frac{dx}{|dx|}-idu\right)\otimes e^+-\frac{i}{2}\left(\frac{dx}{|dx|}+idu\right)\otimes e^-=\sigma^-(e).$$
\begin{prop}\label{prop4.1}
For any $C\in A^0(\p\Sigma\times I,\op{End}(\mc{E}_{\mb{C}}))$, one has
$$\pi^\pm(d+Cdx)=\sigma^{\pm}\left(\frac{\p}{\p u}\pm |dx|\mbf{J}\left(\frac{\p}{\p x}+C\right)\right).$$	
\end{prop}
\begin{proof}
For any local smooth section $e$ of $\mc{E}_{\mb{C}}$, one has
	\begin{align*}
&\quad\sigma^{\pm}\left(\frac{\p}{\p u}\pm |dx|\mbf{J}\left(\frac{\p}{\p x}+C\right)\right)e\\
	&=\sigma^{\pm}\left(\frac{\p e}{\p u}\pm i|dx|\left(\left(\frac{\p e}{\p x}\right)^+-\left(\frac{\p e}{\p x}\right)^-\right)\pm i|dx|\left((Ce)^+-(Ce)^-\right)\right)\\
	&=\pi^{\pm}\left(\frac{\p e}{\p u}du+dx\otimes\frac{\p e}{\p x}+dx\otimes Ce\right)\\
	&=\pi^{\pm}(d+Cdx)e,
	\end{align*}
which completes the proof.
\end{proof}
As a corollary, we get
\begin{cor}\label{cor4.2}
$d^{\pm}=\sigma^{\pm}(\frac{\p}{\p u}+A^{\pm}_{\mbf{J}})$, where $A^{\pm}_{\mbf{J}}=\pm |dx|\mbf{J}\frac{\p}{\p x}$.	
\end{cor}
When restricted on $\p\Sigma$, the metric is $g_{\p\Sigma}=g(x)dx\otimes dx$, and $|dx|=\frac{1}{\sqrt{g(x)}}$. By taking an another parameter 
$x'=\int_0^x\sqrt{g(\ell)}d\ell,$
then $dx'=\sqrt{g(x)}dx$, and so the metric is given by $g_{\p\Sigma}=dx'\otimes dx'$, the operator $A_{\mbf{J}}$ is 
$$A_{\mbf{J}}=\mbf{J}|dx|\frac{d}{d x}= \mbf{J}\frac{d}{d x'},$$
which is exactly the operator defined by \eqref{ABJ}.

 Let $P_{\pm}$ denote the orthogonal projections of $\mathrm{L}^2(\p\Sigma,\mc{E}_{\mb{C}})$ onto the subspace spanned by all eigenfunctions of $A^{\pm}_{\mbf{J}}$ with eigenvalues $\lambda>0$. Then $P_\pm$ are  pseudo-differential operators. Let $A^0(\Sigma,\mc{E}_{\mb{C}};P_\pm)$ be the subspaces of $A^0(\Sigma,\mc{E}_{\mb{C}})$ consisting of all sections $\v$ which satisfy the boundary conditions
$$P_\pm(\v|_{\p\Sigma})=0.$$
Denote by 
\begin{align}\label{dP}
	d^{\pm}_{P}:A^0(\Sigma,\mc{E}_{\mb{C}};P_\pm)\to A^0(\Sigma,\wedge^{\pm})
\end{align}
the restriction of $d^{\pm}$. By Atiyah-Patodi-Singer index theorem \cite[Theorem 3.10]{APS}, $d^{\pm}_P$ are Fredholm operators and 
\begin{align}\label{APSformula}
\op{Index}(d^{\pm}_P)=\int_{\Sigma}\alpha_{{\pm}}(z)d\mu_g-\frac{\eta(A_{\mbf{J}}^{\pm})+\dim\op{Ker}A^{\pm}_{\mbf{J}}}{2},	
\end{align}
where $d\mu_g$ denotes the volume form of the Riemannian metric $g$ on $\Sigma$,
and $\alpha_{\pm}(z)$ is the constant term in the asymptotic expansion  (as $ t\to 0$) of
$$
\sum e^{-t \mu^{\prime}_\pm}\left|\phi_{\mu_\pm}^{\prime}(x)\right|^{2}-\sum e^{-t \mu_\pm^{\prime \prime}}\left|\phi_{\mu_\pm}^{\prime \prime}(x)\right|^{2},
$$
where $\mu^{\prime}_\pm, \phi_{\mu_\pm}^{\prime}$ denote the eigenvalues and eigenfunctions of $(d^\pm)^{*} d^\pm$ on the double $\Sigma\cup_{\p\Sigma}\Sigma$ of $\Sigma,$ and $\mu_\pm^{\prime \prime},$ $\phi_{\mu_\pm}^{\prime \prime}$ are the corresponding objects for $d^\pm (d^\pm)^{*}$.

 Since $\eta(A_{\mbf{J}}^{\pm})=\eta(\pm A_{\mbf{J}})=\pm\eta(A_{\mbf{J}})$ and
 \begin{align}\label{KerAJ}
 	\op{Ker}A^{\pm}_{\mbf{J}}=\op{Ker}|dx|\mbf{J}\frac{\p}{\p x}={\op{Ker}|dx| \frac{\p}{\p x}}=\op{Ker}d|_{A^0(\p\Sigma,\mc{E}_{\mb{C}})}=\mathrm{H}^0(\p\Sigma,\mc{E}_{\mb{C}}),
 \end{align}
so we obtain
\begin{align*}
\op{Index}(d^-_{P})-\op{Index}(d^+_P)=\int_{\Sigma}\alpha_-(z)d\mu_g-\int_{\Sigma}\alpha_+(z)d\mu_g+\eta(A_{\mbf{J}}).
\end{align*}
By Proposition \ref{prop3.2}, the signature of the flat symplectic vector bundle $(\mc{E},\Omega)$ is 
\begin{align}\label{sign}
\begin{split}
\op{sign}(\mc{E},\Omega)&=\int_{\Sigma}\alpha_-(z)d\mu_g-\int_{\Sigma}\alpha_+(z)d\mu_g\\
&\quad+h_{\infty}(\wedge^-)-h_{\infty}(\wedge^+)+\eta(A_{\mbf{J}}).
\end{split}
\end{align}

\subsubsection{The Atiyah-Singer integrands}

In this subsection, we will deal with the Atiyah-Singer integrands $\int_{\Sigma}\alpha_-(z)d\mu_g$ and $\int_{\Sigma}\alpha_+(z)d\mu_g$.

Let
  $$g_\Sigma=g(x,y)(dx^2+dy^2)=\frac{g(z)}{2}(dz\otimes d\b{z}+d\b{z}\otimes dz)$$ be a Riemannian metric on the surface $\Sigma$, and is a  product metric on the collar neighborhood $\p\Sigma\times I$ of $\p\Sigma$, 
  where $z=x+iy$ denotes the holomorphic coordinate of $\Sigma$, and one has
$*dz=-idz$ and $*d\b{z}=id\b{z}$. 

 Following \cite{APS}, we will consider the double $\Sigma\cup_{\p\Sigma}\Sigma$ of $\Sigma$, which is a closed surface.
\begin{center}
\begin{tikzpicture}[x=1cm,y=1cm]

\begin{scope}[shift={(2,0)}, thick]
\clip(-1.8,-2)rectangle(3,2);
\draw (0,0) circle [x radius=2, y radius=1];
\end{scope}

\draw[shift={(2,0)}, yscale=cos(70), thick] (-1,0) arc (-180:0:1);
\draw[shift={(2,0)}, yscale=cos(70), thick] (-0.7,-0.6) arc (180:0:0.7);

\draw [xscale=cos(70), dashed, thick] (0,-0.3) arc (-90:90:0.3);
\draw [xscale=cos(70), thick] (0,0.3) arc (90:270:0.3);

\draw [thick] (0,0.3) .. controls (0.1,0.3) and (0.1,0.34) .. (0.21,0.446);
\draw [thick] (0,-0.3) .. controls (0.1,-0.3) and (0.1,-0.34) .. (0.21,-0.446);

%%%%
\begin{scope}[shift={(-4,0)}, thick]
\clip(-3,-2)rectangle(1.8,2);
\draw (0,0) circle [x radius=2, y radius=1];
\end{scope}

\draw[shift={(-4,0)}, yscale=cos(70), thick] (-1,0) arc (-180:0:1);
\draw[shift={(-4,0)}, yscale=cos(70), thick] (-0.7,-0.6) arc (180:0:0.7);

\draw [shift={(-2,0)}, xscale=cos(70), dashed, thick] (0,-0.3) arc (-90:90:0.3);
\draw [shift={(-2,0)}, xscale=cos(70), thick] (0,0.3) arc (90:270:0.3);

\draw [thick,shift={(-2,0)}] (0,0.3) .. controls (-0.1,0.3) and (-0.1,0.34) .. (-0.21,0.446);
\draw [thick,shift={(-2,0)}] (0,-0.3) .. controls (-0.1,-0.3) and (-0.1,-0.34) .. (-0.21,-0.446);

\draw[thick] (-2,-0.3) -- (0,-0.3);
\draw[thick] (-2,0.3) -- (0,0.3);

\draw [shift={(-1,0)}, xscale=cos(70), dashed, thick] (0,-0.3) arc (-90:90:0.3);
\draw [shift={(-1,0)}, xscale=cos(70), thick] (0,0.3) arc (90:270:0.3);

%\draw (-4,-0.5) node{{\tiny $-\infty$}};
\draw (-1,-0.5) node{{\tiny $0$}};
\draw (-1.4,0) node{{\tiny $\partial \Sigma$}};
\draw (0,-0.5) node{{\tiny $1$}};
\draw (-2,-0.5) node{{\tiny $1$}};%%%%
\draw (2.7, 0.5) node{{\tiny $\Sigma$}};
\draw (-4.7, 0.5) node{{\tiny $\Sigma$}};
\draw (-1,-1.3) node{{$\Sigma\cup_{\partial\Sigma}\Sigma$}};

\end{tikzpicture}
\end{center}
The vector bundle $\mc{E}$ and the  operators $d^{\pm}$ can be extended canonically on the double $\Sigma\cup_{\p\Sigma}\Sigma$. 
Let $\mc{F}=\mc{F}^+\oplus\mc{F^-}$ be a $\mb{Z}_2$-graded vector bundle over the double $\Sigma\cup_{\p\Sigma}\Sigma$, where 
$$\mc{F}^+:=\mc{E}_{\mb{C}},\quad \mc{F}^-:=\wedge^-.$$
Let $D:\Gamma(\Sigma\cup_{\p\Sigma}\Sigma,\mc{F})\to \Gamma(\Sigma\cup_{\p\Sigma}\Sigma,\mc{F})$ be an  operator defined as follows:
\begin{align}\label{Dirac operator}
D=\left(\begin{array}{cc}
0 & D^-=(d^-)^*	\\
D^+=d^- &0
\end{array}
\right):\Gamma(\Sigma\cup_{\p\Sigma}\Sigma,\mc{F}^{\pm})\to \Gamma(\Sigma\cup_{\p\Sigma}\Sigma,\mc{F}^{\mp}).
\end{align}
 \begin{prop}
 $D$ is a self-adjoint Dirac operator.
  \end{prop}
\begin{proof}
Since $D^-=(D^+)^*$, so $D$ is self-adjoint. On the other hand, one has
\begin{align*}
D^2=\left(\begin{array}{cc}
	(d^-)^*d^- &0\\
	0&d^-(d^-)^*
\end{array}
\right)	
\end{align*}
so $D^2$ is a generalized Laplacian. In fact, for any local section $s=fe$ of $\mc{E}_{\mb{C}}$, where $e$ is a local parallel section, i.e.  $d e=0$, {we have}
\begin{align*}
(d^-)^*d^-s &=(d^-)^*\frac{1-*\mbf{J}}{2}(\p f+\b{\p}f)e\\
&=\frac{1}{2}(d^-)^*\left((\p f+\b{\p}f)e-(-i\p f+i{\b\p} f)\mbf{J}e\right),
\end{align*}
 since $(d^-)^*=-2g(z)^{-1}(\p_{z}i_{\frac{\p}{\p\b{z}}}+\p_{\b{z}}i_{\frac{\p}{\p z}})+\text{zeroth terms}$, so the second order terms of $(d^-)^*d^-s$ is 
 $-2g(z)^{-1}\frac{\p^2f}{\p z\p{\b{z}}}e$. Similar for the local sections of $\wedge^-$. Thus $D$ is a Dirac operator.
\end{proof}
By the definition of $\alpha_\pm(z)$, one has
\begin{align*}
\int_{\Sigma}\alpha_-(z)d\mu_g &=\lim_{t\to 0}\int_{\Sigma}\left(\operatorname{tr}\left(e^{-t (d^{-})^{*} d^{-}}(z, z)\right)-\operatorname{tr}\left(e^{-t d^{-}  (d^{-})^*}(z, z)\right)	\right)\\
&=\lim_{t\to 0}\int_{\Sigma}\op{Str}\langle z|e^{-tD^2}|z\rangle d\mu_g,
\end{align*}
where $d\mu_g=\frac{i}{2}g(z)dz\wedge d\b{z}$ and $\op{Str}$ denotes the supertrace, see e.g. \cite[Section 1.5]{BGV} for its definition.

By complex linear extension, $\mbf{J}$ is an endomorphism of $\mc{E}_{\mb{C}}$ and $\mbf{J}^2=-\op{Id}$. Denote by $\mc{E}_{\mb{C}}=\mc{E}_{\mb{C}}^+\oplus\mc{E}^-_{\mb{C}}$ the decomposition corresponding the $\pm i$-eigenspaces. Then $\mc{E}_{\mb{C}}$  is a complex vector bundle with a Hermitian metric $H(e_1,e_2)=2\Omega(e_1,\mbf{J}e_2)$.
\begin{defn}\label{connection}
	For a connection $\n$ on $\mc{E}$, we call it is a {\it peripheral connection} if the induced (complex linear extension) connection on $\mc{E}_{\mb{C}}$ satisfies the following conditions on a collar neighborhood of $\p\Sigma$:
	\begin{itemize}
\item[(i)] 	$\n=d+C(x)dx$ for some $C=C(x)\in A^0(\p\Sigma,\op{End}(\mc{E}_{\mb{C}}))$;
\item[(ii)] $[\n,\mbf{J}]=0$;
\item[(iii)] $\n$ preserves the symplectic form $\Omega$.
\end{itemize}
\end{defn}
\begin{rem}
Similar to the proof of Proposition \ref{prop2.0}, {the condition (iii) is equivalent to  $C(x)\in \mf{s}\mf{p}(E,\Omega)$, i.e., $C(x)^\top\Omega+\Omega C(x)=0$}, which implies that   
 the operator $|dx|\mbf{J}(\frac{\p}{\p x}+C(x))$ is a $\mb{C}$-linear formally self-adjoint elliptic first order differential operator  in the space $A^0(\p\Sigma,\mc{E}_{\mb{C}})$.	
\end{rem}
\begin{rem}
	For a peripheral connection $\n$, then it preserves the Hermitian metric $H$. In fact, for any local smooth sections $e_1,e_2$ of $\mc{E}_{\mb{C}}$, one has
	\begin{align*}
\begin{split}
  d H(e_1,e_2)&=2d\Omega (e_1,\mbf{J}e_2)\\
  &=2\Omega(\n e_1, \mbf{J}e_2)+2\Omega(e_1,\n\mbf{J}e_2)\\
  &=2\Omega(\n e_1, \mbf{J}e_2)+2\Omega(e_1,\mbf{J}\n e_2)\\
  &=H(\n e_1,e_2)+H(e_1,\n e_2).
 \end{split}
\end{align*}

\end{rem}
\begin{rem}\label{remexample}
If $L=\pm\exp(2\pi B)$ and $\mbf{J}=\exp(-xB)J\exp(xB)$ on $\p\Sigma\times [0,1]$, where $B\in \mf{s}\mf{p}(E,\Omega)$, {then 
  the flat connection $\n=d+Bdx$ is a peripheral connection on $\mc{E}|_{\p\Sigma\times[0,1/2]}$}. Hence any extension of $\n=d+Bdx$ is a peripheral connection on $\mc{E}$. 

\end{rem}

Let $\n$ be any peripheral connection on $\mc{E}$, we assume it has the form $\n=d+C(x)dx$ when restricted to a small collar neighborhood of $\p\Sigma$. { The peripheral connection is always existing, e.g. 
see Proposition \ref{proppullbackconnection} via a complex linear symplectic isomorphism between $(\mbf{J}^*F_\rho,\mbf{J}^*\mbf{J}_{F_\rho},\mbf{J}^*\Omega)$ and $(\mc{E},\mbf{J},\Omega)$}.  Without loss of generality, we assume that it satisfies the above conditions $(i)-(iii)$ on $\p\Sigma\times [0,1]$. 
\begin{lemma}
There exists a complex connection on $\mc{E}_{\mb{C}}^+$ and equals $\n|_{\mc{E}^+_{\mb{C}}}$ on $\p\Sigma\times [0,1/2]$.	
\end{lemma}
\begin{proof}
Let $\n'$ be any complex linear connection on $\mc{E}^+_{\mb{C}}$. Since $[\n,\mbf{J}]=0$ on $\p\Sigma\times [0,1]$, so $\n|_{\mc{E}^+_{\mb{C}}}$ is a well-defined complex linear connection on $\mc{E}^+_{\mb{C}}$. Define 
$$\alpha:=\n'-\n|_{\mc{E}^+_{\mb{C}}}\in A^{1}(\p\Sigma\times [0,1],\op{End}(\mc{E}^+_{\mb{C}}))$$
Let $\rho_1(z)\in [0,1]$ be a smooth function which  equals $0$ on $\p\Sigma\times [0,1/2]$ and equals $1$ on $\Sigma\backslash(\p\Sigma\times [0,3/4])$. We define 
\begin{align*}
\n^{\rho_1}:=	\n|_{\mc{E}^+_{\mb{C}}}+\rho_1(z)\alpha=\n'+(\rho_1(z)-1)\alpha,
\end{align*}
which is a complex connection since $(\rho_1(z)-1)\alpha\in A^1(\Sigma,\op{End}(\mc{E}^+_{\mb{C}}))$,
and  $\n^{\rho_1}=\n'$  on 	$\Sigma\backslash(\p\Sigma\times [0,3/4])$;  $\n^{\rho_1}=\n|_{\mc{E}^+_{\mb{C}}}$ on  $\p\Sigma\times [0,1/2]$. The proof is complete.
\end{proof}
By the above lemma,  we can construct a complex linear connection $\n^{\rho_1}$ on $\mc{E}^+_{\mb{C}}$ and equals $\n|_{\mc{E}^+_{\mb{C}}}$ on $\p\Sigma\times [0,1/2]$. On the other hand, there exists a unique holomorphic structure on $\mc{E}^+_{\mb{C}}$ such that 
$\b{\p}=(\n^{\rho_1})^{0,1}$, see e.g. \cite[Proposition 1.3.7]{Kobayashi}. In this case, 
$$\n^{\rho_1}=(\n^{\rho_1})^{1,0}+(\n^{\rho_1})^{0,1}=(\n^{\rho_1})^{1,0}+\b{\p}$$
is a $(1,0)$-type connection. Now $(\mc{E}^+_{\mb{C}},h=H|_{\mc{E}^+_{\mb{C}}})$ is a holomorphic Hermitian vector bundle, then there exists a unique  Chern connection,  denote it by $\n^{\mc{E}^+_{\mb{C}}}$, which is of $(1,0)$-type and preserves the Hermitian metric. Note that $\n$ also preserves the Hermitian metric $h$ and is of $(1,0)$-type on $\p\Sigma\times [0,1/2]$. By  the uniqueness of Chern connection, $\n|_{\mc{E}^+_{\mb{C}}}=\n^{\mc{E}^+_{\mb{C}}}$ on $\p\Sigma\times [0,1/2]$. There is a natural  connection
$$\n^{\mc{E}_{\mb{C}}}=\n^{\mc{E}^+_{\mb{C}}}\oplus\o{\n^{\mc{E}^+_{\mb{C}}}}$$
on the complex vector bundle
$\mc{E}_{\mb{C}}=\mc{E}^+_{\mb{C}}\oplus \mc{E}^-_{\mb{C}}=\mc{E}^+_{\mb{C}}\oplus \o{\mc{E}^+_{\mb{C}}}$. Similarly, we consider the operator
$$D^{\mc{F}}=\pi^-\n^{\mc{E}_{\mb{C}}}+(\pi^-\n^{\mc{E}_{\mb{C}}})^*$$
on the superbundle $\mc{F}=\mc{E}_{\mb{C}}\oplus \wedge^-$. One can check that  $D^{\mc{F}}$ is also a self-adjoint Dirac operator. Moreover, by Proposition \ref{prop4.1}, one has
$$\pi^-\n^{\mc{E}_{\mb{C}}}=\pi^-\n=\sigma^-\left(\frac{\p}{\p u}-\mbf{J}|dx|\left(\frac{\p}{\p x}+C(x)\right)\right)$$
 on $\p\Sigma\times[0,1]$, and the operator $-\mbf{J}|dx|(\frac{\p}{\p x}+C(x))$ is a formally self-adjoint elliptic first order differential operator.
\begin{lemma}\label{lemma4.7}
It holds
	$$\lim_{t\to 0}\int_{\Sigma}\op{Str}\langle z|e^{-tD^2}|z\rangle d\mu_g=\lim_{t\to 0}\int_{\Sigma}\op{Str}\langle z|e^{-t(D^{\mc{F}})^2}|z\rangle d\mu_g.$$
\end{lemma}
\begin{proof}
	Denote  $\n_s=(1-s)d+s\n^{\mc{E}_{\mb{C}}}$, $s\in [0,1]$. Then 
	$$\pi^-\n_s=(1-s)\pi^-d+s\pi^-\n^{\mc{E}_{\mb{C}}}=\sigma^-\left(\frac{\p}{\p u}-\mbf{J}|dx|\frac{\p}{\p x}-s\mbf{J}|dx|C(x)\right),$$ 
	with the first order operator $-\mbf{J}|dx|\frac{\p}{\p x}-s\mbf{J}|dx|C(x)$ is formally self-adjoint and elliptic.
	The adjoint operator of $\pi^-\n_s$ is given by 
	$(\pi^-\n_s)^*=(1-s)(\pi^-d)^*+s(\pi^-\n^{\mc{E}_{\mb{C}}})^*$, so
	$$D_s=(1-s)D+sD^{\mc{F}}.$$
From   \cite[Corollary 2.50]{BGV}, one has 
	\begin{align*}
		\lim_{t\to 0}\frac{\p}{\p s}\int_{\Sigma}\op{Str}\langle z|e^{-tD^2_s}|z\rangle d\mu_g=\lim_{t\to 0}-t\int_{\Sigma}\op{Str}\left\langle z|\frac{\p D_s^2}{\p s}e^{-tD_s^2}|z\right\rangle d\mu_g=0.
	\end{align*}
Hence
$$\lim_{t\to 0}\int_{\Sigma}\op{Str}\langle z|e^{-tD^2}|z\rangle d\mu_g=\lim_{t\to 0}\int_{\Sigma}\op{Str}\langle z|e^{-t(D^{\mc{F}})^2}|z\rangle d\mu_g.$$
\end{proof}
Now we will calculate the term  $\lim_{t\to 0}\int_{\Sigma}\op{Str}\langle z|e^{-t(D^{\mc{F}})^2}|z\rangle d\mu_g$.  Firstly, we need {to} find the Clifford connection $\n^{\mc{F}}$ on $\mc{F}$ such that $D^{\mc{F}}=c\circ\n^{\mc{F}}$. The adjoint operator $(\n^{\mc{E}_{\mb{C}}})^*$ is given by
$$(\n^{\mc{E}_{\mb{C}}})^{*}=-2g(z)^{-1}i_{\frac{\p}{\p z}}\n^{\mc{E}_{\mb{C}}}_{\frac{\p}{\p \b{z}}}-2g(z)^{-1}i_{\frac{\p}{\p \b{z}}}\n^{\mc{E}_{\mb{C}}}_{\frac{\p}{\p {z}}}.$$
The Dirac operator $D^{\mc{F}}$ induces a Clifford action of $T^*\Sigma$ on $\mc{F}$ by
$$[D^{\mc{F}},f]=c(df)$$
for any smooth function $f$. Note that $\wedge^-=\wedge^{1,0}T^*\Sigma\otimes \mc{E}^{-}_{\mb{C}}\oplus\wedge^{0,1}T^*\Sigma\otimes \mc{E}^+_{\mb{C}}$, so
$$c(df)=-2g(z)^{-1}\frac{\p f}{\p\b{z}}i_{\frac{\p}{\p z}}-2g(z)^{-1}\frac{\p f}{\p z}i_{\frac{\p}{\p\b{z}}}$$
when acting on $\wedge^-$, and 
$$c(df)=\p f\otimes p^-+\b{\p}f\otimes p^+$$
when acting on $\mc{E}_{\mb{C}}=\mc{E}^+_{\mb{C}}\oplus \mc{E}^-_{\mb{C}}$. Since 
$$\mc{F}=\mc{E}_{\mb{C}}\oplus \wedge^-=\mc{E}_{\mb{C}}\oplus \wedge^{1,0}T^*\Sigma\otimes \mc{E}^{-}_{\mb{C}}\oplus\wedge^{0,1}T^*\Sigma\otimes \mc{E}^+_{\mb{C}},$$
so there exists a natural connection on $\mc{F}$ induced from the connections on $\mc{E}_{\mb{C}}$ and $T^*\Sigma$, we denote this connection by $\n^{\mc{F}}$.
\begin{lemma}
$\n^{\mc{F}}$ is a Clifford connection and
$$D^{\mc{F}}=c(dz)\n^{\mc{F}}_{\frac{\p}{\p z}}+c(d\b{z})\n^{\mc{F}}_{\frac{\p}{\p\b{z}}}.$$ 	
\end{lemma}
\begin{proof}
	$\n^{\mc{F}}$ is a Clifford connection if $[\n^{\mc{F}}_X,a]=\n_Xa$ for any local section $a$ of Clifford bundle $ C(\Sigma\cup_{\p\Sigma}\Sigma)$. Let $\sigma=c(\bullet)\cdot 1:C(\Sigma\cup_{\p\Sigma}\Sigma)\to \wedge^*T^*\Sigma$ denote the symbol map, which identifies $C(\Sigma\cup_{\p\Sigma}\Sigma)$ with $\wedge^*T^*\Sigma$.
	By a direct checking, one has 
	$$[\n^{\mc{F}}_{\frac{\p}{\p z}},c(dz)]=-\p_z\log g(z)c(dz)=\n_\frac{\p}{\p z}c(dz),\quad [\n^{\mc{F}}_{\frac{\p}{\p z}},c(d\b{z})]=0=\n_{\frac{\p}{\p z}}c(d\b{z}),$$
	and so
	\begin{align*}
	[\n^{\mc{F}}_{\frac{\p}{\p z}},c(dz) c(d\b{z})]	
	=[\n^{\mc{F}}_{\frac{\p}{\p z}},c(dz)] c(d\b{z})=\n_\frac{\p}{\p z}c(dz)c(d\b{z})=\n_\frac{\p}{\p z}(c(dz)c(d\b{z})).
	\end{align*}
	For any smooth function $f$, one has 
$$[\n^{\mc{F}}_X,c(f)]=\n_X c(f).$$	
Thus $\n^{\mc{F}}$ is a Clifford connection. If $s\in \Gamma(\Sigma\cup_{\p\Sigma}\Sigma,\mc{E}_{\mb{C}})$, then 
\begin{align*}
c(dz)\n^{\mc{F}}_{\frac{\p}{\p z}}s+c(d\b{z})\n^{\mc{F}}_{\frac{\p}{\p\b{z}}}s&=dz\otimes \n^{\mc{F}}_{\frac{\p}{\p z}}s^-+d\b{z}\otimes \n^{\mc{F}}_{\frac{\p}{\p\b{z}}}s^+\\
&=\pi^-\n^\mc{F}s=D^{\mc{F}}s.
\end{align*}
If $dz\otimes s^-\in\wedge ^-$, then 
\begin{align*}
	&\quad c(dz)\n^{\mc{F}}_{\frac{\p}{\p z}}(dz\otimes s^-)+c(d\b{z})\n^{\mc{F}}_{\frac{\p}{\p\b{z}}}(dz\otimes s^-)\\
	&=-2g(z)^{-1}i_{\frac{\p}{\p z}}\n^{\mc{F}}_{\frac{\p}{\p \b{z}}}dz\otimes s^-\\
	&=-2g(z)^{-1}\n^{\mc{F}}_{\frac{\p}{\p \b{z}}}s^-=(\n^{\mc{F}})^*(dz\otimes s^-)=D^{\mc{F}}(dz\otimes s^-).
\end{align*}
	Similarly, for $d\b{z}\otimes s^+\in \wedge^-$, one has $$(c(dz)\n^{\mc{F}}_{\frac{\p}{\p z}}+c(d\b{z})\n^{\mc{F}}_{\frac{\p}{\p\b{z}}})(d\b{z}\otimes s^+)=D^{\mc{F}}(d\b{z}\otimes s^+).$$
	The proof is complete.
\end{proof}
The Clifford module $\mc{F}$ has the following decomposition
\begin{align*}
\mc{F}=\mc{E}^+_{\mb{C}}\oplus\mc{E}^-_{\mb{C}}\oplus \mc{E}^+_{\mb{C}}\otimes \wedge^{0,1}T^*\Sigma\oplus \mc{E}^-_{\mb{C}}\otimes \wedge^{1,0}T^*\Sigma,	
\end{align*}
where $\mc{F}^+=\mc{E}^+_{\mb{C}}\oplus\mc{E}^-_{\mb{C}}$ and $\mc{F}^-=\mc{E}^+_{\mb{C}}\otimes \wedge^{0,1}T^*\Sigma\oplus \mc{E}^-_{\mb{C}}\otimes \wedge^{1,0}T^*\Sigma$. Let 
$$S=S^+\oplus S^-=\mb{C}\oplus\wedge^{0,1}T^*\Sigma=\wedge^{0,*}T^*\Sigma$$
denote the spinor module. Then the complex module $\mc{F}$ is isomorphic to
$$\mc{F}\cong \mc{W}\otimes S$$
where $\mc{W}=\mc{E}^+_{\mb{C}}\oplus \mc{E}^-_{\mb{C}}\otimes (\wedge^{0,1}T^*\Sigma)^*\cong \mc{E}^+_{\mb{C}}\oplus \mc{E}^-_{\mb{C}}\otimes\wedge^{1,0}T^*\Sigma$. Let $\Gamma$ be the chirality operator, which is an element  in $C(\Sigma)\otimes \mb{C}\cong \op{End}(S)$, and is $+\op{Id}$ when acting on $S^+=\mb{C}$, is $-\op{Id}$ when acting on $S^-=\wedge^{0,1}T^*\Sigma$. Thus it induces an endomorphism of $\mc{F}\cong \mc{W}\otimes S$ by the action
$\op{Id}_{\mc{W}}\otimes \Gamma$, we also denote it by $\Gamma$. Thus 
\begin{align*}
\Gamma=\op{Id}_{\mc{E}^+_{\mb{C}}}\oplus-\op{Id}_{\mc{E}^-_{\mb{C}}}\oplus -\op{Id}_{\mc{E}^+_{\mb{C}}\otimes \wedge^{0,1}T^*\Sigma}	\oplus \op{Id}_{\mc{E}^-_{\mb{C}}\otimes \wedge^{1,0}T^*\Sigma}.
\end{align*}
The Clifford connection $\n^{\mc{F}}$ is given by 
\begin{multline*}
\n^{\mc{F}}=\n^{\mc{E}^+_{\mb{C}}}\oplus \o{\n^{\mc{E}^+_{\mb{C}}}}\oplus (\n^{\mc{E}^+_{\mb{C}}}\otimes \op{Id}_{\wedge^{0,1}T^*\Sigma}+\op{Id}_{\mc{E}^+_{\mb{C}}}\otimes \n^{\wedge^{0,1}T^*\Sigma})\\\oplus ( \o{\n^{\mc{E}^+_{\mb{C}}}}\otimes \op{Id}_{\wedge^{1,0}T^*\Sigma}+\op{Id}_{\mc{E}^-_{\mb{C}}}\otimes \o{ \n^{\wedge^{0,1}T^*\Sigma}}),
\end{multline*}
and the curvature is 
{\begin{multline*}
	(\n^{\mc{F}})^2=R^{\mc{E}^+_{\mb{C}}}\oplus \o{R^{\mc{E}^+_{\mb{C}}}}\oplus (R^{\mc{E}^+_{\mb{C}}}\otimes \op{Id}_{\wedge^{0,1}T^*\Sigma}+
	R^{T^{1,0}\Sigma}\cdot\op{Id}_{\mc{E}^+_{\mb{C}}\otimes{\wedge^{0,1}T^*\Sigma} })\\\oplus ( \o{R^{\mc{E}^+_{\mb{C}}}}\otimes \op{Id}_{\wedge^{1,0}T^*\Sigma}-R^{T^{1,0}\Sigma}\cdot\op{Id}_{\mc{E}^-_{\mb{C}}\otimes { \wedge^{1,0}T^*\Sigma}} ).
\end{multline*}}
Denote by $R^{\mc{F}}$ the action of the Riemannian curvature $R$ of $\Sigma$ on the bundle $\mc{F}$, which is given by
\begin{align*}
R^{\mc{F}}&:=\frac{1}{4}\left(R\frac{\p}{\p z},\frac{\p}{\p\b{z}}\right)c(dz)c(d\b{z})+\frac{1}{4}\left(R\frac{\p}{\p\b{z}},\frac{\p}{\p z}\right)	c(d\b{z})c(dz)\\
&=\frac{g(z)}{8}R^{T^{1,0}\Sigma}(c(dz)c(d\b{z})-c(d\b{z})c(dz))\\
&=\frac{1}{4}R^{T^{1,0}\Sigma}(-\op{Id}_{\mc{E}^+_{\mb{C}}}\oplus \op{Id}_{\mc{E}^-_{\mb{C}}}\oplus\op{Id}_{\mc{E}^+_{\mb{C}}\otimes \wedge^{0,1}T^*\Sigma}\oplus-\op{Id}_{\mc{E}^-_{\mb{C}}\otimes \wedge^{1,0}T^*\Sigma}).
\end{align*}
From \cite[Proposition 4.3]{BGV}, the curvature $F^{\mc{F}/S}$ is given by
\begin{align*}
F^{\mc{F}/S}&=(\n^{\mc{F}})^2-R^{\mc{F}}\\
&=	(R^{\mc{E}^+_{\mb{C}}}+\frac{1}{4}R^{T^{1,0}\Sigma}\op{Id}_{\mc{E}^+_{\mb{C}}})\oplus(\o{R^{\mc{E}^+_{\mb{C}}}}-\frac{1}{4}R^{T^{1,0}\Sigma}\op{Id}_{\mc{E}^-_{\mb{C}}})\\
&\quad\oplus (R^{\mc{E}^+_{\mb{C}}}\otimes \op{Id}_{\wedge^{0,1}T^*\Sigma}+\frac{3}{4}\op{Id}_{\mc{E}^+_{\mb{C}}\otimes \wedge^{0,1}T^*\Sigma} R^{T^{1,0}\Sigma})\\
&\quad \oplus( \o{R^{\mc{E}^+_{\mb{C}}}}\otimes \op{Id}_{\wedge^{1,0}T^*\Sigma}-\frac{3}{4}\op{Id}_{\mc{E}^-_{\mb{C}}\otimes \wedge^{1,0}T^*\Sigma} R^{T^{1,0}\Sigma}).
\end{align*}
Thus $\Gamma F^{\mc{F}/S}$ is 
\begin{align*}
\Gamma F^{\mc{F}/S}
&=	(R^{\mc{E}^+_{\mb{C}}}+\frac{1}{4}R^{T^{1,0}\Sigma}\op{Id}_{\mc{E}^+_{\mb{C}}})\oplus(-\o{R^{\mc{E}^+_{\mb{C}}}}+\frac{1}{4}R^{T^{1,0}\Sigma}\op{Id}_{\mc{E}^-_{\mb{C}}})\\
&\quad\oplus (-R^{\mc{E}^+_{\mb{C}}}\otimes \op{Id}_{\wedge^{0,1}T^*\Sigma}-\frac{3}{4}\op{Id}_{\mc{E}^+_{\mb{C}}\otimes \wedge^{0,1}T^*\Sigma} R^{T^{1,0}\Sigma})\\
&\quad \oplus( \o{R^{\mc{E}^+_{\mb{C}}}}\otimes \op{Id}_{\wedge^{1,0}T^*\Sigma}-\frac{3}{4}\op{Id}_{\mc{E}^-_{\mb{C}}\otimes \wedge^{1,0}T^*\Sigma} R^{T^{1,0}\Sigma}).
\end{align*}
Hence the supertrace $\op{Str}_{\mc{F}/S}(F^{\mc{F}/S})$ is 
\begin{align*}
	\op{Str}_{\mc{F}/S}(F^{\mc{F}/S})&=\frac{1}{2}\op{Str}_{\mc{F}}(\Gamma F^{\mc{F}/S})=\frac{1}{2}\op{Tr}_{\mc{F}^+}(\Gamma F^{\mc{F}/S})-\frac{1}{2}\op{Tr}_{\mc{F}^-}(\Gamma F^{\mc{F}/S})\\
	&=\op{Tr}R^{\mc{E}^+_{\mb{C}}}+\frac{1}{4}\op{rank}\mc{E}^+_{\mb{C}} R^{T^{1,0}\Sigma}+\op{Tr}R^{\mc{E}^+_{\mb{C}}}+\frac{3}{4}\op{rank}\mc{E}^+_{\mb{C}} R^{T^{1,0}\Sigma}\\
	&=2\op{Tr}R^{\mc{E}^+_{\mb{C}}}+\op{rank}\mc{E}^+_{\mb{C}} R^{T^{1,0}\Sigma}.
\end{align*}
By the local index theorem, see e.g. \cite[Theorem 8.34]{Melrose}, one has 
\begin{align*}
	&\quad \lim_{t\to 0}\op{Str}\langle z|e^{-t(D^{\mc{F}})^2}|z\rangle d\mu_g \\
	&=\left[(2\pi i)^{-1}\det\left(\frac{R/2}{\sinh(R/2)}\right)\op{Str}_{\mc{F}/S}(\exp(-F^{\mc{F}/S}))\right]^{(1,1)}\\
	&=\frac{i}{2\pi}\op{Str}_{\mc{F}/S}(F^{\mc{F}/S}),
\end{align*}
since $\hat{A}(\Sigma)=\det\left(\frac{R/2}{\sinh(R/2)}\right)\in A^{4*}(\Sigma,\mb{R})$. Thus,
\begin{align}\label{Ind-}
\begin{split}
\int_{\Sigma}\alpha_-(z)d\mu_g &=\frac{i}{2\pi}\int_{\Sigma}	(2\op{Tr}R^{\mc{E}^+_{\mb{C}}}+\op{rank}\mc{E}^+_{\mb{C}} R^{T^{1,0}\Sigma})\\
&=\int_{\Sigma}(2c_1(\mc{E}^+_{\mb{C}},\n^{\mc{E}^+_{\mb{C}}})+\op{rank}\mc{E}^+_{\mb{C}} c_1(T^{1,0}\Sigma),\n^{T^{1,0}\Sigma)}).
\end{split}
\end{align}
Similarly, one has
\begin{align}\label{Ind+}
\int_{\Sigma}\alpha_+(z)d\mu_g =\int_{\Sigma}(-2c_1(\mc{E}^+_{\mb{C}},\n^{\mc{E}^+_{\mb{C}}})+\op{rank}\mc{E}^+_{\mb{C}} c_1(T^{1,0}\Sigma,\n^{T^{1,0}\Sigma})).
\end{align}
Therefore,
\begin{align}\label{4.1}
	\int_{\Sigma}\alpha_{-}(z)d\mu_g-\int_{\Sigma}\alpha_{+}(z)d\mu_g=4\int_{\Sigma}c_1(\mc{E}^+_{\mb{C}},\n^{\mc{E}^+_{\mb{C}}}).
\end{align}

Note that on $\p\Sigma\times[0,1/2]$, $\n^{\mc{E}^+_{\mb{C}}}=\n|_{\mc{E}^+_{\mb{C}}}$ is a flat connection, which follows that $c_1(\mc{E}^+_{\mb{C}},\n^{\mc{E}^+_{\mb{C}}})=0$ on $\p\Sigma\times [0,1/2]$.
It is also note that $\mc{E}^+_{\mb{C}}$ is complex linear isomorphic to the complex vector bundle $(\mc{E},\mbf{J})$, with the isomorphism is given by 
$\Psi:(\mc{E},\mbf{J})\to\mc{E}^+_{\mb{C}}$, $\Psi(e)=\frac{1}{2}(\op{Id}-i\mbf{J})e$. Thus the first Chern class 
$c_1(\mc{E},\mbf{J})=c_1(\mc{E}^+_{\mb{C}})$ can be represented by a differential form $c_1(\mc{E},\mbf{J},\n)=c_1(\mc{E}^+_{\mb{C}},\n^{\mc{E}^+_{\mb{C}}})$ with compact support. We define
\begin{align}\label{FCC}
	c_1(\mc{E},\Omega, \mbf{J}):=[c_1(\mc{E}^+_{\mb{C}},\n^{\mc{E}^+_{\mb{C}}})]_c\in \mathrm{H}^2_{\text{dR,comp}}(\Sigma_o ,\mb{R})
\end{align}
the de Rham cohomology class of $c_1(\mc{E}^+_{\mb{C}},\n^{\mc{E}^+_{\mb{C}}})$ with compact support, see e.g. \cite[Chapter 1]{BT} for the definition of de Rham cohomology with compact support,  where $\Sigma_o:=\Sigma\backslash\p\Sigma$. 
Denote
\begin{align*}
	\mc{J}(\mc{E},\Omega, \mbf{J})=\{\mbf{J}'\in \mc{J}(\mc{E},\Omega): \mbf{J}'=\mbf{J}\text{ on a collar neighborhood of } \p\Sigma\}.
\end{align*}
From the following Remak \ref{rem411},   the class $c_1(\mc{E},\Omega,\mbf{J})$  is independent of the complex structure $\mbf{J}'\in \mc{J}(\mc{E},\Omega, \mbf{J})$ and  the peripheral connections, so it is well-defined. 
Note that $$\int_{\Sigma}c_1(T^{1,0}\Sigma),\n^{T^{1,0}\Sigma)})=\chi(\Sigma)=2-2g-q,$$ where $q$ denotes the number of components in $\p\Sigma$. Hence 
\begin{prop}\label{prop410}
	For any $\mbf{J}'\in \mc{J}(\mc{E},\Omega, \mbf{J})$, and $\n$ is any peripheral connection on $\mc{E}$, then 
	\begin{equation*}
  \int_{\Sigma}\alpha_{\pm}(z)d\mu_g=\mp2\int_{\Sigma}c_1(\mc{E},\Omega, \mbf{J}')+\frac{\dim E}{2}\chi(\Sigma),
  \end{equation*}     
 where $c_1(\mc{E},\Omega, \mbf{J}')$ denotes the cohomology class with compact support, which is defined by \eqref{FCC}.
\end{prop}
 \begin{rem}\label{rem411}
 	From Lemma \ref{lemma4.7}, the terms $\int_{\Sigma}\alpha_{\pm}(z)d\mu_g$ are independent of the compatible complex structures  in $\mc{J}(\mc{E},\Omega, \mbf{J})$ and the peripheral connections. From the above proposition,  the first Chern class $c_1(\mc{E},\Omega, \mbf{J}')$ with compact support is also independent of $\mbf{J}'\in \mc{J}(\mc{E},\Omega, \mbf{J})$ and the peripheral connection $\n$. 
 \end{rem}
Substituting \eqref{4.1} into \eqref{sign}, one gets
\begin{align}\label{sign1}
\op{sign}(\mc{E},\Omega)=4\int_\Sigma c_1(\mc{E},\Omega, \mbf{J})+h_{\infty}(\wedge^-)-h_{\infty}(\wedge^+)+\eta(A_{\mbf{J}}).	
\end{align}

\subsubsection{Limiting values of extended $\mathrm{L}^2$-sections}

In this subsection, we will calculate the terms $h_{\infty}(\wedge^{\pm})$ and show $h_{\infty}(\wedge^-)=h_{\infty}(\wedge^+)$.

By Atiyah-Patodi-Singer index theorem \cite[Theorem 3.10]{APS}, one has
\begin{align}\label{3.8}
\op{Index}(d^-_P)=	\int_{\Sigma}\alpha_{-}(z)dz-\frac{\dim \mathrm{H}^0(\p\Sigma,\mc{E})- \eta(A_{\mbf{J}})}{2}.
\end{align}
On the other hand, we have
\begin{align}\label{3.9}
\op{Index}(d^-_P)+h_{\infty}(\wedge^-)=\mathrm{L}^2\op{Index}(d^-).	
\end{align}
Following \cite[(3.20)--(3.25)]{APS},  we consider the operator $(d^-)^*$, then 
$$(d^-)^*=-(\sigma^-)^{-1}\left(\frac{\p}{\p u}+\sigma^-A_{\mbf{J}}(\sigma^-)^{-1}\right).$$
Since $\eta(\sigma^-A_{\mbf{J}}(\sigma^-)^{-1})=\eta(A_{\mbf{J}})$, so
\begin{align*}
\op{Index}(d^-)^*_P+h_{\infty}(\mc{E}_{\mb{C}})=\mathrm{L}^2\op{Index}(d^-)^*=-\mathrm{L}^2\op{Index}(d^-),	
\end{align*}
\begin{align*}
	\op{Index}(d^-)^*_P=-\int_{\Sigma}\alpha_-(z)dz-\frac{\dim \mathrm{H}^0(\p\Sigma,\mc{E})+\eta(A_{\mbf{J}})}{2}.
\end{align*}
Combining with the above equalities, we have
\begin{equation}\label{hinfE}
  h_{\infty}(\mc{E}_{\mb{C}})+h_{\infty}(\wedge^-)=\dim \mathrm{H}^0(\p\Sigma,\mc{E}).
\end{equation}

Denote by $\ms{K}^-$ the set of all  extended $\mathrm{L}^2$-solutions of $(d^-)^*\phi=0$ in  $\wedge^-$, that is,  for any $\phi\in \ms{K}^-$, one has  $d^*\phi=0$ and $\phi$ is with valued in $\wedge^-$, and in the cylinder $\p\Sigma\times (-\infty,u_0]$ for some large negative $u_0$, we can write 
$$\phi=\psi+\theta$$
where $\psi=\psi_0+\psi_1du\in \op{Ker}(\sigma^-A_{\mbf{J}}(\sigma^-)^{-1})$ and $\theta\in \op{Ker}(d^-)^*\cap \mathrm{L}^2(\widehat{\Sigma},\wedge^-)$ is a $\mathrm{L}^2$-section in $\wedge^-$ (hence decaying exponentially). From Proposition \ref{prop3}, one has $d\theta=0$. For any $\mbf{J}\in \mc{J}(\mc{E},\Omega)$ and extend it to the vector bundle  $\mc{E}$ over $\widehat{\Sigma}$ such that  $\mbf{J}=\mbf{J}(x)$ on $\p\Sigma\times (-\infty,u_0]$. By the definition of $\sigma^-$, see Section \ref{TAPST}, then $[d,\sigma^-]=0$. Since $\psi\in \op{Ker}(\sigma^-A_{\mbf{J}}(\sigma^-)^{-1})$ and by \eqref{KerAJ} so $(\sigma^-)^{-1}\psi\in \op{Ker}(A_{\mbf{J}})=\op{Ker}d$. Hence $d\psi=d\sigma^-(\sigma^-)^{-1}\psi=\sigma^-d((\sigma^-)^{-1}\psi)=0$, which follows all elements of $\ms{K}^-$ are harmonic.
 Denote by $\delta^-:\mathscr{K}^-\to \mathrm{H}^1(\Sigma,\mc{E}_{\mb{C}})$ the natural map, then  
$$*\mbf{J}(\psi_0+\psi_1du)=-(\psi_0+\psi_1du),\quad *\mbf{J}(\theta)=-\theta.$$
If moreover, $\psi_0=0$, then $*\mbf{J}(\psi_1du)=-\psi_1du$. However,  since $*du=-\frac{dx}{|dx|}$, so we conclude that 
$\psi_1=0$, and so $\phi=\theta\in\op{Ker}(d^-)^*\cap \mathrm{L}^2(\widehat{\Sigma},\wedge^-)$, then 
$$\op{Ker}(\iota^*\delta^-)=\op{Ker}(d^-)^*\cap \mathrm{L}^2(\widehat{\Sigma},\wedge^-),$$
where $\iota^*: \mathrm{H}^{1}(\Sigma,\mc{E}_{\mb{C}})\to \mathrm{H}^{1}(\p\Sigma,\mc{E}_{\mb{C}})$ is the  induced map on cohomology by restriction.
By the definition of limit values of extended $\mathrm{L}^2$-sections \cite{APS}, $h_{\infty}(\wedge^-)$ is the dimension of subspace of all  $\psi$, so $h_{\infty}(\wedge^-)=\dim (\mathscr{K}^-/\op{Ker}(d^-)^*\cap \mathrm{L}^2(\widehat{\Sigma},\wedge^-))$, and we have 
\begin{align*}
h_{\infty}(\wedge^-)=\dim (\mathscr{K}^-/\op{Ker}(\iota^*\delta^-))=\dim\op{Im}(\iota^*\delta^-)\leq \dim\op{Im}(\iota^*).
\end{align*}
\begin{lemma}
We have 
$$\dim\op{Im}(\iota^*)=\dim \mathrm{H}^0(\p\Sigma,\mc{E})-\dim \mathrm{H}^0(\Sigma,\mc{E}). $$	
\end{lemma}
\begin{proof}
From the following exact sequence
\begin{align*}
\cdots \to \mathrm{H}^{1}(\Sigma,\mc{E}_{\mb{C}}) \stackrel{\iota^{*}}{\longrightarrow} \mathrm{H}^{1}(\p\Sigma,\mc{E}_{\mb{C}}) \stackrel{\alpha^{*}}{\rightarrow} \mathrm{H}^{2}(\Sigma,\p\Sigma,\mc{E}_{\mb{C}}) \stackrel{\beta^{*}}{\rightarrow} \mathrm{H}^{2}(\Sigma,\mc{E}_{\mb{C}}) \to 0, 
\end{align*}
one has
$$ \mathrm{H}^1(\p\Sigma,\mc{E}_{\mb{C}})/\op{Im}\iota^*\simeq \dim \mathrm{H}^1(\p\Sigma,\mc{E}_{\mb{C}})/\op{Ker}\alpha^*\simeq \op{Im}\alpha^*\simeq \op{Ker}\beta^*,$$
and 
$$\mathrm{H}^2(\Sigma,\p\Sigma,\mc{E}_{\mb{C}})/\op{Ker}\beta^*\simeq\op{Im}\beta^*\simeq \mathrm{H}^2(\Sigma,\mc{E}_{\mb{C}}).$$
Hence 
\begin{align*}
\dim\op{Im}(\iota^*)
&=\dim \mathrm{H}^1(\p\Sigma,\mc{E}_{\mb{C}})-(\dim \mathrm{H}^2(\Sigma,\p\Sigma,\mc{E}_{\mb{C}})-\dim \mathrm{H}^2(\Sigma,\mc{E}_{\mb{C}}))\\
&=\dim \mathrm{H}^0(\p\Sigma,\mc{E}_{\mb{C}})-(\dim \mathrm{H}^0(\Sigma,\mc{E}_{\mb{C}})-\dim \mathrm{H}^2(\Sigma,\mc{E}_{\mb{C}}))\\	
&=\dim \mathrm{H}^0(\p\Sigma,\mc{E})-(\dim \mathrm{H}^0(\Sigma,\mc{E})-\dim \mathrm{H}^2(\Sigma,\mc{E}))\\
&=	\dim \mathrm{H}^0(\p\Sigma,\mc{E})-\dim \mathrm{H}^0(\Sigma,\mc{E}),
\end{align*}
where the second equality by the Poincar\'e duality, and the third equality by considering the real dimension, the last equality follows from the fact that $\dim \mathrm{H}^2(\Sigma,\mc{E})=0$.	
\end{proof}
Hence
\begin{align}\label{hinf}
	h_{\infty}(\wedge^-)\leq \dim\op{Im}(\iota^*)=\dim \mathrm{H}^0(\p\Sigma,\mc{E})-\dim \mathrm{H}^0(\Sigma,\mc{E}).
\end{align}

On the other hand, we can also consider the term $h_{\infty}(\mc{E}_{\mb{C}})$. Denote by $\mathscr{K}_0^-$ the set of all extended  $\mathrm{L}^2$ solutions of $d^-\phi=0$ in $\mc{E}_{\mb{C}}$, so that for any $\phi\in \mathscr{K}_0^-$, one has $d\phi=0$. In the cylinder $\p\Sigma\times (-\infty,u_0]$, we can write
$\phi=\psi+\theta$
where $\psi\in\op{Ker}(A_{\mbf{J}})$ is a harmonic section on $\p\Sigma$ and $\theta$ is a $\mathrm{L}^2$ harmonic section. 
From \eqref{H0} and \eqref{KerAJ}, one has $d\psi=d\theta=0$. Hence $d\phi=0$ for any $\phi\in \ms{K}^-_0$.
 Denote by $\delta_0^-:\mathscr{K}_0^-\to \mathrm{H}^0(\Sigma,\mc{E}_{\mb{C}})$ the natural map, then 
$$\op{Ker}(\iota_0^*\delta_0^-)=\op{Ker}(d)\cap \mathrm{L}^2(\widehat{\Sigma},\mc{E}_{\mb{C}}),$$
where $\iota_0^*$ is defined by
$$
\cdots \to0 \stackrel{\beta^{*}}{\rightarrow} \mathrm{H}^{0}(\Sigma,\mc{E}_{\mb{C}}) \stackrel{\iota_0^{*}}{\longrightarrow} \mathrm{H}^{0}(\p\Sigma,\mc{E}_{\mb{C}}) \rightarrow \cdots
$$
Since $h_{\infty}$ is the dimension of the space of all $\psi$, so
 $h_{\infty}(\mc{E}_{\mb{C}})=\dim (\mathscr{K}_0^-/\op{Ker}(d)\cap \mathrm{L}^2(\widehat{\Sigma},\mc{E}_{\mb{C}}))$, and we have
\begin{align}\label{hE}
\begin{split}
h_{\infty}(\mc{E}_{\mb{C}})&=\dim (\mathscr{K}_0^-/\op{Ker}(\iota_0^*\delta_0^-))=\dim\op{Im}(\iota^*_0\delta_0^-)\\
&\leq \dim\op{Im}(\iota_0^*)=\dim \mathrm{H}^0(\Sigma,\mc{E}).
\end{split}
\end{align}
From \eqref{hinfE}, \eqref{hinf} and \eqref{hE}, we obtain
\begin{equation*}
  h_{\infty}(\wedge^-)=\dim \mathrm{H}^0(\p\Sigma,\mc{E})-\dim \mathrm{H}^0(\Sigma,\mc{E})
\end{equation*}
Similarly, one has 
\begin{equation}\label{hpm}
  h_{\infty}(\wedge^+)=\dim \mathrm{H}^0(\p\Sigma,\mc{E})-\dim \mathrm{H}^0(\Sigma,\mc{E})=h_{\infty}(\wedge^-).
\end{equation}
Substituting \eqref{hpm} into \eqref{sign1}, we obtain a formula for signature:
\begin{thm}\label{thmsign2}
	The signature is given by
	\begin{equation}\label{sign2}
	\op{sign}(\mc{E},\Omega)= 4\int_{\Sigma}c_1(\mc{E},\Omega, \mbf{J})+\eta(A_{\mbf{J}}).
\end{equation}
	\end{thm}

\section{Toledo invariant}\label{Tol}

In this section, we will recall the definition of Toledo invariant $\op{T}(\Sigma,\rho)$ and prove 
\begin{equation*}
  \op{T}(\Sigma,\rho)=2\int_{\Sigma}c_1(\mc{E},\Omega, \mbf{J})-\frac{1}{2\pi}\sum_{i=1}^q\int_{c_i}\wt{\mbf{J}}^*\alpha_i
\end{equation*}
 for any $\mbf{J}\in\mc{J}_o(\mc{E},\Omega)$, see \eqref{J0space} for the definition of $\mc{J}_o(\mc{E},\Omega)$, where $\alpha_i$ is a $\rho(c_i)$-invariant one form with $d\alpha_i=\omega_{\op{D}^{\op{III}}_n}$, see \eqref{Fixpoint2}. The Rho invariant is defined by $$\bs{\rho}(\p\Sigma)=\frac{1}{\pi}\sum_{i=1}^q\int_{c_i}\wt{\mbf{J}}^*\alpha_i+\eta(A_{\mbf{J}}),$$ see Definition \ref{Rho invariant}. From Theorem \ref{thmsign2}, we obtain 
\begin{thm}\label{thmsign22}
	The signature is given by
	\begin{equation}\label{sign2}
	\op{sign}(\mc{E},\Omega)=2\op{T}(\Sigma,\rho)+\bs{\rho}(\p\Sigma).
\end{equation}
	\end{thm}
\begin{rem}\label{remunitary}
	Following \cite{APSII}, if we consider the unitary representation, i.e. $$\rho:\pi_1(\Sigma)\to \op{U}(n)=\{A+iB\in \op{U}(n)\}\cong\left\{Z=\left(\begin{matrix}
A & B\\
-B & A	
\end{matrix}
\right)\in\op{Sp}(2n,\mb{R})\right\}.$$ Then $[Z,J]=0$ where $J=\left(\begin{matrix}
0 &-I_n\\
I_n &0	
\end{matrix}
\right)$ is the standard complex structure. Hence $J\in \mc{J}_o(\mc{E},\Omega)$, and so 
\begin{align*}
  \op{T}(\Sigma,\rho)&=2\int_{\Sigma}c_1(\mc{E},\Omega, {J})-\frac{1}{2\pi}\sum_{i=1}^q\int_{c_i}\wt{{J}}^*\alpha_i\\
  &=\frac{1}{2\pi}\int_\Sigma \wt{{J}}^*\omega_{\op{D}^{\op{III}}_n}-\frac{1}{2\pi}\sum_{i=1}^q\int_{c_i}\wt{{J}}^*\alpha_i=0,
\end{align*}
where the second equality by Section \ref{FBPB},
and \begin{equation*}
  \bs{\rho}(\p\Sigma)=\frac{1}{\pi}\sum_{i=1}^q\int_{c_i}\wt{{J}}^*\alpha_i+\eta(A_{{J}})=\eta(A_J).
\end{equation*}
Thus, Theorem \ref{thmsign22} gives 
$$\op{sign}(\mc{E},\Omega)=  \bs{\rho}(\p\Sigma)=\eta(A_{{J}}),$$
which is agree with \cite[Theorem 2.2, Theorem 2.4]{APSII}.

\end{rem}

\subsection{Definition of Toledo invariant}\label{Toledo}

Let $\Sigma$ be a connected oriented surface with boundary $\p\Sigma$, and $\rho:\pi_1(\Sigma)\to G$ be a surface group representation into a Lie group $G$ which is of Hermitian type. Burger,  Iozzi and  Wienhard \cite[Section 1.1]{BIW} introduced the definition of Toledo invariant $\op{T}(\Sigma,\rho)$, which generalizes
 the Toledo invariant for closed surface case. 

A Lie group $G$ is of {\it Hermitian type} if it is connected, semisimple with finite center and no compact factors, and if the associated symmetric space is Hermitian. Let $G$ be a group of Hermitian type so that in particular the associated symmetric space $\mathscr{X}$ is Hermitian of noncompact type, then $\mathscr{X}$ carries a unique Hermitian (normalized) metric of minimal holomorphic sectional curvature $-1$. The associated K\"ahler form $\omega_{\mathscr{X}}$ is in $\Omega^2(\ms{X})^G$ the space of $G$-invariant $2$-forms on $\ms{X}$. A Lie group $G$ is of type $(\mathrm{RH})$ if it is connected reductive with compact center and the quotient $G / G_{c}$ by the largest connected compact normal subgroup $G_{c}$ is of Hermitian type.
 By the van Est isomorphism \cite{Van},
$\Omega^2(\ms{X})^G\cong \mathrm{H}^2_c(G,\mb{R})$,  
where $\mathrm{H}^{\bullet}_c(G,\mb{R})$ denotes the continuous cohomology of the group $G$ with $\mb{R}$-trivial coefficients, there exists a unique class $\kappa_G\in \mathrm{H}^2_c(G,\mb{R})$ corresponding to the K\"ahler form $\omega_{\ms{X}}$, and thus gives rise to a bounded K\"ahler class $\kappa^b_G\in \mathrm{H}^2_{c,b}(G,\mb{R})$ by the  isomorphism \cite{BO},
$\mathrm{H}^2_{c}(G,\mb{R})\cong \mathrm{H}^2_{c,b}(G,\mb{R})$, 
where $\mathrm{H}^{\bullet}_{c,b}(G,\mb{R})$ denotes the bounded continuous cohomology. In fact, $\kappa^b_G\in  \mathrm{H}^2_{c,b}(G,\mb{R})$ is defined by a bounded cocycle 
\begin{equation}\label{bounded cocyle}
  c(g_0,g_1,g_2)=\frac{1}{2\pi}\int_{\triangle(g_0x, g_1x, g_2x)} \omega_{\ms{X}},
\end{equation}
 where $\triangle(g_0x, g_1x, g_2x)$ is a geodesic triangle with ordered vertices $g_0x, g_1x,g_2x$ for some base point $x\in \ms{X}$.

 By Gromov isomorphism \cite{Gromov}, one has 
$$\rho_b^*(\kappa^b_G)\in \mathrm{H}^2_b(\pi_1(\Sigma),\mb{R})\cong \mathrm{H}^2_b(\Sigma,\mb{R}).$$
The canonical map $j_{\p\Sigma}:\mathrm{H}^2_b(\Sigma,\p\Sigma,\mb{R})\to \mathrm{H}^2_b(\Sigma,\mb{R})$ from singular bounded cohomology relative to $\p\Sigma$ to singular bounded cohomology is an isomorphism. Then the Toledo invariant is defined as 
$$\op{T}(\Sigma,\rho)=\langle j^{-1}_{\p\Sigma}\rho_b^*(\kappa^b_G),[\Sigma,\p\Sigma]\rangle,$$
where   $j^{-1}_{\p\Sigma}\rho_b^*(\kappa^b_G)$ is considered as an ordinary relative cohomology class and $[\Sigma,\p\Sigma]\in H_2(\Sigma,\p\Sigma,\mb{Z})\cong\mb{Z}$ denotes the relative fundamental class.

\subsection{Relation to the first Chern class}

In this subsection, we will prove $\op{T}(\Sigma,\rho)=2\int_{\Sigma}c_1(\mc{E},\Omega, \mbf{J})-\frac{1}{2\pi}\sum_{i=1}^q\int_{c_i}\wt{\mbf{J}}^*\alpha_i$. We will divide it into two steps: in the first step we will prove $\int_\Sigma c_1(\mc{E},\Omega, \mbf{J})=\frac{1}{4\pi}\int_\Sigma \wt{\mbf{J}}^*\omega_{\op{D}^{\op{III}}_n}$, and in the second step we will prove  $\op{T}(\Sigma,\rho)=\frac{1}{2\pi}\int_\Sigma \wt{\mbf{J}}^*\omega_{\op{D}^{\op{III}}_n}-\frac{1}{2\pi}\sum_{i=1}^q\int_{c_i}\wt{\mbf{J}}^*\alpha_i$.

\subsubsection{Flat bundle and pullback bundle}\label{FBPB}

In this subsection, we will prove the first step.
Let $\rho:\pi_1(\Sigma)\to \op{Sp}(E,\Omega)$ be a surface group representation into a real symplectic  group $\op{Sp}(E,\Omega)$. The maximal compact subgroup of $\op{Sp}(E,\Omega)$ is a unitary group $\op{U}(E,\Omega)$, and the space $\mc{J}(E,\Omega)$ of all compatible complex structures is isomorphic to
$$\mc{J}(E,\Omega)\cong \op{Sp}(E,\Omega)/\op{U}(E,\Omega),$$ 
which is a Hermitian symmetric space of non-compact type. 

By taking a symplectic basis such that the symplectic form $\Omega$ is given by the following matrix
\begin{align}\label{symplectic form}
\Omega=\left(\begin{array}{cc}
0 & I_n\\	
-I_n &0
\end{array}
\right),	
\end{align}
then $\op{Sp}(E,\Omega)\cong \op{Sp}(2n,\mb{R})$ and $\op{U}(E,\Omega)\cong \op{U}(n)\cong \op{Sp}(2n,\mb{R})\cap \op{SO}(2n)$  is a maximal compact subgroup of $\op{Sp}(2n,\mb{R})$, $\dim E=2n$, and  the space $\mc{J}(E,\Omega)\cong \op{Sp}(2n,\mb{R})/\op{U}(n)$ can be identified with Siegel disk:
$$\op{D}^{\op{III}}_n:=\{W\in \mf{g}\mf{l}(n,\mb{C}):W=W^\top,I_n-\o{W}W>0\},$$
see e.g. \cite[Chapter VIII, \S 7]{Hel}. In fact, the identification is given by the following map
$$\sigma: \op{Sp}(2n,\mb{R})/\op{U}(n)\to \op{D}^{\op{III}}_n,$$
$$ \sigma:g\op{U}(n)\mapsto \left((C+B)i+(D-A)\right)\left((C-B)i+(D+A)\right)^{-1},$$
is a holomorphic diffeomorphism, where $
g=\left(\begin{array}{ll}
A & B \\
C & D
\end{array}\right)
$, $A^\top C$ and $D^\top B$ are symmetric and $A^\top D-C^\top B=I_n$,
see \cite[Page 399]{Hel}. By this identification, the representation $\rho:\pi_1(\Sigma)\to \op{Sp}(E,\Omega)$ induces  a canonical  representation from $\pi_1(\Sigma)$ into the holomorphic automorphism group  $\op{Aut}(\op{D}^{\op{III}}_n)$, we also denote it by $\rho$, see \eqref{complexstructureF} and \eqref{Action}.

Let $\omega_{\op{D}^{\op{III}}_n}$ be the K\"ahler form such that its holomorphic sectional curvature is in $[-1,-1/n]$. From  Appendix \ref{App2}, one has
$$\omega_{\op{D}^{\op{III}}_n}:=-2i\p\b{\p}\log\det(I_n-\o{W}W).$$
For the bounded symmetric domain $\op{D}^{\op{III}}_n$, we consider a trivial symplectic vector bundle over it, i.e. 
$$(F,\Omega):=\op{D}^{\op{III}}_n\times (\mb{R}^{2n},\Omega)\to \op{D}^{\op{III}}_n,$$
where $\Omega$ is given by \eqref{symplectic form}. 
There exists a canonical complex structure $\mbf{J}_F$ on $F$, which is defined as follows:
\begin{align}\label{complexstructureF}
\mbf{J}_F(W):=\left(\begin{array}{cc}
	-\op{Im}A+\op{Im} B & \op{Re}A+\op{Re}B\\
	-\op{Re}A+\op{Re}B & -\op{Im}A-\op{Im} B
\end{array}\right)=U\left(\begin{matrix}{}
  i A& iB \\
  -i\o{B}&-i\o{A}  
\end{matrix}\right)	U^{-1},
\end{align}
where 
\begin{align}\label{U}
U=\left(\begin{matrix}{}
 -iI_n & iI_n \\
  I_n&I_n 
\end{matrix}
\right),\quad U^{-1}=\left(\begin{matrix}{}
 \frac{i}{2}I_n & \frac{1}{2}I_n \\
 -\frac{i}{2} I_n&\frac{1}{2}I_n 
\end{matrix}
\right),
\end{align}
and 
$$A=-(I_n-W\o{W})^{-1}(I_n+W\o{W}),\quad B=2(I_n-W\o{W})^{-1}W.$$
Thus 
\begin{align}\label{W}
W=(I_n-A)^{-1}B.	
\end{align}

\begin{prop}
$\mbf{J}_F\in \mc{J}(F,\Omega)$.	
\end{prop}
\begin{proof}
	 By using $(I_n-W\o{W})^{-1}W=W(I_n-\o{W}W)^{-1}$, one can get $A^2-B\o{B}=I_n$ and $AB=B\o{A}$, so $\mbf{J}_F(W)^2=-I_{2n}$. On the other hand,
	\begin{align*}
	\Omega \mbf{J}_F(W)&=\Omega U\left(\begin{matrix}{}
  i A& iB \\
  -i\o{B}&-i\o{A}  
\end{matrix}\right)	U^{-1}=U(U^{-1}\Omega U)\left(\begin{matrix}{}
  i A& iB \\
  -i\o{B}&-i\o{A}  
\end{matrix}\right)	U^{-1}\\
&=U\left(\begin{matrix}
iI_n & 0\\
0 & -i I_n	
\end{matrix}
\right)\left(\begin{matrix}{}
  i A& iB \\
  -i\o{B}&-i\o{A}  
\end{matrix}\right)	U^{-1}=-U\left(\begin{matrix}{}
 A& B \\
  \o{B}&\o{A}  
\end{matrix}\right)U^{-1}\\
&=-U\left(\begin{matrix}
I_n &0	\\
-\o{B}A^{-1} & I_n
\end{matrix}
\right)^{-1}\left(\begin{matrix}
A & 0\\
0 & \o{A}-\o{B}A^{-1}B	
\end{matrix}
\right)\left(\left(\begin{matrix}
I_n &0	\\
-\o{B}A^{-1} & I_n
\end{matrix}
\right)^{-1}\right)^*U^{-1}.
	\end{align*}
Since $A$ is strictly negative definite and $\o{A}-\o{B}A^{-1}B	=-(I_n-\o{W}W)(I_n+\o{W}W)^{-1}$ is also strictly negative definite, so {$\Omega \mbf{J}_F(W)$} is strictly positive definite, which means that $\Omega(\cdot,\mbf{J}_F(W)\cdot)>0$. { Now we show that $\Omega$ is $\mbf{J}_F(W)$-invariant, 
 and} $\mbf{J}_F(W)^\top\Omega \mbf{J}_F(W)=\Omega$ is equivalent to $\Omega \mbf{J}_F(W)$ { being} symmetric. So
\begin{align*}
\Omega \mbf{J}_F(W)=	-U\left(\begin{matrix}{}
 A& B \\
  \o{B}&\o{A}  
\end{matrix}\right)U^{-1}=-\left(\begin{matrix}
\op{Re}A-\op{Re}B &\op{Im}A+\op{Im}B\\
-\op{Im}A+\op{Im}B &\op{Re}A+\op{Re}B	
\end{matrix}
\right),
\end{align*}
and note that $A^*=A$, $B^\top=B$, so $(\Omega \mbf{J}_F(W))^\top=\Omega \mbf{J}_F(W)$. Therefore, $\mbf{J}_F\in \mc{J}(F,\Omega)$.
\end{proof}
\begin{rem}\label{rem5.2}
	The complex structure $\mbf{J}_F$ is also a smooth map $\mbf{J}_F:\op{D}^{\op{III}}_n\to \mc{J}(\mb{R}^{2n},\Omega)$. One can check that it is a bijection. For any $Z\in \op{Sp}(2n,\mb{R})$, then  $Z\mbf{J}_FZ^{-1}\in \mc{J}(\mb{R}^{2n},\Omega)$, and the induced action on $\op{D}^{\op{III}}_n$ is given by
	\begin{equation}\label{Action}
  Z(W)=(Z_1W+Z_2)(\o{Z_2}W+\o{Z_1})^{-1}\in \op{D}^{\op{III}}_n,
\end{equation}
 where $Z_1$ and $Z_2$ is defined by
	 $Z=U\left(\begin{matrix}
Z_1 & Z_2\\
\o{Z_2} & \o{Z_1}	
\end{matrix}
\right)U^{-1}$. The isomorphism $\op{Sp}(2n,\mb{R})/\op{U}(n)\cong\mc{J}(\mb{R}^{2n},\Omega)$ is given by $Z\cdot \op{U}(n)\mapsto ZJZ^{-1}$, where $J=\left(\begin{matrix}
0 & -I_n\\
I_n & 0	
\end{matrix}
\right)$ and $\op{U}(n)\cong \{Z\in \op{Sp}(2n,\mb{R}): ZJZ^{-1}=J\}$.
\end{rem}

Hence $(F,\mbf{J}_F)$ is complex vector bundle over $\op{D}^{\op{III}}_n$. Next we will  give a complex connection on the complex
vector bundle. Note that
\begin{align*}
\left(\begin{matrix}{}
  i A& iB \\
  -i\o{B}&-i\o{A}  
\end{matrix}\right)=\left(\begin{matrix}
  I_n & -W\\
  \o{W}&-I_n
\end{matrix}
\right)	\left(\begin{matrix}
 -i I_n & 0\\
 0&iI_n
\end{matrix}
\right)	\left(\begin{matrix}
  I_n & -W\\
  \o{W}&-I_n
\end{matrix}
\right)^{-1},
\end{align*}
if we set
\begin{align*}
V=		U\left(\begin{matrix}
  I_n & -W\\
  \o{W}&-I_n
\end{matrix}
\right)=\left(\begin{matrix}
  -i(I_n-\o{W}) & -i(I_n-W)\\
  I_n+\o{W}&-I_n-W
\end{matrix}
\right),
\end{align*}
then
\begin{equation}\label{JV}
  \mbf{J}_F=V	\left(\begin{matrix}
 -i I_n & 0\\
 0&iI_n
\end{matrix}
\right)V^{-1}.
\end{equation}
One can define a connection  on $F_{\mb{C}}$ by 
\begin{align*}
	\n^F=V\left(\begin{matrix}
 d+(I_n-W\o{W})^{-1}\b{\p}(I_n-W\o{W}) & 0\\
 0&  d+(I_n-\o{W}W)^{-1}\p(I_n-\o{W}W)
\end{matrix}
\right)V^{-1}.
\end{align*}
One can check that $\n^F=\o{\n^F}$ and $\n^F \mbf{J}_F=\mbf{J}_F\n^F$. Hence $\n^F$ gives rise to a complex connection on the complex vector bundle $(F,\mbf{J}_F)$. By a direct calculation, one has $\n^F=d+C$, where 
\begin{align*}
C=V\left(\begin{matrix}
0 & (I_n-W\o{W})^{-1}dW	\\
(I_n-\o{W}W)^{-1}d\o{W} &0
\end{matrix}
\right)	V^{-1},
\end{align*}
and $C$ satisfies $C=\o{C}$ and $C^\top\Omega+\Omega C=0$, i.e. $C\in \mf{s}\mf{p}(2n,\mb{R})$.

The symplectic form $\Omega$ defines a Hermitian metric on $F_{\mb{C}}$ by $(\cdot,\cdot)=2\Omega(\cdot,\mbf{J}_F\o{\cdot})$.
Let $F_{\mb{C}}=F_{\mb{C}}^+\oplus F_{\mb{C}}^-$ be the decomposition of $F_{\mb{C}}$ corresponding the $\pm i$-eigenspaces of $\mbf{J}_F$. Let $\{\tilde{f}_i$, $1\leq i\leq n\}$ denote a frame of $F_{\mb{C}}^-$ and  $\{\tilde{f}_{n+i}$, $1\leq i\leq n\}$ denote a frame of $F_{\mb{C}}^+$. Here we  take $\tilde{f}_j=V_j^kf_k$ where $\{f_j, 1\leq j\leq 2n\}$ denotes the standard basis of $\mb{R}^{2n}$. With respect to the frame $\{\tilde{f}_i,1\leq i\leq 2n\}$, the Hermitian metric is given by the following matrix
\begin{align}\label{Hermitian matrix}
2\Omega(V\cdot,\mbf{J}_F\o{V}\cdot)=2V^\top\Omega \mbf{J}_F\o{V}=2(V^*\Omega \mbf{J}_F V)^\top	=4\left(\begin{matrix}
  I_n-\o{W}W& 0 \\
 0 & I_n-W\o{W}
\end{matrix}
\right)
\end{align}
where the second equality 
since $\Omega \mbf{J}_F$ is symmetric. Thus $\n^F|_{F^+_{\mb{C}}}$ is a Chern connection of the holomorphic Hermitian vector bundle $(F^+_{\mb{C}},2\Omega(V\cdot,\mbf{J}_F\o{V}\cdot)|_{F_{\mb{C}}^+})$. The first Chern form is 
\begin{align}\label{Chern forms 1}
c_1(F^+_{\mb{C}},\n^F|_{F^+_{\mb{C}}})=\frac{i}{2\pi}\b{\p}\p\log\det(I_n-W\o{W})=\frac{1}{2}\cdot\frac{1}{2\pi}\omega_{\op{D}^{\op{III}}_n}.	
\end{align}

For any representation $\rho:\pi_1(\Sigma)\to \op{Sp}(E,\Omega)$, and any $\mbf{J}\in \mc{J}(\mc{E},\Omega)$, it gives a $\rho$-equivariant map from $ \wt{\Sigma}$ into $\mc{J}(E,\Omega)$. Using the following identification
$$\mc{J}(E,\Omega)\cong \op{D}_n^{\op{III}},$$
see Remark \ref{rem5.2}, the $\mbf{J}\in \mc{J}(\mc{E},\Omega)$ defines a $\rho$-equivariant map by 
\begin{equation}\label{tJ}
  \wt{\mbf{J}}:\wt{\Sigma}\to \op{D}_n^{\op{III}},\quad z\mapsto  W=\wt{\mbf{J}}(z)=\mbf{J}_F^{-1}(\mbf{J}(z)).
\end{equation}
Here $\mbf{J}(z)\in \mc{J}(E,\Omega)\cong \mc{J}(\mb{R}^{2n},\Omega)$,  where the isomorphism  $(E,\Omega)\cong (\mb{R}^{2n},\Omega)$ is given by identified the symplectic basis of $(E,\Omega)$ with the standard basis of $(\mb{R}^{2n},\Omega)$. The  $\rho$-equivariant map $ \wt{\mbf{J}}:\wt{\Sigma}\to \op{D}_n^{\op{III}}$
 is also equivalent to the complex structure $\mbf{J}$ which is a smooth section of the associated bundle $\wt{\Sigma}\times_\rho \op{D}_n^{\op{III}}\to\Sigma$. There is a vector bundle 
$$F_\rho:=\wt{\Sigma}\times_\rho F=\wt{\Sigma}\times_\rho(\op{D}^{\op{III}}_n\times\mb{R}^{2n})$$ 
over $\wt{\Sigma}\times_\rho \op{D}_n^{\op{III}}$. By pullback, $\mbf{J}^*F_\rho$ is a vector bundle over $\Sigma$. 
The complex structure $\mbf{J}_F$ on $F$ defines a canonical complex structure on $F_{\rho}$ by 
$$\mbf{J}_{F_\rho}([z,W,e]):=[z,W,\mbf{J}_F(W)e],$$
where $[z,W,e]\in F_\rho$. Thus $(\mbf{J}^*F_\rho,\mbf{J}^*\mbf{J}_{F_\rho})$ is a complex vector bundle. One can check that the identification $(E,\Omega)\cong (\mb{R}^{2n},\Omega)$ induces a complex linear symplectic isomorphism between $(\mbf{J}^*F_\rho,\mbf{J}^*\mbf{J}_{F_\rho},\mbf{J}^*\Omega)$ and $(\mc{E},\mbf{J},\Omega)$, where  $\mbf{J}^*\Omega$ denotes the induced symplectic form on $F_\rho$. In fact, the pullback bundle $\mbf{J}^*F_\rho$ is the space of all elements $(z_0,[z,W,e])\in\Sigma\times (\wt{\Sigma}\times_\rho(\op{D}^{\op{III}}_n\times \mb{R}^{2n}))$ with $\mbf{J}(z_0)=[z,W]\in \wt{\Sigma}\times_\rho \op{D}_n^{\op{III}}$, the identification $(\mb{R}^{2n},\Omega)\cong (E,\Omega)$ gives rise to the desired isomorphism $\tau:(z_0,[z,W,e])\mapsto [z,e]$.

For the flat symplectic vector bundle $(F_\rho, \mbf{J}_{F_\rho},\Omega)$,  by complexification, 
we consider the complex vector bundle $(F_\rho)_{\mb{C}}$, it has the following decomposition
$$(F_\rho)_{\mb{C}}=(F_\rho)_{\mb{C}}^+\oplus (F_\rho)_{\mb{C}}^-=(\wt{\Sigma}\times_\rho F^+_{\mb{C}})\oplus (\wt{\Sigma}\times_\rho F^-_{\mb{C}}),$$
where $(F_\rho)_{\mb{C}}^\pm$ denotes the $\pm i$-eigenspaces of $\mbf{J}_{F_\rho}$. The Hermitian metric $(\cdot,\cdot)=2\Omega(\cdot,\mbf{J}_F\o{\cdot})$ on $F$ induces a well-defined Hermitian metric $(\cdot,\cdot)_\rho=\Omega(\cdot,\mbf{J}_{F_\rho}\o{\cdot})$ on $(F_\rho)_{\mb{C}}$. Similar to \eqref{Hermitian matrix}, with respect to the frame $\{[z,W,\wt{f}_j], 1\leq j\leq 2n\}$, the Hermitian matrix of the Hermitian metric is 
\begin{equation}
(\cdot,\cdot)_\rho=2\Omega(V\cdot,\mbf{J}_{F_\rho}\o{V\cdot})	=4\left(\begin{matrix}
  I_n-\o{W}W& 0 \\
 0 & I_n-W\o{W}
\end{matrix}
\right).
\end{equation}
On the other hand, there exists a connection $\n^{F_\rho}$ on $F_\rho$, in terms of the local frame $\{[z,W,\wt{f}_j], 1\leq j\leq 2n\}$, it is given by 
\begin{align*}
	\n^{F_\rho}=\left(\begin{matrix}
 d+(I_n-W\o{W})^{-1}\b{\p}(I_n-W\o{W}) & 0\\
 0&  d+(I_n-\o{W}W)^{-1}\p(I_n-\o{W}W)
\end{matrix}
\right),
\end{align*}
which is just the induced connection from $\n^F$. Moreover, $\n^{F_\rho}$ is a complex linear connection on $(F_\rho, \mbf{J}_{F_\rho})$, i.e. $[\n^{F_\rho}, \mbf{J}_{F_\rho}]=0$, and $\n^{F_\rho}$ preserves the Hermitian metric $(\cdot,\cdot)_\rho=2\Omega(\cdot,\mbf{J}_{F_\rho}\o{\cdot})$. Similar to \eqref{Chern forms 1}, the first Chern form of the connection is 
\begin{equation}\label{Chern forms 2}
  c_1((F_\rho)^+_{\mb{C}},\n^{F_\rho}|_{(F_\rho)^+_{\mb{C}}})=\frac{i}{2\pi}\op{Tr}_{(F_\rho)^+_{\mb{C}}}((\n^{F_\rho})^2)=\frac{1}{4\pi}\omega_{\op{D}^{\op{III}}_n}.	
\end{equation}

Since the two form $\omega_{\op{D}^{\op{III}}_n}$ is invariant under the action of $\op{Aut}(\op{D}^{\op{III}}_n)$, so it defines a well-defined two form on $\wt{\Sigma}\times_\rho \op{D}_n^{\op{III}}$, we  denote it also by $\omega_{\op{D}^{\op{III}}_n}$. Hence $\mbf{J}^*\omega_{\op{D}^{\op{III}}_n}$ is a two form on $\Sigma$. In fact, the pullback two form $\wt{\mbf{J}}^*\omega_{\op{D}^{\op{III}}_n}$ is a $\rho$-equivariant two form on $\wt{\Sigma}$, so it descends to a two form on $\Sigma$, which is just $\mbf{J}^*\omega_{\op{D}^{\op{III}}_n}$. Moreover, by \eqref{Chern forms 2}, one has 
\begin{equation}\label{Chern forms 3}
 \mbf{J}^* c_1((F_\rho)^+_{\mb{C}},\n^{F_\rho}|_{(F_\rho)^+_{\mb{C}}})=\frac{i}{2\pi}\op{Tr}_{(F_\rho)^+_{\mb{C}}}((\mbf{J}^*\n^{F_\rho})^2)=\frac{1}{4\pi}\mbf{J}^*\omega_{\op{D}^{\op{III}}_n}=\frac{1}{4\pi}\wt{\mbf{J}}^*\omega_{\op{D}^{\op{III}}_n}.	
\end{equation}

When restricted to a small collar neighborhood of $\p\Sigma$,  $\mc{E}\cong p^*(\mc{E}|_{\p\Sigma})$ near the boundary, where 
	$p:\p\Sigma\times [0,1]\to \p\Sigma$ denotes the natural projection. For any compatible complex structure on $\mc{E}|_{\p\Sigma}$, {the pullback of the complex structure by $p$ gives a compatible complex structure $\mbf{J}$ on $\mc{E}$ near the boundary}.
	Denote 
	\begin{align}\label{J0space}
	\begin{split}
	\mc{J}_o(\mc{E},\Omega)=\{\mbf{J}\in \mc{J}(\mc{E},\Omega)|  &\mbf{J}=p^*J
	\text{ on a small collar neighborhood}\\
	&\text{ of } \p\Sigma, \text{ where } J \in \mc{J}( \mc{E}|_{\p\Sigma},\Omega)\}.	
	\end{split}
	\end{align}
	 For any $\mbf{J}\in \mc{J}_o(\mc{E},\Omega)$, one has  $\mbf{J}=\mbf{J}(x)$ and 
 $\mbf{J}^*W=W(x)$ depends only on $x$ near $\partial\Sigma$.
\begin{prop}\label{proppullbackconnection}
	 For any $\mbf{J}\in \mc{J}_o(\mc{E},\Omega)$, the pullback connection $\mbf{J}^*\n^{F_\rho}$ is a peripheral connection on $\mbf{J}^*F_\rho$.
\end{prop}
\begin{proof}
	 The pullback complex structure $\mbf{J}$ depends only on $x$, so $\mbf{J}^*\n^{F_\rho}$ has the form $d+C(x)dx$. Since $\n^{F_{\rho}}$ is a complex linear connection on $(F_\rho,\mbf{J}_{F_\rho})$, so 
	$$[\mbf{J}^*\n^{F_\rho},\mbf{J}^*\mbf{J}_{F_\rho}]=\mbf{J}^*[\n^{F_\rho},\mbf{J}_{F_\rho}]=0.$$
	%By the definition of $\n^{F_\rho}$, the connection $\n^{F_\rho}$ preserves the Hermitian metric $(\cdot,\cdot)_\rho=2\Omega(\cdot,\mbf{J}_{F_\rho}\o{\cdot})$, so the pullback connection $\mbf{J}^*\n^{F_\rho}$ preserves the pullback Hermitian metric $2\mbf{J}^*\Omega(\cdot,\mbf{J}^*\mbf{J}_{F_\rho}\o{\cdot})$. 
	By the definition of $\n^F$, it preserves the standard symplectic form $\Omega$ on $\mb{R}^{2n}$, so the induced connection $\n^{F_\rho}$ also preserves the induced symplectic form $\Omega$ on $F_\rho$, hence the pullback connection $\mbf{J}^*\n^{F_\rho}$ preserves $\mbf{J}^*\Omega$. Thus the connection $\mbf{J}^*\n^{F_\rho}$ is a peripheral connection on $\mbf{J}^*F_\rho$.
\end{proof}
\begin{prop}\label{propchern}
For any $\mbf{J}\in \mc{J}_o(\mc{E},\Omega)$, one has
	\begin{align*}
		c_1(\mc{E},\Omega, \mbf{J})=\frac{1}{4\pi}[\wt{\mbf{J}}^*\omega_{\op{D}^{\op{III}}_n}]_c,
	\end{align*}
	where $[\bullet]_c$ denotes the de Rham class with compact support.
\end{prop}
\begin{proof}
Since  $\mbf{J}^*\n^{F_\rho}$ is a peripheral connection, by \eqref{Chern forms 3}, $\mbf{J}^*\omega_{\op{D}^{\op{III}}_n}$ has compact support.  On the other hand,  there is a natural complex linear symplectic isomorphism $\tau$ between the two vector bundles $(\mbf{J}^*F_\rho,\mbf{J}^*\mbf{J}_{F_\rho},\mbf{J}^*\Omega)$ and $(\mc{E},\mbf{J},\Omega)$, one can check that  $\tau\mbf{J}^*\n^{F_\rho}\tau^{-1}$ is a  peripheral connection on $\mc{E}$. From Remark \ref{rem411}, one has
\begin{align*}
	c_1(\mc{E},\Omega, \mbf{J})&=\left[\frac{i}{2\pi}\op{Tr}_{(F_\rho)^+_{\mb{C}}}((\tau\mbf{J}^*\n^{F_\rho}\tau^{-1})^2)\right]_c\\
	&=\left[\frac{i}{2\pi}\op{Tr}_{(F_\rho)^+_{\mb{C}}}((\mbf{J}^*\n^{F_\rho})^2)\right]_c\\
	&=\frac{1}{4\pi}[\mbf{J}^*\omega_{\op{D}^{\op{III}}_n}]_c=
	\frac{1}{4\pi}[\wt{\mbf{J}}^*\omega_{\op{D}^{\op{III}}_n}]_c,
\end{align*}
which completes the proof.
\end{proof}

\subsubsection{Invariant K\"ahler potential}\label{IKP}

%For any $L\in \op{Sp}(2n,\mb{R})$, it acts on $\op{D}^{\op{III}}_n$ holomorphically, and the action  extends continuously to $\o{\op{D}^{\op{III}}_n}$. By  Brouwer fixed point theorem, there exists a fixed point in $\o{\op{D}^{\op{III}}_n}$, see e.g. \cite[Lemma 3.1]{FF}. 
In this section, for any point $W_0\in \o{\op{D}^{\op{III}}_n}$, we will give an isotropy group $K_{W_0}$-invariant (up to a constant) K\"ahler potential $\psi_{W_0}$ of $\frac{1}{2}\omega_{{\op{D}^{\op{III}}_n}}$, i.e. $i\partial\bar\partial\psi_{W_0}=\frac{1}{2}\omega_{{\op{D}^{\op{III}}_n}}$ and $L^*\psi_{W_0}=\psi_{W_0}+c_{L}$ where $c_{L}$ is a constant depending only on $L$, and $L\in K_{W_0}:=\{L\in\op{Sp}(2n,\mb{R}): L(W_0)=W_0\}$.

 For any $L\in K_{W_0}$,  we assume that 
 \begin{align*}
 L=U\left(\begin{matrix}
a & b\\
\b{b} & \b{a}	
\end{matrix}
\right)U^{-1}\in \op{Sp}(2n,\mb{R})	
 \end{align*}
 for some complex $n\times n$ matrices $a,b$, see e.g. \cite[Page 71]{Mok}.
 Since $L_0\in K_{W_0}$, i.e. $L(W_0):=\lim_{W\to W_0}L(W)=W_0$, which follows that $aW_0+b=W_0(\b{b}W_0+\b{a})$ {by (\ref{Action})}, and is equivalent to
 \begin{equation}\label{Fixpoint0.5}
(a-W_0\b{b})W_0=W_0\b{a}-b.
\end{equation}
Now we define
\begin{equation}\label{invariant Kahler potential}
  \psi_{W_0}(W):=-\log\left(|\det(\o{W_0}W-I_n)|^{-2}\det(I_n-\o{W}W)\right)
\end{equation}
Then
\begin{align*}
&\quad\,\, |\det(\o{W_0}L(W)-I_n)|^{-2}\det(I_n-\o{L(W)}L(W)) \\
&=|\det((\o{W_0}a-\o{b})W+\o{W_0}b-\o{a})|^{-2}\det(I_n-\o{W}W)\\
&=|\det(\o{a}-\o{W_0}b)|^{-2}|\det(\o{W_0}W-I_n)|^{-2}\det(I_n-\o{W}W).
\end{align*}
which follows that 
\begin{align*}
L^*\psi_{W_0}(W)=\psi_{W_0}\circ L(W)=\psi_{W_0}(W)+2\log|\det(\o{a}-\o{W_0}b)|,
\end{align*}
that is, $\psi_{W_0}$ is $L$-invariant up to a constant. 
 Therefore
\begin{prop}\label{Kahler potential}
For any $W\in \o{\op{D}^{\op{III}}_n}$, there exists a $K_W$-invariant (up to a constant) K\"ahler potential $\psi_W$ with 
$i\p\b{\p}\psi_W=\frac{1}{2}\omega_{\op{D}^{\op{III}}_n}$.
\end{prop}

\begin{rem}\label{rempotential}
 Denote by 
\begin{align*}
\mb{H}_n=\{Z\in \mb{C}^{n\times n}|Z=Z^\top, \op{Im}Z>0\}	
\end{align*}
the Siegel upper half plane, and 
\begin{align*}
\Phi:\op{D}^{\op{III}}_n\to \mb{H}_n,\quad Z=\Phi(W)=i(I_n-W)(I_n+W)^{-1}	
\end{align*}
the identification between $\op{D}^{\op{III}}_n$ and $\mb{H}_n$. The induced action of $L\in \op{Sp}(2n,\mb{R})$ on $\mb{H}_n$ is the generalized M\"obius transformation. Moreover, 
$$\det\op{Im}Z=|\det(I_n+W)|^{-2}\det(I_n-\o{W}W).$$
 Suppose 
 \begin{align*}
 L=U\left(\begin{matrix}
a & b\\
\b{b} & \b{a}	
\end{matrix}
\right)U^{-1}=\left(\begin{matrix}
L_1 & L_2	\\
L_3 & L_4
\end{matrix}
\right)\in \op{Sp}(2n,\mb{R}),	
 \end{align*}
 where $L_3=\frac{i}{2}(a-\b{a})+\frac{i}{2}(\b{b}-b)$. Then  $W_0=-I_n$ is a fixed point of $L$  is equivalent to 
 $b-\o{b}=a-\o{a}$, and so $L_3=0$. Hence $L(-I_n)=-I_n$ if and only if $L$ has the following matrix form 
  \begin{align*}
 L=\left(\begin{matrix}
L_1 & L_2	\\
0 & L_4
\end{matrix}
\right)\in \op{Sp}(2n,\mb{R}).	
 \end{align*}
Thus  the K\"ahler potential can be given by
 \begin{equation}
  \psi=-\log\det\op{Im}Z.
\end{equation}
\end{rem}
\vspace{3mm}

Let $\Sigma$ be a surface with boundary $\p\Sigma=\cup_{i=1}^qc_i$. For each $i$, $L_i:=\rho(c_i)\in \op{Sp}(2n,\mb{R})$, by Proposition \ref{Kahler potential}, there exists a K\"ahler potential $\psi_i$ of $\frac{1}{2}\omega_{\op{D}^{\op{III}}_n}$, which is invariant under $L_i$ up to a constant. Set $d^c=-i(\p-\b{\p})$, which satisfies $dd^c=2i\p\b{\p}$. Then 
\begin{equation}\label{Fixpoint2}
  \alpha_i:=d^c\psi_i
\end{equation}
 is $L_i$-invariant one form and $d\alpha_i=dd^c\psi_i=\omega_{\op{D}^{\op{III}}_n}$.

Let $\chi_i(u)\in [0,1]$ be a smooth function on $\Sigma$ satisfies 
\begin{align*}
\chi_i(u)=
\begin{cases}
& 1,\quad u\in c_i\times [0,1/2];\\
&0, \quad u\in\Sigma\backslash (c_i\times [0,3/4]).	
\end{cases}
\end{align*}
For any $\mbf{J}\in \mc{J}(\mc{E},\Omega)$, which is equivalent to a $\rho$-equvariant map $\wt{\mbf{J}}:\wt{\Sigma}\to \op{D}^{\op{III}}_n$, then the two form 
\begin{align*}
\frac{1}{4\pi}\wt{\mbf{J}}^*	\omega_{\op{D}^{\op{III}}_n}-\frac{1}{4\pi}\sum_{i=1}^q d(\chi_i\wt{\mbf{J}}^*\alpha_i)
\end{align*}
is $\rho$-equivariant on $\wt{\Sigma}$, and so it descends to a two form on $\Sigma$. Moreover, it has compact support on $\Sigma$. 
By the same proof as in \cite[Proposition-definition 4.1]{KM}, the cohomology class $[\frac{1}{4\pi}\wt{\mbf{J}}^*	\omega_{\op{D}^{\op{III}}_n}-\frac{1}{4\pi}\sum_{i=1}^q d(\chi_i\wt{\mbf{J}}^*\alpha_i)]_c$ with compact support depends only on the conjugate class of the representation $\rho$ (independent of $\wt{\mbf{J}}$). Following \cite{KM}, we set 
 \begin{align*}
\left[\rho^*\omega_{\op{D}^{\op{III}}_n}\right]_c=\left[\wt{\mbf{J}}^*	\omega_{\op{D}^{\op{III}}_n}-\sum_{i=1}^q d(\chi_i\wt{\mbf{J}}^*\alpha_i)\right]_c.	
 \end{align*}
For any $\mbf{J}\in \mc{J}_o(\mc{E},\Omega)$,  both $\wt{\mbf{J}}^*	\omega_{\op{D}^{\op{III}}_n}$ and $d(\chi_i\wt{\mbf{J}}^*\alpha_i)$ have compact support on $\Sigma_o$, and by Proposition \ref{propchern}, one has 
\begin{align*}
c_1(\mc{E},\Omega, \mbf{J})=	\frac{1}{4\pi} \left[\rho^*\omega_{\op{D}^{\op{III}}_n}\right]_c+\left[\frac{1}{4\pi}\sum_{i=1}^q d(\chi_i\wt{\mbf{J}}^*\alpha_i)\right]_c.
\end{align*}
Hence 
$$\int_\Sigma c_1(\mc{E},\Omega, \mbf{J})=\frac{1}{4\pi}\int_\Sigma \left[\rho^*\omega_{\op{D}^{\op{III}}_n}\right]_c+\frac{1}{4\pi}\sum_{i=1}^q\int_{c_i}\wt{\mbf{J}}^*\alpha_i.$$
\begin{defn}[Rho invariant]\label{Rho invariant}
For a representation $\rho:\pi_1(\Sigma)\to \op{Sp}(E,\Omega)$, we define the Rho invariant by 
	\begin{equation}\label{Rho}
\boldsymbol{\rho}(\p\Sigma)=\frac{1}{\pi}\sum_{i=1}^q\int_{c_i}\wt{\mbf{J}}^*\alpha_i+\eta(A_{\mbf{J}}).
\end{equation}
\end{defn}
\begin{rem}\label{remRho}
From the formula \eqref{sign0} for signature, the Rho invariant $\bs{\rho}(\p\Sigma)$ is independent of the complex structure $\mbf{J}$, and just depends on the representations of the boundary $\p\Sigma$. The Rho invariant was firstly introduced by Atiyah, Patodi and Singer \cite[Theorem 2.4]{APSII} considered  the unitary representations on boundary. In particular, if the representation of boundary can be extended to a unitary representation of the whole manifold, then the signature of the representation can be expressed in terms of the signature of trivial representation and the Rho invariant. For the case of surfaces, the Rho invariant defined in \eqref{Rho} is a natural generalization  to the real symplectic group. 
\end{rem}

\subsubsection{Toledo invariant and pullback forms with compact supoort}

In this subsection, we will prove the second step, i.e.
\begin{equation}\label{Toledoexp}
  \op{T}(\Sigma,\rho)=\frac{1}{2\pi}\int_\Sigma \left[\rho^*\omega_{\op{D}^{\op{III}}_n}\right]_c=\frac{1}{2\pi}\int_\Sigma \wt{\mbf{J}}^*\omega_{\op{D}^{\op{III}}_n}-\frac{1}{2\pi}\sum_{i=1}^q\int_{c_i}\wt{\mbf{J}}^*\alpha_i.
\end{equation}

Firstly, we will recall some definitions on the cohomology group of a topological space with a group action, we refer to \cite{Kim} and the references therein. Let $X$ be a topological space and $G$ be a group acting continuously on $X$.  For any $k>0$, one can define the space
$$F^k_{\text{alt}}(X,\mb{R})=\{f:X^{k+1}\to \mb{R}| f \text{ is alternating}\}.$$
Let $F^k_{\text{alt}}(X,\mb{R})^G$ denote the subspace of $G$-invariant functions, where the action of $G$ on $F^k_{\text{alt}}(X,\mb{R})$ is given by 
$$(g\cdot f)(x_0,\ldots,x_k)=f(g^{-1}x_0,\ldots,g^{-1}x_k),$$
for any $f\in F^k_{\text{alt}}(X,\mb{R})$ and $g\in G$. The natural coboundary operator $\delta_k:F^k_{\text{alt}}(X,\mb{R})\to F^{k+1}_{\text{alt}}(X,\mb{R})$ is given by
$$(\delta_k f)(x_0,\ldots,x_{k+1})=\sum_{i=0}^{k+1}(-1)^if(x_0,\cdots,\hat{x}_i,\ldots,x_{k+1}),$$
which also gives a coboundary operator on the complex $F^*_{\text{alt}}(X,\mb{R})^G$. The cohomology $\mathrm{H}^*(X;G,\mb{R})$ is defined as the cohomology of this complex. Define $F^*_{\text{alt},b}(X,\mb{R})$ as the subspace of $F^*_{\text{alt}}(X,\mb{R})$ consisting of bounded alternating functions. The coboundary operator restricts to the complex $F^*_{\text{alt},b}(X,\mb{R})^G$ and so it defines a cohomology, denoted by $\mathrm{H}^*_b(X;G,\mb{R})$,
 see \cite{Dupre} and also \cite[Section 3]{Kim}. In particular, for a manifold $X$, $\mathrm{H}^*_b(\wt{X};\pi_1(X),\mb{R})\cong \mathrm{H}^*_b(\pi_1(X),\mb{R})$.
 
 Similarly, if $G$ is a semisimple Lie group and $X$ is the associated symmetric space, one can also define the complex for the continuous (resp. bounded) and alternating functions, we denote this complex by $C^*_c(X,\mb{R})_{\text{alt}}$ (resp. $C^*_{c,b}(X,\mb{R})_{\text{alt}}$). Then the continuous cohomology $\mathrm{H}^*_c(G,\mb{R})$ (resp.$\mathrm{H}^*_{c,b}(G,\mb{R})$) can be isomorphically computed by the cohomology of $G$-invariant complex $C^*_c(X,\mb{R})_{\text{alt}}^G$ (resp. $C^*_{c,b}(X,\mb{R})_{\text{alt}}^G$), see \cite[Chapitre III]{Gui} and \cite[Corollary 7.4.10]{Monod}. 

If $X$ is a countable CW-complex, then one can define the cohomology  groups $\mathrm{H}^*_b(X,\mb{R})$ and  $\mathrm{H}^*_b(X,A,\mb{R})$ associated to the complex $C^*_b(X,\mb{R})$ of  bounded real-valued cochains on $X$ and  the subcomplex 
$C^*_b(X,A,\mb{R})$ of the bounded cochains  that vanish on simplices with image contained in $A$, repsectively.
Let $C^k_b(\wt{X},\mb{R})_{\text{alt}}$ denote the complex of bounded, alternating real-valued Borel functions on $\wt{X}^{k+1}$, then the cohomology of $\pi_1(X)$-invariant complex  $C^*_b(\wt{X},\mb{R})^{\pi_1(X)}_{\text{alt}}$
is isomorphic to $\mathrm{H}^*_b(X,\mb{R})$, see \cite{Iva} and also \cite[Section 2]{Kim}.

Next, we will show $\op{T}(\Sigma,\rho)=\frac{1}{2\pi}\int_\Sigma \left[\rho^*\omega_{\op{D}^{\op{III}}_n}\right]_c$.
Let $\Sigma$ be a  connected oriented surface with boundary $\p\Sigma$, $\Sigma_o:=\Sigma\backslash \p\Sigma$. Consider a representation $\rho:\pi_1(\Sigma)\to G$ where $G=\op{Sp}(2n,\mb{R})$. Denote by $\mathscr{X}:=G/K$ the associated symmetric space, which can identified with the bounded symmetric domain $\op{D}^{\op{III}}_n$ of type $\op{III}$, we denote $$\omega:=\omega_{\op{D}^{\op{III}}_n}$$ for  simplicity.  The K\"ahler form $\omega$ gives the cohomology classes $\kappa_G\in \mathrm{H}^2_c(G,\mb{R})$ and $\kappa_G^b\in \mathrm{H}^2_{c,b}(G,\mb{R})$ which both correspond to the cochain $c_\omega$ defined by \eqref{bounded cocyle}.

The natural inclusion $C^*_{c,b}(\mathscr{X},\mb{R})_{\text{alt}}\subset F^*_{\text{alt},b}(\mathscr{X},\mb{R})$ induces a homomorphism $i_G:\mathrm{H}^*_{c,b}(G,\mb{R})\to \mathrm{H}^*_b(\mathscr{X};G,\mb{R})$. Then we have the following commutative diagram:
\begin{equation*}
\begin{CD}
\mathrm{H}^2_b(\mathscr{X};G,\mb{R}) @>f^*_b>> \mathrm{H}^2_b(\wt{\Sigma};\pi_1(\Sigma),\mb{R})\cong \mathrm{H}^2_b(\Sigma,\mb{R})\\
@AA i_GA @AA i_\Sigma A\\
\mathrm{H}^2_{c,b}(G,\mb{R}) @>\rho^*_b>> \mathrm{H}^2_b(\pi_1(\Sigma),\mb{R})
\end{CD}
\end{equation*}
see \cite[Page 58]{Kim}, where $f^*_b$ is induced from any $\rho$-equivariant map $f:\wt{\Sigma}\to\mathscr{X}$, $i_\Sigma$ is the Gromov isomorphism. 
 Then the cochain  representing the class $i_\Sigma\rho^*_b(\kappa_G^b)=f^*_bi_G(\kappa_G^b)$ is given by
\begin{equation}\label{strsimplex}
  \frac{1}{2\pi}\int_{\op{Str(\sigma)}}f^*\omega,
\end{equation}
where $\sigma\in C_2(\wt{\Sigma},\mb{R})$ is any two dimensional singular simplex on $\wt{\Sigma}$, $\op{Str}(\sigma)=\Delta(v_1,v_2,v_3)$ denotes the geodesic $2$-simplex, $v_1,v_2,v_3$ are the vertices of $\sigma$.
Denote $$[f^*\omega]_b:=2\pi i_\Sigma\rho^*_b(\kappa_G^b)=2\pi f^*_bi_G(\kappa_G^b).$$

For any $\rho$-equivariant map $f:\wt{\Sigma}\to \mr{D}^{\mr{III}}_{n}$, then the differential form 
$$f^*\omega-\sum_{i=1}^qd(\chi_i f^*\alpha_i)$$
descends to a well-defined form on $\Sigma$, and has compact support in $\Sigma_o$. Then it defines a relative bounded class 
$$\left[f^*\omega-\sum_{i=1}^qd(\chi_i f^*\alpha_i)\right]_{r,b}\in \mr{H}^2(\Sigma,\p\Sigma,\mb{R}).$$
Note that the natural inclusion $C^2_b(\Sigma,\p\Sigma,\mb{R})\to C^2_b(\Sigma,\mb{R})$ induces the following isomorphism
$$j_{\p\Sigma}:\mr{H}^2_b(\Sigma,\p\Sigma,\mb{R})\to \mr{H}^2_b(\Sigma,\mb{R}),$$
see \cite[\S 2.2, (2.e)]{BIW}. Hence 
\begin{equation*}
  j_{\p\Sigma}\left(\left[f^*\omega-\sum_{i=1}^qd(\chi_i f^*\alpha_i)\right]_{r,b}\right)=\left[f^*\omega-\sum_{i=1}^qd(\chi_i f^*\alpha_i)\right]_{b}\in \mr{H}^2_b(\Sigma,\mb{R}).
\end{equation*}
Since $\chi_i f^*\alpha_i$ is a one form on $\Sigma$, which can be viewed as a $\pi_1(\Sigma)$-invariant one form on $\wt{\Sigma}$, and  it defines a cochain in $\pi_1(\Sigma)$-invariant alternating complex $C^1(\wt{\Sigma},\mb{R})^{\pi_1(\Sigma)}_{\mr{alt}}$. More precisely, the cochain is given by 
\begin{align*}
\begin{split}
  (\chi_i f^*\alpha_i)(\sigma):=\int_{\mr{Str}(\sigma)}\chi_i f^*\alpha_i,
 \end{split}
\end{align*}
where $\sigma$ is any one dimensional singular simplex in $\wt{\Sigma}$, and $\mr{Str}(\sigma)=\Delta(v_1,v_2)$ denotes the geodesic $1$-simplex, $v_1$ and $v_2$ are two vertices of $\sigma$. The $\ell^\infty$-norm for cochain is defined by 
$$\|\bullet\|_\infty:=\sup_{\sigma\in C_1(\wt{\Sigma},\mb{R})}|\bullet(\sigma)|.$$
Hence
\begin{equation}\label{bound1}
  \left\|\sum_{i=1}^q\chi_i f^*\alpha_i\right\|_\infty\leq \sum_{i=1}^q   \left\|\chi_i f^*\alpha_i\right\|_\infty\leq \sum_{i=1}^q\left\|f^*\alpha_i\right\|_\infty\leq \sum_{i=1}^q\|\alpha_i\|_{\infty}.
\end{equation}
 \begin{prop}\label{potential}For $\alpha=d^c\psi_W$, where $d^c:=-i(\p-\b{\p})$, then $||\alpha||_\infty\leq n\pi$.
 \end{prop}
 \begin{proof} For any $W\in \o{\op{D}^{\op{III}}_n}$, let $\psi_W$ be any $K_W$-invariant (up to a constant) K\"ahler potential. Let $\gamma(t)$ be a geodesic with $\gamma(0)=Q\in \op{D}^{\op{III}}_n$ and $\gamma(\infty)=W$. Then $\gamma$ is contained in a maximal flat $F$ which is a totally real subspace. The following argument is basically due to Domic-Toledo \cite{DT}. The difficult case is when $W$ is an ideal point. If $W$ is regular,  $K_W$ is $MAN$ where $N$ is a minimal parabolic group. If $J$ denotes a complex structure, then $J\gamma'(t)$ is tangent to the orbit $N\gamma(t)$ since $J\gamma'(t)$ is orthogonal to the geodesic and $N\gamma(t)$ contains all the directions orthogonal to $F$.
 
 If $W$ is singular,  $K_W$ is $M'A'N'$ where $N'$ contains $N$ \cite{Eb}. Hence $J\gamma'(t)$ is tangent to the orbit $N'\gamma(t)$.
 
 Then $$d^c\psi_W(\gamma'(t))=d\psi_W(J\gamma'(t))=0$$ since $\psi_W$ is $K_W$-invariant. This shows that $\alpha=d^c\psi_W$ is zero along $\gamma$.
 
 Hence, for any $Q,R\in \op{D}^{\op{III}}_n$ with geodesic $\gamma(Q,R)$ connecting them,
 $$\int_{\gamma(Q,R)}\alpha=\int_{\gamma(Q,R)} d^c\psi_W=\int_{\triangle(W,Q,R)} dd^c\psi_W=\int_{\triangle(W,Q,R)} \omega_{\op{D}^{\op{III}}_n}.$$
By \cite[Theorem 1]{DT}, one has
 $$\left|\int_{\gamma(Q,R)} \alpha\right|=\left|\int_{\triangle(W,Q,R)} \omega_{\op{D}^{\op{III}}_n}\right|\leq \sup_{\triangle}\left|\int_\triangle \omega_{\op{D}^{\op{III}}_n}\right|\leq n\pi.$$
 This shows that $||\alpha||_\infty\leq n\pi$.
 \end{proof}

 From the above proposition and \eqref{bound1}, one has
\begin{align*}
\begin{split}
  \left\|\sum_{i=1}^q\chi_i f^*\alpha_i\right\|_\infty&\leq \sum_{i=1}^{q}\|\alpha_i\|_\infty\leq q\cdot n\pi,
 \end{split}
\end{align*}
which means that $\sum_{i=1}^q\chi_i f^*\alpha_i$ defines a bounded cochain in $C^1_b(\wt{\Sigma},\mb{R})^{\pi_1(\Sigma)}_{\op{alt}}$, and so 
$$\left[\sum_{i=1}^qd(\chi_i f^*\alpha_i)\right]_b=0\in \mr{H}^1_b(\Sigma,\mb{R}).$$
Hence
\begin{equation*}
  j_{\p\Sigma}\left(\left[f^*\omega-\sum_{i=1}^qd(\chi_i f^*\alpha_i)\right]_{r,b}\right)=\left[f^*\omega-\sum_{i=1}^qd(\chi_i f^*\alpha_i)\right]_{b}=\left[f^*\omega\right]_b.
\end{equation*}

There exist the following natural maps 
\begin{equation*}
	\mathrm{H}^2_b(\Sigma,\p\Sigma,\mb{R})\stackrel{c}{\longrightarrow}\mathrm{H}^2(\Sigma,\p\Sigma,\mb{R})\stackrel{j_o^{-1}}{\longrightarrow}\mathrm{H}^2_{\text{comp}}(\Sigma_o,\mb{R})\stackrel{D}{\longrightarrow}
\mathrm{H}^2_{\text{dR,comp}}(\Sigma_o,\mb{R})
\end{equation*}
where $c$ is the canonical map induced from the inclusion $C^2_b(\Sigma,\p\Sigma,\mb{R})\subset C^2_b(\Sigma,\mb{R})$, $j_o$ is the natural map from the singular cohomology with compact support to the relative  singular cohomology, which is an isomorphism.  Recall that the singular  cohomology with compact support
is the complex  $C^*_{\text{comp}}(\Sigma,\mb{R})$ of all cochains have compact support, 
 a cochain $u \in C^{*}(\Sigma_o, \mb{R})$ has compact support if and only if there exists a compact set $K \subset \Sigma_o$ such that $u \in C^{*}(\Sigma_o, \Sigma_o-K, \mb{R})$,
see \cite[Chapter IX,\S 3]{WM}.  The map $D$ is the de Rham map between the singular cohomology with compact support and de Rham cohomology with compact support, which is also an isomorphism, see e.g. \cite[Appendix, Page 261]{WM}. Under these canonical maps and note that $f^*\omega-\sum_{i=1}^nd(\chi_i f^*\alpha_i)$ is also a differential form on $\Sigma_o$ with compact support, so
$$Dj_o^{-1}c\circ j^{-1}_{\p\Sigma}([f^*\omega]_b)=\left[f^*\omega-\sum_{i=1}^qd(\chi_i f^*\alpha_i)\right]_{c}\in \mathrm{H}^2_{\text{dR,comp}}(\Sigma_o,\mb{R}).$$
Hence the Toledo invariant can be given by
\begin{align*}
\op{T}(\Sigma,\rho)&=\langle j^{-1}_{\p\Sigma}i_\Sigma\rho^*_b(\kappa_G^b),[\Sigma,\p\Sigma]\rangle\\
&=\langle cj^{-1}_{\p\Sigma}i_\Sigma\rho^*_b(\kappa_G^b),[\Sigma,\p\Sigma]\rangle\\
&=\frac{1}{2\pi}\int_\Sigma\left(	f^*\omega-\sum_{i=1}^qd(\chi_i f^*\alpha_i)\right)\\
&=\frac{1}{2\pi}\int_\Sigma \left[\rho^*\omega\right]_c.
\end{align*}
\begin{rem}
When $\p\Sigma\neq\emptyset$ and $G=\op{PU}(1,m)$, Koziarz and Maubon \cite[Proposition-Definition 4.1]{KM} introduced the invariant $\frac{1}{2\pi}\int_\Sigma \left[\rho^*\omega\right]_c$ by using the de Rham cohomology with compact support. In fact, the invariant can be shown to be equal to the Toledo invariant defined by  Burger-Iozzi-Wienhard, see \cite[Remark 6]{BIW}.

\end{rem}

\section{Milnor-Wood  inequality}\label{AMW}

In this section, by using the formula for signature, we will prove a Milnor-Wood type inequality for signature and obtain some modified Milnor-Wood inequalities for Toledo invariant.  

\subsection{A Milnor-Wood type inequality for signature}
Let $\rho:\pi_1(\Sigma)\to \op{Sp}(E,\Omega)$ be a representation, we have 
$$\op{sign}(\mc{E},\Omega)=2\op{T}(\Sigma,\rho)+\frac{1}{\pi}\sum_{i=1}^q\int_{c_i}\wt{\mbf{J}}^*\alpha_i+\eta(A_{\mbf{J}})=2\op{T}(\Sigma,\rho)+\bs{\rho}(\p\Sigma).$$
\begin{lemma}
	The indices of $d^-_P$ and $d^+_P$ can be given by
	\begin{equation*}
\op{Index}(d^{\mp}_P)= \pm\frac{1}{2}\op{sign}(\mc{E},\Omega)+\frac{\dim E}{2}\chi(\Sigma)-\frac{\dim \mathrm{H}^0(\p\Sigma,\mc{E})}{2}.
\end{equation*}
\end{lemma}
\begin{proof}
For the index of $d^-_P$, by \eqref{APSformula}, one has
\begin{align*}
\op{Index}(d^{-}_P)&=\int_{\Sigma}\alpha_{{-}}(z)d\mu_g-\frac{\eta(A^-_{\mbf{J}})+\dim\op{Ker}A^{-}_{\mbf{J}}}{2}\\
&=\int_{\Sigma}\alpha_{{-}}(z)d\mu_g+\frac{\eta(A_{\mbf{J}})}{2}-\frac{\dim \mathrm{H}^0(\p\Sigma,\mc{E})}{2}\\
&=\frac{\dim E}{2}\chi(\Sigma)+2\int_{\Sigma}c_1(\mc{E},\Omega, \mbf{J})+\frac{\eta(A_{\mbf{J}})}{2}-\frac{\dim \mathrm{H}^0(\p\Sigma,\mc{E})}{2}\\
&=\frac{\dim E}{2}\chi(\Sigma)+\frac{1}{2}\op{sign}(\mc{E},\Omega)-\frac{\dim \mathrm{H}^0(\p\Sigma,\mc{E})}{2}
\end{align*}
where the second equality by $\eta(A_{\mbf{J}}^-)=-\eta(A_{\mbf{J}})$ and \eqref{KerAJ}, the third equality by Proposition \ref{prop410}, and the last equality by Theorem \ref{thmsign2}. Similarly,  we can also obtain
\begin{equation*}
  \op{Index}(d^{+}_P)=\frac{\dim E}{2}\chi(\Sigma)-\frac{1}{2}\op{sign}(\mc{E},\Omega)-\frac{\dim \mathrm{H}^0(\p\Sigma,\mc{E})}{2}.
\end{equation*}
\end{proof}
From the above lemma and using $\mathrm{L}^2\op{Index}(d^{\pm})=\op{Index}(d^{\pm}_P)+h_{\infty}(\wedge^\pm)$, one has
\begin{align*}
\begin{split}
\pm\frac{1}{2}\op{sign}(\mc{E},\Omega)&=-\frac{\dim E}{2}\chi(\Sigma)+\frac{\dim \mathrm{H}^0(\p\Sigma,\mc{E})}{2}+\op{Index}(d^{\mp}_P)\\
&=-\frac{\dim E}{2}\chi(\Sigma)+\frac{\dim \mathrm{H}^0(\p\Sigma,\mc{E})}{2}-h_{\infty}(\wedge^{\mp}) +\mathrm{L}^2\op{Index}(d^{\mp}).
\end{split}
\end{align*}
Here the $\mathrm{L}^2$-index is given by 
\begin{align*}
\begin{split}
\mathrm{L}^2\op{Index}(d^{\pm})&=\dim_{\mb{C}}\op{Ker}(d^\pm)\cap \mathrm{L}^2(\widehat{\Sigma},\mc{E}_{\mb{C}})-\dim_{\mb{C}}\op{Ker}(d^\pm)^*\cap \mathrm{L}^2(\widehat{\Sigma},\wedge^\pm)\\
&=\dim_{\mb{C}}\mathscr{H}^0(\widehat{\Sigma},\mc{E}_{\mb{C}})-\dim_{\mb{C}}\op{Ker}(d^\pm)^*\cap \mathrm{L}^2(\widehat{\Sigma},\wedge^\pm)\\
&=\dim\op{Im}(\mathrm{H}^0(\Sigma,\p\Sigma,\mc{E})\to \mathrm{H}^0(\Sigma,\mc{E}))-\dim_{\mb{C}}\op{Ker}(d^\pm)^*\cap \mathrm{L}^2(\widehat{\Sigma},\wedge^\pm),
\end{split}
\end{align*}
where the second equality follows from (\ref{H0}) and the last equality follows from \cite[Proposition 4.9]{APS}. Since
$\mathrm{H}^0(\Sigma,\p\Sigma,\mc{E})=\{0\}$, so  $$\mathrm{L}^2\op{Index}(d^{\pm})= -\dim_{\mb{C}}\op{Ker}(d^\pm)^*\cap \mathrm{L}^2(\widehat{\Sigma},\wedge^\pm).$$
On the other hand, by \eqref{hpm}, one has
$$\dim \mathrm{H}^0(\p\Sigma,\mc{E})-h_{\infty}(\wedge^{\pm})=\dim \mathrm{H}^0(\Sigma,\mc{E}).$$
Hence
\begin{multline}\label{5.6}
 \pm\frac{1}{2}\op{sign}(\mc{E},\Omega)	=-\frac{\dim E}{2}\chi(\Sigma)-\frac{\dim \mathrm{H}^0(\p\Sigma,\mc{E})}{2}\\
+\dim \mathrm{H}^0(\Sigma,\mc{E})-\dim_{\mb{C}}\op{Ker}(d^\mp)^*\cap \mathrm{L}^2(\widehat{\Sigma},\wedge^\pm).
\end{multline}
\begin{rem}
From the above equality, one gets immediately 
\begin{multline}\label{relative dimension}
	\dim\op{Im}(\mathrm{H}^1(\Sigma,\p \Sigma,\mc{E})\to \mathrm{H}^1(\Sigma,\mc{E}))\\=-\dim E\cdot\chi(\Sigma)-\dim \mathrm{H}^0(\p\Sigma,\mc{E})+2\dim \mathrm{H}^0(\Sigma,\mc{E}).
\end{multline}
In fact, by Proposition \ref{prop3.2}, one has
$$\op{sign}(\mc{E},\Omega)=\dim_{\mb{C}}\op{Ker}(d^+)^*\cap \mathrm{L}^2(\widehat{\Sigma},\wedge^+)-\dim_{\mb{C}}\op{Ker}(d^-)^*\cap \mathrm{L}^2(\widehat{\Sigma},\wedge^-).$$ 
 On the other hand, by Proposition \ref{prop3} and \cite[Proposition 4.9]{APS}, one has
\begin{multline*}
	\dim\op{Im}(\mathrm{H}^1(\Sigma,\p \Sigma,\mc{E})\to \mathrm{H}^1(\Sigma,\mc{E}))\\=\dim_{\mb{C}}\op{Ker}(d^-)^*\cap \mathrm{L}^2(\widehat{\Sigma},\wedge^-)+\dim_{\mb{C}}\op{Ker}(d^+)^*\cap \mathrm{L}^2(\widehat{\Sigma},\wedge^+).
\end{multline*}	
Hence we obtain \eqref{relative dimension} by combining with \eqref{5.6}.
\end{rem}

\begin{rem}
 If $\p\Sigma=\emptyset$, then $\frac{1}{2}\op{sign}(\mc{E},\Omega)=\op{T}(\Sigma,\rho)$ and we have 
\begin{align*}
	\pm\op{T}(\Sigma,\rho)&=-\frac{\dim E}{2}\cdot\chi(\Sigma)+\op{Index}(d^{\mp})\\
	&=-\frac{\dim E}{2}\cdot\chi(\Sigma)+\dim \mathrm{H}^0(\Sigma,\mc{E})-\dim\op{Ker}(d^\mp)^*.
	\end{align*}
	The second equality follows from the observation: if $d^-a=0$, then $d a=d^+a$, and so 
$d a=*\mbf{J}d a$, which follows that $(d^+)^*d^+a=d^*d a=0$. A similar argument holds for $d^+$.
	\end{rem}
	
\begin{lemma}\label{lemma3}
If $\Sigma$ is a  surface with  genus $g\geq 1$, then
the set of all representations $\rho$ such that $\{v\in E: \rho(a_i)v=v=\rho(b_i)v,1\leq i\leq g\}=\emptyset$ is dense in $\op{Hom}(\pi_1(\Sigma),\op{Sp}(E,\Omega))$, where $a_i,b_i\in \pi_1(\Sigma),1 \leq i\leq g$ denote the  generators in the interior of $\Sigma$.	
\end{lemma}
\begin{proof}
Suppose that the boundary of $\Sigma$ is the union $\partial \Sigma=\bigsqcup_{j=1}^q c_j$ of oriented circles, the fundamental group of $\Sigma$ is 
$$
\pi_{1}(\Sigma)=\left\langle a_{1}, b_{1}, \ldots, a_{g}, b_{g}, c_{1}, \ldots, c_{q}: \prod_{i=1}^{g}\left[a_{i}, b_{i}\right] \prod_{j=1}^{q} c_{j}=e\right\rangle.
$$
Then  $\op{Hom}(\pi_1(\Sigma),\op{Sp}(E,\Omega))$ is the space of all homomorphisms  with the relation 
\begin{align}\label{relation}
\prod_{i=1}^{g}\left[\rho(a_{i}), \rho(b_{i})\right] \prod_{j=1}^{q} \rho(c_{j})=\rho(e)=\op{Id}.	
\end{align}
If $1$ is an eigenvalue of $\rho(a_1)\in \op{Sp}(E,\Omega)$ and $\rho(b_1)\in \op{Sp}(E,\Omega)$, then by a small perturbation, we can find two families of  operators  $A_{\epsilon}, B_{\epsilon}\in \op{Sp}(E,\Omega)$  such that 	$1$ is not the an eigenvalue for $\epsilon>0$ and
	$$\lim_{\epsilon\to 0}A_{\epsilon}=\rho(a_1),\quad \lim_{\epsilon\to 0}B_{\epsilon}=\rho(b_1),\quad [A_{\epsilon},B_{\epsilon}]=[\rho(a_1),\rho(b_1)],$$
	for $\epsilon>0$ small enough. In fact, the linear transformations $\sigma\in \op{Sp}(E,\Omega)$ and $\tau\in \op{Sp}(E,\Omega)$ with $1$ as an eigenvalue is equivalent to $\det(\sigma-\op{Id})=0$ and $\det(\tau-\op{Id})=0$, which defines a variety $H_1$ in $\op{Sp}(E,\Omega)\times \op{Sp}(E,\Omega)$, and $[\sigma,\tau]=[\rho(a_1),\rho(b_1)]$ is a variety $H_2$ in $\op{Sp}(E,\Omega)\times \op{Sp}(E,\Omega)$. One can see that $H_1\cap H_2$ is a proper subvariety of $H_2$ since the polynomial $[\sigma,\tau]=[\rho(a_1),\rho(b_1)]$ is not a factor of the polynomials $\det(\tau-\op{Id})$ and $\det(\sigma-\op{Id})$. Thus we can find such $A_{\epsilon},B_{\epsilon}$ in $H_2\backslash (H_1\cap H_2)$.	
	
	Now we can take $\rho_{\epsilon}$ by 
	$$\rho_{\epsilon}(a_1)=A_{\epsilon},\quad \rho_{\epsilon}(b_1)=B_{\epsilon},\quad \rho_{\epsilon}=\rho,\,\text{when acting on other cycles.}$$ 
	Then $\rho_{\epsilon}\in \op{Hom}(\pi_1(\Sigma),\op{Sp}(E,\Omega))$ and has no global fixed point, and 
	$\lim_{\epsilon\to 0}\rho_{\epsilon}=\rho.$
	Especially, $\{\rho_{\epsilon}\}$ is invariant when acting on boundary, i.e. $\rho_{\epsilon}(c_i)=\rho(c_i)$ for any $\epsilon>0$.
\end{proof}
For any $s\in \mathrm{H}^0(\Sigma,\mc{E})$, it can be viewed as a $\rho$-equivariant map $s=s^i(x)e_i\in E$, and 
$0=ds=ds^i(x)e_i,$
so $s^i(x)=s^i$ is constant, and $s=s^ie_i$ is a constant vector. The $\rho$-equivariant condition is 
$s^i(\gamma x)e_i=\rho(\gamma)^{-1}s^i(x)e_i,$
which follows that 
$s=\rho(\gamma)^{-1}s.$
Thus we obtain
$$\mathrm{H}^0(\Sigma,\mc{E})\cong \{s\in E:s=\rho(\gamma)s,\forall \gamma\in \pi_1(\Sigma)\}.$$	
\begin{thm}\label{thmsign}
The signature satisfies the following Milnor-Wood type inequality:
$$|\op{sign}(\mc{E},\Omega)|\leq\dim E \cdot |\chi(\Sigma)|. $$	
\end{thm}
\begin{proof}
For $q\geq 2$, since 
\begin{equation}\label{5.7}
  \dim \mathrm{H}^0(\p\Sigma,\mc{E})\geq q\dim \mathrm{H}^0(\Sigma,\mc{E})\geq 2\dim \mathrm{H}^0(\p\Sigma,\mc{E}),
\end{equation}
 so by (\ref{5.6}),
$\pm\op{sign}(\mc{E},\Omega)\leq \dim E|\chi(\Sigma)|.$

For $q\leq 1$, by the above lemma, if the genus $g\geq 1$ , for any representation $\rho$, there exists a family representations $\rho_\epsilon$ with $\rho_\epsilon(\pi_1(\p\Sigma))=\rho(\pi_1(\p\Sigma))$, each $\rho_\epsilon$ has no global fixed point, and $\lim_{\epsilon\to 0}\rho_\epsilon=\rho$. Then $\dim \mathrm{H}^0(\Sigma,\mc{E}_\epsilon)=0$ for any $\epsilon>0$. 
Since the representations $\rho_\epsilon$ are fixed on the boundary, so the eta invariant  and $\int_{c_i}\wt{\mbf{J}}^*\alpha_i$ are  fixed, hence the Rho invariant is  fixed, $\bs{\rho}_{\epsilon}(\p\Sigma)=\bs{\rho}(\p\Sigma)$.
By \cite[Corollary 8.11]{BIW}, the Toledo invariant is also fixed, i.e.  $\op{T}(\Sigma,\rho_\epsilon)=\op{T}(\Sigma,\rho)$. Hence  the signature associated to $\rho_\epsilon$  is independent of $\epsilon$, and equals $\op{sign}(\mc{E},\Omega)$. Thus
\begin{align}\label{siginequality}
\begin{split}
\pm\frac{1}{2}\op{sign}(\mc{E},\Omega) &=
\pm\op{T}(\Sigma,\rho_\epsilon)\pm\frac{\bs{\rho}_{\epsilon}(\p\Sigma)}{2}\\
& =	-\frac{\dim E}{2}\chi(\Sigma)-\frac{\dim \mathrm{H}^0(\p\Sigma,\mc{E})}{2} -\dim_{\mb{C}}\op{Ker}(d^\mp)^*\cap \mathrm{L}^2(\widehat{\Sigma},\wedge^\mp)\\
&\leq 	-\frac{\dim E}{2}\chi(\Sigma), 
\end{split}
\end{align}
which follows that $|\op{sign}(\mc{E},\Omega)|\leq\dim E \cdot |\chi(\Sigma)|$.  If $g=0$ and $q\leq 1$, then $\pi_1(\Sigma)$ is trivial, and so $\op{sign}(\mc{E},\Omega)=0$.

Hence, we obtain 
\begin{equation}\label{sign4}
  |\op{sign}(\mc{E},\Omega)|\leq \dim E\cdot\max\{-\chi(\Sigma),0\}\leq \dim E\cdot |\chi(\Sigma)|.
\end{equation}
The proof is complete.
\end{proof}

If $\p\Sigma=\emptyset$, then $\op{sign}(\mc{E},\Omega)=2\op{T}(\Sigma,\rho)$. So
\begin{cor}[{Turaev \cite{Tur}}]
If $\rho\in \op{Hom}(\pi_1(\Sigma),\op{Sp}(E,\Omega))$ and $\p\Sigma=\emptyset$, then 
$$|\op{T}(\Sigma,\rho)|\leq \frac{\dim E}{2}|\chi(\Sigma)|.$$	
\end{cor}
\begin{rem}
	For the case $\p\Sigma\neq\emptyset$ and $G$ is a group of Hermitian type, it was proved by Burger, Iozzi and  Wienhard \cite[Theorem 1 (1)]{BIW} that 
	\begin{equation*}
  |\op{T}(\Sigma,\rho)|\leq \op{rank}(G/K)|\chi(\Sigma)|.
\end{equation*}
The above inequality was also generalized to the higher dimensional case by using the isometric isomorphism of $j_{\p\Sigma}$, see \cite[Theorem 1, Corollary 2]{BBFIPP} or \cite[Theorem 1.2]{KK1}. More precisely, 
\begin{align*}
\begin{split}
  |\op{T}(\Sigma,\rho)|&=|\langle j^{-1}_{\p\Sigma}i_\Sigma\rho^*_b(\kappa_G^b),[\Sigma,\p\Sigma]\rangle|\\
  &\leq \|j^{-1}_{\p\Sigma}i_\Sigma\rho^*_b(\kappa_G^b)\|\cdot\|[\Sigma,\p\Sigma]\|_1\\
  &=\|\rho^*_b(\kappa_G^b)\|\cdot\|[\Sigma,\p\Sigma]\|_1\\
  &\leq \|\kappa_G^b\|\cdot\|[\Sigma,\p\Sigma]\|_1\\
  &=\op{rank}(G/K)\cdot|\chi(\Sigma)|,
 \end{split}
\end{align*}
 where the third equality since $j_{\p\Sigma}$ is an isometric isomorphism and the Gromov isomorphism $i_\Sigma$ is also isometric, the fourth inequality  by the fact that the pullback is norm decreasing, and the last equality since  $\|\kappa_G^b\|=\frac{\op{rank}(G/K)}{2}$ and $\|[\Sigma,\p\Sigma]\|_1=2|\chi(\Sigma)|$.
\end{rem}

\subsection{Modified Milnor-Wood inequalities}\label{Class}

In this subsection, we will give  definitions of hyperbolic, elliptic and parabolic elements in $\op{Sp}(2n,\mb{R})$, and obtain some modified inequalities   for Toledo invariant. 

\begin{defn}
\label{defhyperell}
Let $(E,\Omega)$ be a real symplectic vector space. We define that $L\in \op{Sp}(E,\Omega)$ is \emph{hyperbolic} if all the eigenvalues $\lambda$ are $|\lambda|\neq 1$.
% it commutes with an involution $I$ such that $I^*\Omega=-\Omega$ and there exists an almost complex structure $J$ compatible with $\Omega$ such that $IJI=-J$.
 We define that $L$ is \emph{elliptic} if it is semisimple and all the eigenvalues $\lambda$ are $|\lambda|=1$ and $\lambda\neq \pm 1$. If all the eigenvalues of $L$ are $\pm 1$, $L$ is called \emph{parabolic}.
 %commutes with an almost complex structure $J$ compatible with $\Omega$. 
\end{defn}

\begin{prop}
\label{dechyperellu}
Let $(E,\Omega)$ be a symplectic vector space and let $L\in \op{Sp}(E,\Omega)$. There exists a unique $\Omega$-orthogonal decomposition $E=E_h\oplus E_{eu}\oplus E_{u^+}\oplus E_{u^-}$ (called  hyperbolic/elliptic-unipotent/unipotent decomposition) such that $L_{|E_h}$ is hyperbolic, $L_{|E_{eu}}$ is a product of elliptic $L_{e}$  and unipotent part $L_{u}$, $L_{|E_{u^+}}$ and $-L_{|E_{u^-}}$ are unipotent.   Furthermore,  $L_{|E_h}$ commutes with an involution $I$ such that $I^*\Omega=-\Omega$ and there exists an almost complex structure $J$ compatible with $\Omega$ such that $IJI=-J$. The elliptic part $L_{e}$  commutes with an almost complex structure $J$ compatible with $\Omega$ and it fixes a point in the symmetric space $\mathcal{J}_\lambda$ of almost complex structures on $E_{eu}$ which are compatible with  the restriction of $\Omega$ to $E_{eu}$.
\end{prop}

\begin{proof}
The complexified vector space $E\otimes\bc$ splits into characteristic subspaces of $L$, $E\otimes\bc=\bigoplus_\lambda E_\lambda$, $E_\lambda=\mathrm{ker}(L-\lambda I_m)^N$ for $N\geq\mathrm{dim}(E)=m$. If $\lambda,\mu$ are eigenvalues and $\lambda\mu\not=1$, then $E_\lambda$ and $E_\mu$ are $\Omega$-orthogonal. Indeed, given $v\in E_\lambda$, the normalized powers $\lambda^{-k}L^k(v)=:P_v(k)$ depend polynomially on $k$. Therefore, if $w\in E_\mu$,
$$
(\lambda\mu)^{-k}\Omega(L^k(v),L^k(w))=\Omega(P_v(k),P_w(k)):=Q(k)
$$
is a scalar polynomial in $k$. Since $\Omega$ is $L$-invariant, the function $k\mapsto (\lambda\mu)^{k}Q(k)$ is constant. If $\lambda\mu\not=1$, this can happen only if $Q$ vanishes identically. In particular, $Q(0)=\Omega(v,w)=0$.

Therefore the subspaces $E_{\lambda,\br}, E_{\mu,\br}$ where
$$
E_{\lambda,\br}:=E\cap(E_\lambda +E_{\bar\lambda}+E_{1/\lambda} +E_{1/\bar\lambda})
$$
are mutually orthogonal when distinct.

Pick $\lambda$ such that $|\lambda|\not=1$. Then $F_+=E\cap(E_\lambda +E_{\bar\lambda})$ is isotropic since $\lambda^2\neq 1$. Since the restriction of $\Omega$ to $E_{\lambda,\br}$ is nondegenerate, its $\Omega$-orthogonal in $E_{\lambda,\br}$ is itself, thus $F_+$ is Lagrangian in $E_{\lambda,\br}$. So is $F_-=E\cap(E_{1/\lambda} +E_{1/\bar\lambda})$, and $E_{\lambda,\br}=F_+\oplus F_-$. Define $I_\lambda\in \op{End}(E_{\lambda,\br})$ by $I_\lambda=\pm1$ on $F_{\pm}$. Then $I_\lambda^2=1$ and $I_\lambda^*\Omega=-\Omega$. 
Since $$\Omega(v, -I_\lambda J I_\lambda v)=I_\lambda^*\Omega(I_\lambda v, -JI_\lambda v)=-\Omega(I_\lambda v, -JI_\lambda v)=\Omega(I_\lambda v, J I_\lambda v),$$
the formula $J\mapsto -I_\lambda J I_\lambda$ defines an isometric action on the space $\mathcal{J}_\lambda$ of almost complex structures on $E_{\lambda,\br}$ which are compatible with the restriction of $\Omega$ to $E_{\lambda,\br}$. Since $\mathcal{J}_\lambda$ is a symmetric space of noncompact type, and $I_\lambda^2=1$, the action has a fixed point $J_\lambda$, i.e. $I_\lambda J_\lambda I_\lambda=-J_\lambda$. Let
$$
E_h=\bigoplus_{|\lambda|\not=1} E_{\lambda,\br}, \quad I_h=\bigoplus_{|\lambda|\not=1} I_{\lambda},\quad J_h=\bigoplus_{|\lambda|\not=1} J_{\lambda}.
$$
Then $J_h$ is compatible with $\Omega_{|E_h}$, $I_h$ is an involution, it commutes with $L_{|E_h}$ and maps $\Omega$ and $J_h$ to their opposites. Thus $L_{|E_h}$ is hyperbolic with desired properties.

Let $\lambda$ be an eigenvalue of $L$ such that $|\lambda|=1$, $\lambda\not=\pm 1$. Then
$$
E_{\lambda,\br}=E\cap(E_\lambda \oplus E_{\bar\lambda} ).
$$
Let $L_{|E_{\lambda,\br}}=SU$ be the Jordan decomposition of $L_{|E_{\lambda,\br}}$. Since $|\lambda|=1$ and $S$ is semisimple, the group generated by the restriction of $S$ to $E_{\lambda,\br}$ is relatively compact. Therefore, it fixes a point in the symmetric space $\mathcal{J}_\lambda$ of almost complex structures on $E_{\lambda,\br}$ which are compatible with  the restriction of $\Omega$ to $E_{\lambda,\br}$. Whence an $S$-invariant almost complex structure $J_\lambda$ on $E_{\lambda,\br}$ is compatible with the restriction of $\Omega$. %Conversely, since $\Omega(j\cdot,\cdot)$ is a Euclidean metric, the subgroup $U(E_{\lambda,\br},\Omega,J_\lambda)$ of real, symplectic endomorphisms of $E_{\lambda,\br}$ which commute with $J_\lambda$ is compact.

%Since $S$ and $J_\lambda$ commute, $J_\lambda(E_\lambda)\subset E_\lambda$. Assume that ${J_\lambda}_{|E_\lambda}$ is not a multiple of identity. Then both $\pm i$ are eigenvalues of ${J_\lambda}_{|E_\lambda}$,
%$$
%E_\lambda=E_{\lambda,i}\oplus E_{\lambda,-i}.
%$$
%With respect to $\Omega$, 
%$$
%E_{\lambda,i}\perp E_{\lambda,i}\oplus E_{\lambda,-i}\oplus E_{\bar\lambda,i}.
%$$
%Therefore, for $s\in\br^*$, we can define $M_s=s$ on $E_{\lambda,i}\oplus E_{\bar\lambda,i}$ and $M_s=s^{-1}$ on $E_{\lambda,-i}\oplus E_{\bar\lambda,-i}$. Then $M_s$ is real, symplectic and commutes with $J_\lambda$, so $M_s\in U(E_{\lambda,\br},\Omega,J_\lambda)$. But the one-parameter subgroup $M_s$ is not relatively compact, contradiction. We conclude that ${J_\lambda}_{|E_\lambda}$ is a multiple of identity, hence $J_\lambda$ is a real multiple of $S$. In particular, $L_{|E_\lambda}$ commutes with $J_\lambda$, i.e. $L_{|E_{\lambda,\br}}\in U(E_{\lambda,\br},\Omega,J_\lambda)$. This implies that $L_{|E_{\lambda,\br}}$ is semisimple, $L_{|E_{\lambda,\br}}=S$.
Let 
$$
E_{eu}=\bigoplus_{|\lambda|=1,\,\lambda\not=\pm 1} E_{\lambda,\br},\quad J_{e}=\bigoplus_{|\lambda|=1} J_{\lambda}.
$$
Then $J_{e}$ is compatible with $\Omega_{|E_{eu}}$, and it commutes with $L_e:=S_{|E_{eu}}$. Thus $L_e$ is elliptic with desired properties.

Finally, $E_{u^+}=E_1$, $E_{u^-}=E_{-1}$, $L_{u^+}=L_{|E_{u^+}}$  and $-L_{u^-}=L_{|E_{u^-}}$ are unipotent. 
\end{proof}

\subsubsection{Hyperbolic case}

Now we assume that $L\in \op{Sp}(E,\Omega)$ is hyperbolic and there exists an element $B\in \mf{s}\mf{p}(E,\Omega)$ such that $\exp(2\pi B)=\pm L$. 
From Lemma \ref{dechyperellu}, there exists an involution $I$ such that $I^*\Omega=-\Omega$, and  an almost complex structure $J$ compatible with $\Omega$ such that $IJI=-J$.

Denote 
\begin{equation}\label{ACS}
  \mbf{J}(x)=\exp(-xB)J\exp(xB)
\end{equation}
where $x\in S^1$ and 
\begin{equation}\label{ABJ}
  A_{B,J}:=A_{\mbf{J}}=\mbf{J}\frac{d}{d x}.
\end{equation}

\begin{lemma}
\label{etahyper}
If $L=\pm\exp(2\pi B)\in \op{Sp}(E,\Omega)$ is hyperbolic, then  $\eta(A_{B,J})=0$. 
\end{lemma}

\begin{proof}
Since $I$ commutes with $L$, it defines an automorphism, and still denoted by $I$, of the flat bundle $\mc{E}\to S^1$. Since $IJI=-J$, it induces an orthogonal transformation on the space of sections of $\mc{E}$. One can transport the differential operator $A_{B,J}$ with $I$, and get an orthogonally equivalent operator $I^*A_{B,J}$. Since $IB=BI$ and $IJ=-JI$, $I^*A_{B,J}=-A_{B,J}$. Therefore the spectrum of $A_{B,J}$ is symmetric, its eta function $\eta(s)$ vanishes identically, and $\eta(A_{B,J})=0$.
\end{proof}
From \cite[Theorem 1.1]{Gutt}, there exists a symplectic basis of $(E,\Omega)$ such that the matrix form of $L$ is the symplectic direct sum of matrices of the form 
$
\left(\begin{matrix}
A & 0\\
0 & (A^\top)^{-1}	
\end{matrix}
\right).	
$
From Remark \ref{remhyperbolic}, one has $\bs{\rho}(S^1)=0$. Hence from Theorems \ref{thm0.2}, \ref{thm0.3}
\begin{prop}\label{Hyp}
If $\rho:\pi_1(\Sigma)\ra \op{Sp}(2n,\mb{R})$ is a representation such that each boundary representation is hyperbolic and satisfies  $\pm\rho(c_i)\in \exp(\mf{s}\mf{p}(E,\Omega))\subset\op{Sp}(2n,\mb{R})$, then 
$$|\op{T}(\Sigma,\rho)|\leq n|\chi(\Sigma)|.$$
\end{prop}
For the case of $\dim E=2$, any element $L\in \op{Sp}(2,\mb{R})=\op{SL}(2,\mb{R})$ has the form $L=\pm\exp(2\pi B)$, see Appendix \ref{Appeta}, then 
 we obtain the following result which is contained in \cite[Theorem D]{Gold1}.
\begin{cor}[{Goldman \cite{Gold1}}]\label{MWH}
If $\rho\in\op{Hom}(\pi_1(\Sigma), \op{SL}(2,\mb{R}))$, and $\rho(c_i)\in  \op{SL}(2,\mb{R})$ is hyperbolic  for each boundary component $c_i$, then 
$$|\op{T}(\Sigma,\rho)|\leq |\chi(\Sigma)|.$$	
\end{cor}
\begin{rem}
Under the conditions of the above corollary,  one has $\dim \mathrm{H}^0(\Sigma,\mc{E})=\dim \mathrm{H}^0(\p\Sigma,\mc{E})=0$. Hence 	$|\op{T}(\Sigma,\rho)|= |\chi(\Sigma)|$ if and only if $\op{Ker}(d^+)^*\cap \mathrm{L}^2(\widehat{\Sigma},\wedge^+)=\{0\}$ or $\op{Ker}(d^-)^*\cap \mathrm{L}^2(\widehat{\Sigma},\wedge^-)=\{0\}$.
\end{rem}

\subsubsection{Elliptic case}

If $L\in \op{Sp}(E,\Omega)$ is elliptic, then it commutes with an almost complex structure $J$. In fact, it is conjugate to an element in the unitary group, so there always exists $B\in\mf{s}\mf{p}(E,\Omega)$ such that $ L=\exp(2\pi B)$. With respect to the complex structure and the operator defined in \eqref{ACS} and \eqref{ABJ1}, we have
\begin{lemma}
\label{etaell}
 Let $B$ be the principal determination of $\frac{1}{2\pi}\log L$ so that the eigenvalues $\frac{\theta_j}{2\pi}$ of $-iB$ belong to $[0,1)$. Then 
$$
\eta(A_{B,J})=2n-2\sum_{j=1}^{n}\frac{\theta_j}{\pi}.
$$
\end{lemma}

\begin{proof}
Since $L$ and $B$ are semisimple, the complex vectorspace $(E,J)$ splits as the direct sum of $1$-dimensional $J$-complex subspaces $E=\bigoplus_{j=1}^{n} F_j$, $n=\mathrm{dim}_{\bc}(E,J)=\frac{1}{2}\mathrm{dim}(E)$, and $-iB=\frac{\theta_j}{2\pi}$ on $F_j$, in other words, $L|_{F_j}$ is  $R(\theta_j)=\left(\begin{matrix}
\cos\theta_j & -\sin\theta_j\\
\sin\theta_j & \cos\theta_j	
\end{matrix}
\right)$. The differential operator decomposes as well as $A_{B,J}=\bigoplus_{j=1}^{n}A_j$, where the complex scalar operator $A_j$ is
$
A_j=i\frac{d}{d x}.
$
Therefore $\eta(A_j)=2(1-\frac{\theta_j}{\pi})$ as calculated in real 2-dimensional case in the appendix, and hence
\begin{align}\label{cone}
\eta(A_{B,J})&=\sum_{j=1}^{n}\eta(A_j)=2n-2\sum_{j=1}^{n}\frac{\theta_j}{\pi}.
\end{align}
\end{proof}
Since the complex structure defined in \eqref{ACS} is $\mbf{J}(x)=J$, so $\wt{\mbf{J}}^*\alpha_i=0$. Hence the Rho invariant is 
$\bs{\rho}(S^1)=\eta(A_{B,J})$.
In view of Theorem \ref{thmsign}, we obtain the following modified Milnor-Wood inequality.
\begin{prop}\label{propmod1}
If a representation $\rho:\pi_1(\Sigma)\ra \op{Sp}(2n,\mb{R})$ has only elliptic boundary holonomy, then
$$-n|\chi(\Sigma)|-\sum_{k=1}^q \sum_{j=1}^n \left(1-\frac{\theta_{kj}}{\pi}\right)\leq\op{T}(\Sigma,\rho) \leq n |\chi(\Sigma)| - \sum_{k=1}^q \sum_{j=1}^n \left(1-\frac{\theta_{kj}}{\pi}\right),$$
where $\theta_{kj}\in [0,2\pi)$ denotes the $j$-th rotation angle for the representation of $k$-th boundary.
\end{prop}
Following \cite[Definition 3.5]{BIW}, we call a representation $\rho:\pi_1(\Sigma)\to \op{Sp}(2n,\mb{R})$  maximal if $\op{T}(\Sigma,\rho)=n|\chi(\Sigma)|$. 
Hence when all  $\theta_{kj}<\pi$, the representation $\rho$ can not be maximal. 

Now we consider the case of $n=1$, i.e. consider the representation $\rho:\pi_1(\Sigma)\to \op{SL}(2,\mb{R})$. By Proposition \ref{propmod1}, we obtain the following modified Milnor-Wood inequality.
\begin{prop}\label{propmod2}
If a representation $\rho:\pi_1(\Sigma)\ra \op{SL}(2,\mb{R})$ has only elliptic boundary holonomy,  then 
\begin{equation}\label{TI1}
  -|\chi(\Sigma)|-1 + \sum_{k=1}^q  \frac{\theta_{k}}{\pi}\leq\op{T}(\Sigma,\rho) \leq |\chi(\Sigma)|+1 - \sum_{k=1}^q  \left(1-\frac{\theta_{k}}{\pi}\right),
\end{equation}
where $\theta_k\in (0,\pi)$ such that $[R(\theta_k)]$ is conjugate to $[\rho(c_k)]\in \op{PSL}(2,\mb{R})=\op{SL}(2,\mb{R})/\{\pm I\}$, and $[\bullet]$ denotes the class in $\op{PSL}(2,\mb{R})$, 
 $R(\theta_k)=\left(\begin{matrix}
\cos\theta_k & -\sin\theta_k\\
\sin\theta_k & \cos\theta_k	
\end{matrix}
\right)$.
\end{prop}
\begin{proof}
If each boundary representation of $\rho$ is elliptic, then it is conjugate to a rotation matrix $R(\theta)$ with $\theta\in (0,\pi)\cup (\pi,2\pi)$. Let $q_0$ be the number of all components $c_k$ in $\p\Sigma$ such that $\rho(c_k)$ is conjugate to $R(\theta_k)$, $\theta_k\in (\pi,2\pi)$,  without loss of generality, we assume  $k\in \{1,\cdots, q_0\}$.
Denote $q_0=2m$ or $q_0=2m+1$ for $q_0$ is even or odd. We define a new representation 
\begin{equation}\label{new representation}
  \rho_1(c_i)=
  \begin{cases}
  	& -\rho(c_i),\quad \text{for } 1\leq k\leq 2m;\\
  	&\rho(c),\quad \quad  \text{for } c\in \{c_{2m+1},\cdots,c_q, a_i, b_i, 1\leq i\leq g\}.
  \end{cases}
\end{equation}
Then $\rho_1$ satisfies the relation \eqref{relation}, so it is a representation. Moreover, $\rho_1$ and $\rho$ just differ a sign on some boundaries, so they define the same action on Siegel disk $\op{D}^{\op{III}}_1$, see \eqref{Action}. For any smooth $\rho$-equivariant map $\wt{\mbf{J}}:\wt{\Sigma}\to \op{SL}(2,\mb{R})/\op{U}(1)\cong \op{D}^{\op{III}}_1$, by the construction of the new representation $\rho_1$ \eqref{new representation}, $\wt{\mbf{J}}$ is also a $\rho_1$-equivariant map. From  \eqref{Toledoexp}, one has 
\begin{equation*}
  \op{T}(\Sigma,\rho)=\frac{1}{2\pi}\int_\Sigma \wt{\mbf{J}}^*\omega_{\op{D}^{\op{III}}_1}-\frac{1}{2\pi}\sum_{i=1}^q\int_{c_i}\wt{\mbf{J}}^*\alpha_i=  \op{T}(\Sigma,\rho_1).
\end{equation*}
For the new representation $\rho_1$, $\rho_1(c_k)$ is conjugate to $R(\theta_k)$ with $(0,\pi)$, $1\leq k\leq 2m$. By Proposition \ref{propmod1}, one has
 $$\op{T}(\Sigma,\rho) \leq 
 \begin{cases}
 &	|\chi(\Sigma)| - \sum_{k=1}^{q}  \left(1-\frac{\theta_{k}}{\pi}\right),\quad q_0=2m;\\
 &|\chi(\Sigma)| - \sum_{k=1}^{2m}  \left(1-\frac{\theta_{k}}{\pi}\right)-(1-\frac{\theta_{q_0}}{\pi})- \sum_{k=q_0+1}^{q}  \left(1-\frac{\theta_{k}}{\pi}\right)\quad q_0=2m+1.
 \end{cases}
$$
Since $1-\frac{\theta_{q_0}}{\pi}=(1-\frac{\theta_{q_0}-\pi}{\pi})-1$ with $\theta_{q_0}-\pi\in (0,\pi)$, so 
$$\op{T}(\Sigma,\rho)\leq |\chi(\Sigma)|+1-\sum_{k=1}^{q}  \left(1-\frac{\theta_{k}}{\pi}\right),$$
where $\theta_k\in (0,\pi)$ such that $[R(\theta_k)]$ is conjugate to $[\rho(c_k)]\in \op{PSL}(2,\mb{R})$.
 Similarly, by changing the maximal even number of  $\rho(c_i)$ with angle $\theta_i<\pi$
 to the one with angle $\theta_i'=\theta_i+\pi$, we get $\theta_k'>\pi$, with possible $\theta_0<\pi$,
 $$-\sum(1-\frac{\theta_k'}{\pi})=\sum (\frac{\theta_k'-\pi}{\pi})+\frac{\theta_0-\pi}{\pi}$$ to get $\op{T}(\Sigma,\rho)\geq -|\chi(\Sigma)|-1+\sum_{k=1}^q (\frac{\theta_k'-\pi}{\pi})$. By taking $\theta_k'=\theta_k+\pi$ with $0<\theta_k<\pi$,
  we can get 
  $$\op{T}(\Sigma,\rho)\geq -|\chi(\Sigma)|-1 + \sum_{k=1}^q  \frac{\theta_{k}}{\pi}.$$
\end{proof}
\begin{rem}
	Note that the equalities in \eqref{TI1} can be attained. For example, we consider a cylinder $\Sigma=S^1\times [0,1]$, and the elliptic representation $\rho$ is given by $\rho(S^1\times \{0\})=R(\theta)$ and  $\rho(S^1\times \{1\})=R(2\pi-\theta)$ for some $\theta\in (0,\pi)$. In this case, $\op{T}(\Sigma,\rho)=\chi(\Sigma)=0$, $\theta_1=\theta$ and $\theta_2=\pi-\theta\in (0,\pi)$, and so 
	\begin{equation*}
  -|\chi(\Sigma)|-1+\sum_{k=1}^2\frac{\theta_k}{\pi}=\op{T}(\Sigma,\rho)=|\chi(\Sigma)|+1-\sum_{k=1}^2\left(1-\frac{\theta_k}{\pi}\right).
\end{equation*}
\end{rem}

\begin{rem}
Each element $L=R(\theta)\in \op{SL}(2,\mb{R}),\theta\in (0,\pi)$	gives an automorphism $L_{\mb{D}}=e^{i(2\pi-2\theta)}$ acting on the unit disc $\mb{D}$. In fact, 
note that $\op{D}^{\op{III}}_1=\mb{D}=\{w\in \mb{C}| |w|<1\}$ is the unit disc in the complex plane, and  
\begin{align*}
L=U\left(\begin{matrix}
e^{-i\theta}	& 0\\
0& e^{i\theta}
\end{matrix}
\right)U^{-1},	
\end{align*}
 where $U$ and $U^{-1}$ are given by \eqref{U}, and by Remark \ref{rem5.2}, so 
 $$L_{\mb{D}}(w)=e^{-i\theta}we^{-i\theta}=e^{-2i\theta}w=e^{i(2\pi-2\theta)}w.$$
 If we denote $\phi_k=2\pi-2\theta_k\in (0,2\pi)$, then $\rho(c_k)_{\mb{D}}=e^{i\phi_k}$. By Proposition \ref{propmod2}, one has
 \begin{equation}\label{TI2}
  -|\chi(\Sigma)|-1 + \sum_{k=1}^q  \left(1-\frac{\phi_k}{2\pi}\right)\leq\op{T}(\Sigma,\rho) \leq |\chi(\Sigma)|+1 - \sum_{k=1}^q  \frac{\phi_k}{2\pi}.
\end{equation}

 If the representation $\rho$ is the holonomy of a cone hyperbolic surface $S$ with cone angles $\phi_k\in (0,2\pi)$,
 one can refer to \cite{Mc, Tro} for the definition of the surfaces with conical singularities,
  then $S$ can be identified with $\mb{D}/\rho(\pi_1(\Sigma))$. To get the cone point of cone angle $0<\phi\leq 2\pi$, we need to identify the sector of angle $\phi$ by the rotation of anlge $-(2\pi-\phi)$, hence we need $R(\theta)$ such that $-(2\pi-\phi)=-2\theta$. The induced representation on the boundary is conjugate to the rotation $R(\theta_k)\in \op{SL}(2,\mb{R})$ with $\theta_k\in (0,\pi)$, where $1\leq k\leq q$.   
Then the Toledo invariant is exactly  the area of $S$ and can be given by
  \begin{equation*}
  \op{T}(\Sigma,\rho)=\frac{1}{2\pi}\text{Area}(S)=-\left(\chi(\o{\Sigma})+\frac{1}{2\pi}\sum_{k=1}^q(\phi_k-2\pi)\right)=-\chi(\Sigma)-\sum_{k=1}^q\frac{\phi_k}{2\pi},
\end{equation*}
where $\o{\Sigma}$ is a closed surface obtained by capping off the boundary of $\Sigma$.
Similarly, by conjugating $\rho$
with an orientation-reversing isometry $\tau$ of $\mb{D}$, then we obtain a representation $\rho_\tau$ with cone angles $2\pi-\phi_k$, and $\op{T}(\Sigma,\rho_\tau)\leq 0$. Hence,
\begin{align*}
	\op{T}(\Sigma,\rho_\tau)=\chi(\o{\Sigma})-\frac{1}{2\pi}\sum_{k=1}^q\phi_k=\chi(\Sigma) + \sum_{k=1}^q  \left(1-\frac{\phi_k}{2\pi}\right).
\end{align*}
\end{rem}

\subsubsection{Parabolic case} Recall that $L$ is called parabolic if all eigenvalues of $L$ are $\pm 1$. The reason is as follows.
\begin{prop}
	If $L$ is parabolic, then it fixes a point  at the Shilov boundary of $\op{D}^{\op{III}}_n$.
\end{prop}
\begin{proof}
From \cite[Theorem 1]{Gutt}, there exists a symplectic basis such that $L$ is  symplectic direct sum of matrices of the form
$$\mbf{L}|_{\mb{R}^{2r_j}}=\left(\begin{array}{cc}J\left(\lambda, r_{j}\right)^{-1} & C\left(r_{j}, s_{j}, \lambda\right) \\ 0 & J\left(\lambda, r_{j}\right)^{\top}\end{array}\right)\in \op{Sp}(2r_j,\mb{R})$$
where $C\left(r_{j}, s_{j}, \lambda\right):=J\left(\lambda, r_{j}\right)^{-1} \operatorname{diag}\left(0, \ldots, 0, s_{j}\right)$ with $s_{j} \in\{0,1,-1\}$,  $J(\lambda, r)$ is the elementary $r \times r$ Jordan matrix associated to $\lambda$.  By Remark \ref{rempotential}, $-I_{r_j}$ is a fixed point of $L|_{\mb{R}^{2r_j}}$. Hence $-I_n$ is a fixed point of $L$. With respect to  the other basis, the matrix of $L$ is $P\mbf{L}P^{-1}$
for
 some matrix $P\in\op{Sp}(2n,\mb{R})$. Hence $P(-I_n)$ is a fixed point of $L$, which is also at the Shilov boundary of $\op{D}^{\op{III}}_n$ since  the  Shilov boundary  is an orbit of the action of $\op{Sp}(2n,\mb{R})$ on   $\op{D}^{\op{III}}_n$.                  
\end{proof}
\begin{prop}\label{propmod3}
If there exists a boundary component $c$ such that the representation $\rho(c)$	has an eigenvalue $1$, then 
$|\op{sign}(\mc{E},\Omega)|<\dim E\cdot |\chi(\Sigma)|$. 
\end{prop}
\begin{proof}
From (\ref{siginequality}), one has 
	\begin{align*}
\begin{split}
\pm\op{sign}(\mc{E},\Omega)&\leq -\dim E\chi(\Sigma)-\dim \mathrm{H}^0(\p\Sigma,\mc{E}).
\end{split}
\end{align*}
If  there exists a boundary component $c$ such that the representation $\rho(c)$ has an eigenvalue $1$, then $\rho(c)$ fixes a nonzero vector in $E$, and so $\dim \mathrm{H}^0(\p\Sigma,\mc{E})\geq 1$. Hence 
\begin{equation}\label{signineq0}
  \pm\op{sign}(\mc{E},\Omega)\leq-\dim E\chi(\Sigma)-1<\dim E\cdot|\chi(\Sigma)|,
\end{equation}
which completes the proof.
\end{proof}
\begin{ex}
While for a surface with one boundary component $c$ satisfies $\rho(c)$ has an eigenvalue $-1$, then the signature $\op{sign}(\mc{E},\Omega)$ may attain the maximal $\dim E\cdot|\chi(\Sigma)|$. For example, we consider a pant $P$ with boundary components $p_1,p_2,p_3$, which is homeomorphic to a surface deleting $3$ discs from a $2$-sphere, and consider a representation $\rho:\pi_1(P)\to \op{SO}(2)\subset\op{Sp}(2,\mb{R})$ such that
$$\rho(p_1)=-I_2,\quad \rho(p_2)=R(\theta),\quad \rho(p_3)=R(\pi-\theta)$$
for some $\theta\in (0,\pi)$. From Remark \ref{remunitary} and \eqref{etadim2}, one has
$$\op{T}(P,\rho)=0,\quad \bs{\rho}(\p P)=2(1-\frac{\theta}{\pi})+2(1-\frac{\pi-\theta}{\pi})=2,$$
which follows that 
$$\op{sign}(\mc{E},\Omega)=2\op{T}(P,\rho)+ \bs{\rho}(\p P)=2=\dim E\cdot|\chi(\Sigma)|$$
since $\dim E=2$ and $\chi(\Sigma)=-(2\cdot 0-2+3)=-1$.

\end{ex}

For the case of $n=1$, combining with Corollary \ref{MWH}, Proposition \ref{propmod2}, we  obtain
\begin{prop}
If $\rho:\pi_1(\Sigma)\to \op{SL}(2,\mb{R})$ is a representation, then 
\begin{equation}\label{TI3}
    -|\chi(\Sigma)|-1 + \sum_{\rho(c_k) \text{ is elliptic}}  \frac{\theta_{k}}{\pi}\leq\op{T}(\Sigma,\rho) \leq |\chi(\Sigma)|+1 - \sum_{\rho(c_k) \text{ is elliptic}}  \left(1-\frac{\theta_{k}}{\pi}\right).
\end{equation}
\end{prop}
\begin{proof}
Similar to the proof of Proposition \ref{propmod2}, let 
$\{c_k\}$
denote the maximal set of boundary components such that $\rho(c_k)$ is conjugate to $R(\theta)$, $\theta\in (\pi,2\pi)$ or conjugate to a parabolic element has eigenvalue $1$. Denote by $q_0=\#\{c_k\}$  the number of components in the set $\{c_k\}$. We assume that $q_0=2m$ or $q_0=2m+1$ for $q_0$ is even or odd, and by reversing the representations $\rho(c_k)$ on the boundary components $c_k,1\leq k\leq 2m$,  we obtain a new representation $\rho_1$, and $\op{T}(\Sigma,\rho)=\op{T}({\Sigma,\rho_1})$. Since the Rho invariant $\bs{\rho}_1(c_k)=0$ if $\rho_1(c_k)$ is conjugate to a parabolic element with eigenvalue $-1$, so we can prove \eqref{TI3} for $q_0=2m$ is even. If $q_0=2m+1$ is odd, and $\rho_1(c_{q_0})$ is conjugate to $R(\theta)$, $\theta\in (\pi,2\pi)$, then 
$
   1-\frac{\theta_{q_0}}{\pi}=(1-\frac{\theta_{q_0}-\pi}{\pi})-1
$
with $\theta_{q_0}-\pi\in (0,\pi)$, and we get \eqref{TI3}. If $\rho(c_{q_0})$ is conjugate to a parabolic element has eigenvalue $1$, then $|\frac{\bs{\rho}_1(c_k)}{2}|=1$. Hence \eqref{TI3} is proved.  
\end{proof}

\section{Appendix}

In this section, we will calculate the eta invariant and Rho invariant for two dimension symplectic vector spaces,
classify the elements in $\op{Sp}(2n,\mb{R})$, 
 and  calculate the curvature of the bounded symmetric domain of type $\op{III}$.

\subsection{Eta invariant and Rho invariant for $\op{Sp}(2,\mb{R})$}\label{Appeta}

For any $L\in \op{Sp}(E,\Omega)$, and $B\in\mf{s}\mf{p}(E,\Omega)$ with $L=\pm\exp(2\pi B)$, $B^\top\Omega+\Omega B=0$. Let $J\in\mc{J}(E,\Omega)$, and denote
$\mbf{J}=\exp(-xB)J\exp(xB).$
We define the following $\mb{C}$-linear operator 
$$A_{\mbf{J}}=\mbf{J}\frac{d}{dx}:A^0(S^1,\mc{E}_{\mb{C}})\to A^0(S^1,\mc{E}_{\mb{C}})$$
on the space of smooth sections of $\mc{E}_{\mb{C}}\to S^1$. From Proposition \ref{prop2.0}, the operator $A_{\mbf{J}}$ is a $\mb{C}$-linear first order foramlly self-adjoint and elliptic differential operator, and so the eigenvalues of $A_{\mbf{J}}$ are real.

Suppose $s\in A^0(\mb{R},E)^L=A^0(S^1,\mc{E}_{\mb{C}})$ is an eigenvector of $A_{\mbf{J}}$ belongs to the eigenvalue $\sigma\in \mb{R}$, then 
\begin{align}\label{eigen}
\frac{d}{dx}s=-\sigma \mbf{J}s=-\sigma\exp(-xB)J\exp(xB)s.	
\end{align}
It follows that
$
\frac{d}{dx}(\exp(xB)s)=(-\sigma J+B)\exp(xB)s,	
$
which implies that
$
\exp(xB)s=\exp(x(-\sigma J+B))s(0),	
$
and is equivalent to
\begin{align}\label{5.0}
	s(x)=\exp(-xB)\exp(x(-\sigma J+B))s(0).
\end{align}
The equivariant condition $s(x+2\pi)=L^{-1}s(x)$ implies 
\begin{align}\label{5.1}
\exp(2\pi(-\sigma J+B))s(0)=
\begin{cases}
&s(0),\quad\,\,\,\, \text{ if } L=\exp(2\pi B);\\
&-s(0),\quad\text{ if } L=-\exp(2\pi B);\\ 	
\end{cases}
\end{align}

Let $\mc{S}$ be  a maximally $\mb{C}$-linearly independent set of  eigenvectors of $A_{\mbf{J}}$. 
 For each $s(x)\in \mc{S}$, we denote by $\sigma_{s(x)}$ the associated eigenvalue. Then the  eigenvalues of $A_{\mbf{J}}$ with multiplicity are
\begin{align*}
\op{Eigen}(A_{\mbf{J}})=\bigsqcup_{s(x)\in \mc{S}}\{\sigma_{s(x)}\}.	
\end{align*}
\begin{lemma}\label{lemma1}
If $s_1(x)=s_2(x)$ is an eigenvector of $A_{\mbf{J}}$, then
$$(\sigma_{s_1(x)}, s_1(0))=(\sigma_{s_2(x)}, s_2(0))\in\mb{R}\times(\mb{C}^{2n}\backslash\{0\})$$
is a solution of the following  equation
\begin{align}\label{5.4}
\exp(2\pi(-\sigma J+B))e=
\begin{cases}
&e,\quad \,\,\,\,\text{ if } L=\exp(2\pi B);\\
&-e,\quad\text{ if } L=-\exp(2\pi B);\\ 	
\end{cases}
\end{align}

\end{lemma}
\begin{proof}
	If $s_1(x)=s_2(x)$, then  
	$$\exp(x(-\sigma_1 J+B))s_1(0)=\exp(x(-\sigma_2J+B))s_2(0),$$
	and it holds \eqref{5.4}.
	Taking the first derivative on $x$ to the both sides of the above equality, one has
	$$(-\sigma_1J+B)\exp(x(-\sigma_1 J+B))s_1(0)=(-\sigma_2J+B)\exp(x(-\sigma_2 J+B))s_2(0)$$
	which implies $\sigma_1=\sigma_2$, and so $s_1(0)=s_2(0)$.
\end{proof}

For a subset $\mf{S}$ in the set of all solutions $\mb{R}\times(\mb{C}^{2n}\backslash\{0\})$ of \eqref{5.4}, we call $\mf{S}$ is maximally $\mb{C}$-independent if for any eigenvalue $\sigma$, $\mf{S}_\sigma=\cup_{(\sigma,e)\in\mf{S}}\{e\}$ is  maximally  $\mb{C}$-linearly independent  in the set of all vectors associated to $\sigma$. 
\begin{prop}\label{propA1}
	If $\mf{S}$ is a maximally $\mb{C}$-linearly independent subset in the set of all solutions of \eqref{5.4}, then
	\begin{align*}
		\op{Eigen}(A_{\mbf{J}})=\bigsqcup_{(\sigma,e)\in\mf{S}}\{\sigma\}.
	\end{align*}
\end{prop}
\begin{proof}
	For any $(\sigma, e)\in \mf{S}$, we define 
	\begin{align}
	\Phi(\sigma,e)=\exp(-xB)\exp(x(-\sigma J+B))e,
	\end{align}
	From  \eqref{5.0} and \eqref{5.1},  $\Phi(\sigma,e)$ is an eigenvector of $A_{\mbf{J}}$ belongs to $\sigma$. By Lemma \ref{lemma1}, $\Phi$ is injective. We claim $\Phi(\mf{S})$ is maximally $\mb{C}$-linearly independent. In fact,
if $s(x)$ is an eigenvector of $A_{\mbf{J}}$, then 
$$s(x)=\exp(-xB)\exp(x(-\sigma_{s(x)} J+B))s(0).$$
By assumption, $\mf{S}_{\sigma_{s(x)}}$ is maximally $\mb{C}$-linearly independent, so 
$s(0)=\sum_i a_ie_i,\quad a_i\in\mb{C}$
where the coefficients $a_i,i\geq 1$ are unique. 
Thus
\begin{align*}
	s(x) &=\sum_i a_i\exp(-xB)\exp(x(-\sigma_{s(x)} J+B))e_i
	=\sum_ia_i\Phi(\sigma_{s(x)},e_i),
\end{align*}
so we prove the claim that $\Phi(\mf{S})$ is maximally $\mb{C}$-linearly independent. Hence
\begin{align*}
\op{Eigen}(A_{\mbf{J}})=\bigsqcup_{s(x)\in \Phi(\mf{S})}\{\sigma_{s(x)}\}=\bigsqcup_{(\sigma,e)\in\mf{S}}\{\sigma\}.	
\end{align*}
\end{proof}
The following lemma is useful in the calculation of eta invariant. 
\begin{lemma}\label{lemma2}
If there exists a constant $c_0>0$ such that  $(ak^2+bk+c)^{1/2}\geq l+c_0 k$ for any $k\geq 1$, where $l>0, a>0$, $s\in (0,1/2)$ and $b,c\in \mb{R}$, 
then
	$$\lim_{s\to 0}\sum_{k=1}^{\infty}\left(\frac{1}{((ak^2+bk+c)^{1/2}-l)^s}-	\frac{1}{((ak^2+bk+c)^{1/2}+l)^s}\right)=\frac{2l}{\sqrt{a}}.$$
\end{lemma}
\begin{proof}
Denote
\begin{align*}
	F(k):&=\frac{1}{((ak^2+bk+c)^{1/2}-l)^s}-	\frac{1}{((ak^2+bk+c)^{1/2}+l)^s}\\
	&=ls\int_{-1}^1\frac{d\theta}{((ak^2+bk+c)^{1/2}-\theta l)^{s+1}}.
\end{align*}
Then 
\begin{align}\label{FZeta}
\begin{split}
	&\quad \sum_{k=1}^{\infty}F(k)
	=\sum_{k=1}^{\infty}ls\int^1_{-1}\frac{d\theta}{((ak^2+bk+c)^{1/2}-\theta l)^{s+1}}\\
	&=\sum_{k=1}^{\infty}\frac{ls}{k^{s+1}}\int^1_{-1}\left(\frac{k^{s+1}}{((ak^2+bk+c)^{1/2}-\theta l)^{s+1}}-(\frac{1}{\sqrt{a}})^{s+1}\right)d\theta+\sum_{k=1}^{\infty}\frac{2ls}{k^{s+1}}(\frac{1}{\sqrt{a}})^{s+1}.
	\end{split}
\end{align}
Define a continuous function $f(x)=((a+bx+cx^2)^{1/2}-\theta lx)^{-(s+1)}$, $x\in [0,1]$. By assumption, one has  $0< f(x)\leq c_0^{-(s+1)}$. Thus
\begin{align*}
|f'(x)| &=(s+1)f(x)^{\frac{s+2}{s+1}}	|\frac{1}{2}(a+bx+cx^2)^{-1/2}(b+2cx)-\theta l|\leq C(s+1),
\end{align*}
which follows that 
$|f(\frac{1}{k})-f(0)|\leq C(s+1)\frac{1}{k}$,
and so
$$\left|\frac{k^{s+1}}{((ak^2+bk+c)^{1/2}-\theta l)^{s+1}}-(\frac{1}{\sqrt{a}})^{s+1}\right|\leq C(s+1)\frac{1}{k},$$
which implies that the first term in RHS of \eqref{FZeta} vanishes since $\lim_{s\to 0} s(s+1)\zeta(s+2)=0$, where $\zeta(s)=\sum_{k=1}^{\infty}k^{-s}$ is the zeta function. 
Hence
\begin{align*}
\lim_{s\to 0}\sum_{k=1}^{\infty}F(k)=\lim_{s\to 0}\sum_{k=1}^{\infty}\frac{2ls}{k^{s+1}}(\frac{1}{\sqrt{a}})^{s+1}=2l\lim_{s\to 0}s\zeta(s+1)(\frac{1}{\sqrt{a}})^{s+1}=\frac{2l}{\sqrt{a}},
\end{align*}
where the last equality follows from the fact $\lim_{s\to 0}s\zeta(s+1)=1$.
\end{proof}

For $\dim E=2$ and any $L\in \op{Sp}(E,\Omega)\cong\op{Sp}(2,\mb{R})=\op{SL}(2,\mb{R})$, there exists a symplectic basis of $(E,\Omega)$ such that $L$ has the following matrix form:
\begin{itemize}
\item[(1)] $L$ is hyperbolic, i.e. $|\op{Tr}(L)|>2$. In this case,  $\lambda\not\in S^1$and $\lambda\in \mb{R}$, $L$ has the form: $$
\left(\begin{array}{cc}
\lambda & 0 \\
0 & \frac{1}{\lambda}
\end{array}\right);
$$	

\item[(2)] $L$ is elliptic, i.e. $|\op{Tr}(L)|<2$. In this case, the eigenvalue $\lambda\in S^1\backslash\{\pm 1\}$,  $L$ is given by
$$
R(\theta)=\left(\begin{array}{cc}\cos\theta & -\sin\theta \\
\sin\theta & \cos\theta
\end{array}\right),
$$	
where $\theta\in (0,\pi)\cup(\pi,2\pi)$.
\item[(3)] $L$ is parabolic, i.e. $|\op{Tr}(L)|=2$, and so $\lambda=\pm 1$,  $L$ has the following form: $$
\left(\begin{array}{ll}
\lambda & \mu \\
0 & \lambda
\end{array}\right),
$$	
where $\mu\in \mb{R}$.
\end{itemize}
 Fix a complex structure $J=\left(\begin{matrix}
	0 & -1\\
	1 & 0
\end{matrix}\right)
$. Now we begin to calculate the eta invariant $\eta(A_{\mbf{J}})$ and Rho invariant $\bs{\rho}(S^1)=4\int_{S^1}\mbf{J}^*\alpha+\eta(A_{\mbf{J}})$.

(1) $\lambda\not\in S^1$. For the case $\lambda>0$, we take  
$$B=\frac{1}{2\pi}\log |\lambda|\left(\begin{array}{cc}
1 & 0 \\
0 & -1
\end{array}\right),$$
such that $L=\exp(2\pi B)$,
then 
the solutions of the equation $\exp(2\pi(-\sigma J+B))e=e$ are 
$$(\sigma,e)=(\pm\frac{1}{2\pi}((\log\lambda)^2+(2\pi k)^2)^{1/2},\mb{C}^2),\quad k\in\mb{Z}.$$
Thus the set of nonzero eigenvalues and the corresponding vectors is 
$$\mf{S}=\{(\pm\frac{1}{2\pi}((\log\lambda)^2+(2\pi k)^2)^{1/2},e_1), (\pm\frac{1}{2\pi}((\log\lambda)^2+(2\pi k)^2)^{1/2},e_2),k\in\mb{Z}\},$$
where $e_1=(1,0)^{\top}$ and $e_2=(0,1)^{\top}$.
Since $\mf{S}$ is symmetric, so $\eta(A_{\mbf{J}})=0$.

If $\lambda<0$, then $\exp(2\pi B)=-L$, and the solutions of the equation $\exp(2\pi(-\sigma J+B))e=-e$ are
$$(\sigma,e)=(\pm\frac{1}{2\pi}((\log\lambda)^2+(2\pi k)^2+(1-4k)\pi^2)^{1/2},\mb{C}^2),\quad k\in\mb{Z}.$$
So the spectrum of $A_{\mbf{J}}$ is symmetric, so  $\eta(A_{\mbf{J}})=0$.

From Remark \ref{rempotential}, then 
\begin{equation}\label{potentialalpha}
  \alpha=-d^c\log\det\op{Im}Z=-d^c\log\op{Im}Z=\frac{1}{2\op{Im}Z}(dZ+d\o{Z}).
\end{equation}
The complex structure $\mbf{J}(x)$ is given by  {
\begin{equation*}
  \mbf{J}(x)=\exp(-xB)J\exp(xB)=\left(\begin{matrix}
0 & -|\lambda|^{-\frac{x}{\pi}}\\
|\lambda|^{\frac{x}{\pi}} &0	
\end{matrix}
\right),
\end{equation*}}
By \eqref{W}, one has
$$W\circ \mbf{J}(x)=(2+|\lambda|^{\frac{x}{\pi}}+|\lambda|^{-\frac{x}{\pi}})^{-1}(|\lambda|^{\frac{x}{\pi}}-|\lambda|^{-\frac{x}{\pi}}).$$
Thus
$$Z\circ\mbf{J}(x)=i(1-W\circ \mbf{J}(x))(1+W\circ \mbf{J}(x))^{-1}$$
is purely imaginary, which follows that 
\begin{align*}
\mbf{J}^*\alpha=	\frac{1}{2\op{Im}Z}(d(Z\circ \mbf{J}(x))+\o{d(Z\circ \mbf{J}(x))})=0.
\end{align*}
 Hence 
 $$\bs{\rho}(S^1)=\frac{1}{\pi}\int_{S^1}\mbf{J}^*\alpha+\eta(A_{\mbf{J}})=\eta(A_{\mbf{J}})=0.$$
\begin{rem}\label{remhyperbolic}
For any $L\in \op{Sp}(2n,\mb{R})$ with  the following matrix form
\begin{align*}
L=\pm\left(\begin{matrix}
\exp(2\pi B) & 0\\
0 & \exp(-2\pi B^\top)	
\end{matrix}
\right),	
\end{align*}
where $B\in \mf{s}\mf{p}(2n,\mb{R})$.
One can check that $Z\circ \mbf{J}(x)$ is also purely imaginary, and so $\int_{S^1}\wt{\mbf{J}}^*\alpha=0$, 
 where $\mbf{J}(x)=\exp(-xB)J\exp(xB)$ and $J$ is the standard complex structure. Moreover, the set of all solutions of  equation \eqref{5.4} is symmetric, so $\eta(A_{\mbf{J}})=0$. Hence $\bs{\rho}(S^1)=0$.
\end{rem}

(2) $\lambda\in S^1\backslash\{\pm 1\}$. In this case, $L=R(\theta)$ and $B=\frac{\theta}{2\pi}J$, 
$
	\exp(2\pi(-\sigma J+B))=R{(-2\pi\sigma+\theta)}.
$
Hence the solutions of $\exp(2\pi(-\sigma J+B))e=e$ are given by 
$$\sigma=\frac{\theta}{2\pi}+k,\quad k\in\mb{Z},$$
and for each $\sigma$, $e$ can be taken any element of $\mb{C}^2$. The set
\begin{align*}
\mf{S}=\left\{\left(\frac{\theta}{2\pi}+k,e_1\right),\left(\frac{\theta}{2\pi}+k,e_2\right),k\in\mb{Z}\right\},	
\end{align*}
is maximally $\mb{C}$-linearly independent. The set of eigenvectors is $\Phi(\mf{S})$.
The eta function 
\begin{align*}
\frac{1}{2}\eta_{A_{\mbf{J}}}(s)=\left(\frac{\theta}{2\pi}\right)^{-s}+\sum_{k=1}^{\infty}\left(\frac{1}{|k+\frac{\theta}{2\pi}|^s}-	\frac{1}{|k-\frac{\theta}{2\pi}|^s}\right).
\end{align*}
From Lemma \ref{lemma2}, one has
$\eta(A_{\mbf{J}})=2(1-\frac{\theta}{\pi})$. Since $[J,B]=0$, so
$$\mbf{J}(x)=\exp(-xB)J\exp(xB)=J,$$
and $\mbf{J}^*\alpha=0$. Hence  $\bs{\rho}(S^1)=\frac{1}{\pi}\int_{S^1}\mbf{J}^*\alpha+\eta(A_{\mbf{J}})=\eta(A_{\mbf{J}})=2(1-\frac{\theta}{\pi})$.

(3) $\lambda=\pm1$. In this case, $L$ is given by 
$$
\left(\begin{array}{ll}
\lambda & \mu \\
0 & \lambda
\end{array}\right),
$$	
where $\mu\in \mb{R}$, and $L=\lambda \exp(2\pi B)$, $B$ is given by
$$
B=\left(\begin{array}{cc}
0 & \frac{1}{2 \pi} \frac{\mu}{\lambda} \\
0 & 0
\end{array}\right).
$$
Thus
$$
2\pi(-\sigma J+B)=\left(\begin{array}{cc}
0& 2\pi\sigma+ \frac{\mu}{\lambda} \\
-2\pi\sigma & 0
\end{array}\right).
$$
Then the solution of $\exp(2\pi(-\sigma J+B))e=\lambda e$ is given by
$$\sigma=-\frac{\mu}{4\pi\lambda}\pm\frac{1}{2\pi}\left(\frac{\mu^2}{4}-(\log\lambda+2k\pi i)^2\right)^{\frac{1}{2}},$$
where $\log\lambda=0$ if $\lambda=1$; $\log\lambda=\pi i$ if $\lambda=-1$.
Hence the  maximally $\mb{C}$-linearly independent set without zero eigenvalue is 
\begin{multline*}
\mf{S}_o=\left\{\left(-\frac{\mu}{4\pi\lambda}\pm\frac{1}{2\pi}\left(\frac{\mu^2}{4}-(\log\lambda+2k\pi i)^2\right)^{\frac{1}{2}}, e_i\right),k\in\mb{Z},k\geq 0,i=1,2\right\}\\\backslash\{(0,e_i),i=1,2\}.	
\end{multline*}
Thus, if
 $\lambda=1$ and $\mu> 0$, then 
\begin{multline*}
 \frac{1}{2}\eta_{A_{\mbf{J}}}(s)
=-\left(\frac{\mu}{2\pi}\right)^{-s}\\+\sum_{k=1}^{\infty}\left[-\left(\frac{\mu}{4\pi}+\frac{1}{2\pi}\left(\frac{\mu^2}{4}+(2k\pi)^2\right)^{\frac{1}{2}}\right)^{-s}	+\left(-\frac{\mu}{4\pi}+\frac{1}{2\pi}\left(\frac{\mu^2}{4}+(2k\pi)^2\right)^{\frac{1}{2}}\right)^{-s}\right].
\end{multline*}
By Lemma \ref{lemma2}, the eta invariant is 
$\eta(A_{\mbf{J}})=2(-1+\frac{\mu}{2\pi})$.
 Similarly, one has
\begin{itemize}
	\item If  $\lambda=1$ and $\mu< 0$, then  
$\eta(A_{\mbf{J}})=2(1+\frac{\mu}{2\pi})$;
\item If $\lambda=-1$, then 
$\eta(A_{\mbf{J}})=-\frac{\mu}{\pi}$;
\item If $\mu=0$, then $\eta(A_{\mbf{J}})=0$.
\end{itemize}
In this case, the complex structure is 
\begin{align*}
	\mbf{J}(x)=\exp(-xB)J\exp(xB)=\left(\begin{matrix}
b &-b^2-1\\
1 &-b	
\end{matrix}
\right),
\end{align*}
where $b=-\frac{x}{2\pi}\frac{\mu}{\lambda}$. Then 
{$$ W\circ  \mbf{J}(x)=\frac{2ib-b^2}{4+b^2}$$}
$$Z\circ \mbf{J}(x)=i(2+bi)^{-1}(2+b^2-ib)=i+\frac{b^3+4b}{b^2+4}={ i+b}.$$
By \eqref{potentialalpha}, one has 
\begin{equation*}
 \mbf{J}^* \alpha=db=-\frac{1}{2\pi}\frac{\mu}{\lambda}dx.
\end{equation*}
Hence 
$$\frac{1}{\pi}\int_{S^1}\mbf{J}^*\alpha=-\frac{1}{\pi}\frac{\mu}{\lambda}.$$
Therefore, if $\dim E=2$,  the eta invariant and the Rho invariant  are given by the following table:
\begin{center}
\begin{align}\label{etadim2}
\begin{tabular}{| c | c | c| }
\hline
     $\lambda,\mu$  &  $\eta(A_{\mbf{J}})$  & $\bs{\rho}(S^1)$\\
\hline
   $\lambda\not\in S^1$   & $0$ & $0$\\
\hline
   $\lambda\in S^1\backslash\{\pm1\}$    & $2(1-\frac{\theta}{\pi})$&$2(1-\frac{\theta}{\pi})$\\
\hline
    $\mu=0$  & $0$ & $0$\\
\hline
$\lambda=1, \mu>0$ & $2(-1+\frac{\mu}{2\pi})$ & $-2$\\
\hline
 $\lambda=1, \mu<0$ & $2(1+\frac{\mu}{2\pi})$ & $2$\\
\hline
$\lambda=-1$&$-\frac{\mu}{\pi}$ & $0$\\
\hline
\end{tabular}	
\end{align}
\end{center}

\subsection{Curvature of the bounded symmetric domain of type $\op{III}$}\label{App2}
In this subsection, we will calculate the curvature of the bounded symmetric domain of type III. One can also refer to \cite[Chapter 4, \S 3]{Mok} and \cite[Section 8]{Siu}.

Denote by 
$$\op{D}^{\op{III}}_n:=\{W\in \mf{g}\mf{l}(n,\mb{C}):W=W^\top,I_n-\o{W}W>0\}$$
the Siegel's generalized upper half-plane, see \cite[Chapter VIII, \S 7]{Hel}.  Then
$$\omega_{\op{D}^{\op{III}}_n}:=-2i\p\b{\p}\log\det(I-\o{W}W)=\frac{i}{2}h_{a\b{b}}dw^a\wedge d\o{w}^b$$
is a K\"ahler form on $\op{D}^{\op{III}}_n$. The Hermitian metric is denoted by $h=h_{a\b{b}}dw^a\otimes d\b{w}^b$, and the holomorphic sectional curvature is 
\begin{align*}
\op{K}(V)=\frac{2R(V,\o{V},V,\o{V})}{\|V\|^4}=\frac{2R_{a\b{b}c\b{d}}V^a\o{V^b}V^c\o{V^d}}{(h_{a\b{b}}V^a\o{V^b})^2}	
\end{align*}
 for any $V=V^a\frac{\p}{\p w^a}$ and $R_{a\b{b}c\b{d}}=-\p_{c}\p_{\b{d}}h_{a\b{b}}+\mathrm{H}^{p\b{q}}\p_c h_{a\b{q}}\p_{\b{d}}h_{p\b{b}}$. Since $\omega_{\op{D}^{\op{III}}_n}$ is invariant under the group  $\op{Aut}(\op{D}^{\op{III}}_n)$ of holomorphic automorphisms, so we just need to calculate the holomorphic sectional curvature of the Bergman metric $\omega_{\op{D}^{\op{III}}_n}$ at $W=0$. If $H$ is a Hermitian matrix, then 
$$\p\b{\p}\log\det H=\op{Tr}({H}^{-1}\p\b{\p}H)-\op{Tr}({H}^{-1}\p H\wedge {H}^{-1}\b{\p} H).$$
Denote
$\frac{1}{4}h_{ab\o{cd}}=(I-\o{W}W)^{-1}_{ac}\delta_{bd}+(I-\o{W}W)^{-1}_{n m}(I-\o{W}W)^{-1}_{bc}\o{W_{am}}W_{dn}.$
Then 
\begin{align*}
\omega_{\op{D}^{\op{III}}_n}&=\frac{i}{2}\sum_{a<b,c<d}(h_{ab\o{cd}}+h_{ba\o{cd}}+h_{ab\o{dc}}+h_{ba\o{dc}})dW_{ab}\wedge d\o{W_{cd}}\\
&\quad+\frac{i}{2}\sum_{a<b,c}(h_{ab\o{cc}}+h_{ba\o{cc}})dW_{ab}\wedge d\o{W_{cc}}\\
&\quad+\frac{i}{2}\sum_{c<d,a}(h_{aa\o{cd}}+h_{aa\o{dc}})dW_{aa}\wedge d\o{W_{cd}}\\
&\quad+\frac{i}{2}h_{aa\o{cc}}dW_{aa}\wedge d\o{W_{cc}}=\frac{i}{2}\sum_{a\leq b,c\leq d}H_{ab\o{cd}}dW_{ab}\wedge d\o{W_{cd}},
\end{align*}
where 
$H_{ab\o{cd}}=(h_{ab\o{cd}}+h_{ba\o{cd}}+h_{ab\o{dc}}+h_{ba\o{dc}})(1-\frac{1}{2}\delta_{ab})(1-\frac{1}{2}\delta_{cd}).$
At the point $W=0$, one has
\begin{align*}
\frac{1}{4}h_{ab\o{cd}}=\delta_{ac}\delta_{bd},\quad \p h_{abcd}=0,
\end{align*}
and
\begin{align*}
\frac{1}{4}\frac{\p^2 h_{ab\o{cd}}}{\p W_{kl}\p\o{W_{pq}}}&=(\delta_{ap}\delta_{mq}+\delta_{aq}\delta_{mp})(\delta_{mk}\delta_{cl}+\delta_{ml}\delta_{ck})\delta_{bd}(1-\frac{1}{2}\delta_{pq})(1-\frac{1}{2}\delta_{kl})\\
&	\quad+(\delta_{ap}\delta_{mq}+\delta_{aq}\delta_{mp})(\delta_{dk}\delta_{ml}+\delta_{dl}\delta_{mk})\delta_{bc}(1-\frac{1}{2}\delta_{pq})(1-\frac{1}{2}\delta_{kl}).
\end{align*}
The curvature is 
\begin{align*}
	&\quad R_{ab\o{cd}kl\o{pq}} =-\frac{\p^2 H_{ab\o{cd}}}{\p W_{kl}\p\o{W_{pq}}}\\
	&=-8((\delta_{ap}\delta_{mq}+\delta_{aq}\delta_{mp})(\delta_{mk}\delta_{cl}+\delta_{ml}\delta_{ck})\delta_{bd}+(\delta_{ap}\delta_{mq}+\delta_{aq}\delta_{mp})(\delta_{mk}\delta_{dl}+\delta_{ml}\delta_{dk})\delta_{bc}\\
	&+(\delta_{bp}\delta_{mq}+\delta_{bq}\delta_{mp})(\delta_{mk}\delta_{cl}+\delta_{ml}\delta_{ck})\delta_{ad}+(\delta_{bp}\delta_{mq}+\delta_{bq}\delta_{mp})(\delta_{mk}\delta_{dl}+\delta_{ml}\delta_{dk})\delta_{ac})\\
	&\cdot(1-\frac{1}{2}\delta_{ab})(1-\frac{1}{2}\delta_{cd})(1-\frac{1}{2}\delta_{pq})(1-\frac{1}{2}\delta_{kl}),
\end{align*}
where  $a\leq b,c\leq d,k\leq l,p\leq q$. Since 
$$H(0)=\sum_{a\leq b,c\leq d}H_{ab\o{cd}}(0)dW_{ab}\otimes d\o{W_{cd}}=4(\sum_{a<b}2dW_{ab}\otimes d\o{W_{ab}}+\sum_adW_{aa}\otimes d\o{W_{aa}})$$
where 
$H_{ab\o{cd}}(0)=8(\delta_{ac}\delta_{bd}+\delta_{bc}\delta_{ad})(1-\frac{1}{2}\delta_{ab})(1-\frac{1}{2}\delta_{cd}).$
So the inverse matrix is 
\begin{align*}
{H}^{\o{cd}ab}(0)=\frac{1}{8}	(\delta_{ac}\delta_{bd}+\delta_{bc}\delta_{ad}).
\end{align*}
Then the Ricci curvature is 
\begin{align*}
R_{kl\o{pq}}&=	{H}^{\o{cd}ab}R_{ab\o{cd}kl\o{pq}}=\frac{1}{8}\sum_{a<b}R_{ab\o{ab}kl\o{pq}}+\frac{1}{4}\sum_a R_{aa\o{aa}kl\o{pq}}\\
&=-(n+1)\sum_{a=1}^n (\delta_{ap}\delta_{mq}+\delta_{aq}\delta_{mp})(\delta_{mk}\delta_{al}+\delta_{ml}\delta_{ak})(1-\frac{1}{2}\delta_{pq})(1-\frac{1}{2}\delta_{kl})\\
&=-2(n+1)(\delta_{pl}\delta_{qk}+\delta_{pk}\delta_{ql})(1-\frac{1}{2}\delta_{pq})(1-\frac{1}{2}\delta_{kl})\\
&=-\frac{n+1}{4}H_{kl\o{pq}}.
\end{align*}
Thus the first Chern form satisfies
\begin{align*}
\frac{i}{2\pi}R_{kl\o{pq}}dW_{kl}\wedge d\o{W_{pq}}&=	-\frac{i}{2\pi}\frac{n+1}{4}H_{kl\o{pq}}dW_{kl}\wedge d\o{W_{pq}}=-\frac{n+1}{2}\frac{1}{2\pi}\omega_{\op{D}^{\op{III}}_n}.
\end{align*}
Now we calculate the holomorphic sectional curvature. Note that
\begin{align*}
&\quad \sum_{a\leq b, c\leq d\atop k\leq l, p\leq q}R_{ab\o{cd}kl\o{pq}}dW_{kl}\wedge d\o{W_{pq}}\otimes dW_{ab}\wedge d\o{W_{cd}}\\
&=-\sum_{a\leq b, c\leq d}\p\b{\p}H_{ab\o{cd}}\otimes 	 dW_{ab}\wedge d\o{W_{cd}}\\
&=-\sum_{a,b,c,d}\p\b{\p}h_{ab\o{cd}}\otimes 	 dW_{ab}\wedge d\o{W_{cd}}\\
&=\sum_{a,b,c,d,k,l,p,q}\widetilde{R}_{ab\o{cd}kl\o{pq}}dW_{kl}\wedge d\o{W_{pq}}\otimes dW_{ab}\wedge d\o{W_{cd}},
\end{align*}
where $\widetilde{R}_{ab\o{cd}kl\o{pq}}=-4(\delta_{bd}\delta_{ck}\delta_{ap}\delta_{ql}+\delta_{bq}\delta_{ca}\delta_{pk}\delta_{dl})$. 
For any $(1,0)$-type tangent vector $V=(V^{ab})$, $a\leq b$, at $0$. We also denote by $\mathbf{V}$ the matrix associated with the vector $(V^{ab})$ by setting $V^{ab}=V^{ba}$.  Then  
\begin{align*}
	&\quad R_{ab\o{cd}kl\o{pq}}V^{ab}\o{V^{cd}}V^{kl}\o{V^{pq}}\\
	&=(\sum_{a\leq b, c\leq d\atop k\leq l, p\leq q}R_{ab\o{cd}kl\o{pq}}dW_{kl}\wedge d\o{W_{pq}}\otimes dW_{ab}\wedge d\o{W_{cd}})(V\wedge\o{V}\otimes V\wedge\o{V})\\
	&=(\sum_{a,b,c,d,k,l,p,q}\widetilde{R}_{ab\o{cd}kl\o{pq}}dW_{kl}\wedge d\o{W_{pq}}\otimes dW_{ab}\wedge d\o{W_{cd}})(\mathbf{V}\wedge\o{\mathbf{V}}\otimes \mathbf{V}\wedge\o{\mathbf{V}})\\
	&=\widetilde{R}_{ab\o{cd}kl\o{pq}}\mathbf{V}^{ab}\o{\mathbf{V}^{cd}}\mathbf{V}^{kl}\o{\mathbf{V}^{pq}}=-8\op{Tr}((\o{\mathbf{V}}\mathbf{V})^2),
\end{align*}
and $\|V\|^2=4(2\sum_{a<b}|V^{ab}|^2+\sum_a|V^{aa}|^2)=4\op{Tr}(\o{\mathbf{V}}\mathbf{V})$. Thus
the holomorphic sectional curvature is 
\begin{align*}
\op{K}(V)=-\frac{\op{Tr}((\o{\mathbf{V}}\mathbf{V})^2)}{(\op{Tr}(\o{\mathbf{V}}\mathbf{V}))^2}\in [-1,-1/n],
\end{align*}
and $\op{K}(V)=-1/n$ iff $\mathbf{V}(\det\mathbf{V})^{-1/n}$ is a unitary group. 

\subsection{Milnor-Wood inequality of Toledo invariant for $n=1$}
Let $\Sigma$ be a surface with $q$-boundary components and of genus $g$ of negative Euler number.
Considering a boundary component as a puncture, one can find an ideal triangulation $\triangle$ of $\Sigma$, which is just a maximal collection of disjoint essential arcs that are pairwise non-homotopic,
whose vertices are at punctures. If there are $F$ ideal triangles in $\triangle$, there
there are $\frac{3F}{2}$ edges since each edge is shared by the adjacent triangle.
Here by taking a  triangulation carefully, we  can assume that two adjacent triangles are distinct. Hence by the definition of the Euler number
$$F-\frac{3F}{2}=2-2g-q$$ where $q$ is the number of puctures. Hence there are $-4+4g+2q$ ideal triangles in $\triangle$.
Now when we consider the {punctures} as boundaries, these ideal triangles wrap around each boundary component infinitely many times, and still denote this triangulation by $\triangle$.
\begin{figure}[hbt]
\begin{center}
\centerline{{ \includegraphics{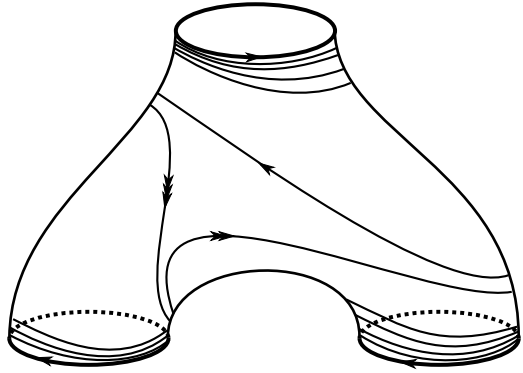}}}
\end{center}
\caption{Ideal triangulation of a pair of pants consisting of two ideal triangles }\label{infinite extreme point}
\end{figure}
Given a representation $\rho:\pi_1(\Sigma)\ra \text{SL}(2,\br)$, find a $\rho$-equivariant map ${\bf J}:\widetilde\Sigma\ra D=\op{D}^{\op{III}}_1$ induced by an equivariant complex structure. Recall that
$$ \op{T}(\Sigma,\rho)=\frac{1}{2\pi}\int_\Sigma {\mbf{J}}^*\omega_{\op{D}^{\op{III}}_1}-2\sum_{i=1}^q\int_{c_i}{\mbf{J}}^*\alpha_i.$$
For hyperbolic and elliptic boundary, the second term is zero. Now we can straighten
${\bf J}$ such that each ideal triangle $\sigma$ in $\triangle$ is mapped to an ideal geodesic triangle in $D$. Now  a new
 surface $\Sigma'=\Sigma\cup \cup C_i$ is obtained from $\Sigma$ by attaching  a cone $C_i$ to each parabolic boundary $c_i$. One can extend $\op{Str}({\bf J})$ in an obvious way on each cone to the corresponding horoball neighborhood.
 In more details, 
$$\mbf{J}:S^1\times [0,\infty)\ra  H,$$ where $H$ is a horoball  and $\mbf{J}$ maps each geodesic $x\times[0,\infty)$ to an arc-length parametrized geodesic from $\mbf{J}(x,0)$ to the base point of the horoball. 
Note here that the complex structure along the  boundary $c_i$ is an orbit
of $J$ under the one-parameter family generated by parabolic element $\rho(c_i)$, hence
the image under $\mbf{J}$ is also the orbit of the one-parameter family generated by parabolic element $\rho(c_i)$. This forces that $\mbf{J}(x \times [0,\infty))$ and $\mbf{J}(y\times [0,\infty))$ get closer exponentially fast as $t\rightarrow \infty$ for any two points $x$ and $y$ on $c_i$.

In more details,
$|\mbf{J}_*({\frac{\partial}{\partial x}}|_{(x,t)})|=e^{-t}|\mbf{J}_*({\frac{\partial}{\partial x}}|_{(x,0)})|$, and $|\mbf{J}_*({\frac{\partial}{\partial t}}|_{(x,t)})|=|\frac{\partial}{\partial t}|=1$, which makes 
$$\left|\mbf{J}^*\omega_{\op{D}^{\op{III}}_1}(\frac{\partial}{\partial x},\frac{\partial}{\partial t})\right|=e^{-t}\left|\omega_{\op{D}^{\op{III}}_1}(\mbf{J}_*({\frac{\partial}{\partial x}}|_{(x,0)}),\frac{\partial}{\partial t})\right|\leq e^{-t}C$$ for some universal constant $C$.
Hence
\begin{eqnarray}\label{fast}
\int_{C_i}|\mbf{J}^*\omega_{\op{D}^{\op{III}}_1}|\leq \int_{0}^{2\pi}\int_{0}^\infty C e^{-t}<\infty,\end{eqnarray} which makes
 $\mbf{J}^*(d\alpha_i)$ a $L^1$ form on  $C_i$. 
% Also $-d^c\phi_i=\alpha_i$ is of the form $dt\circ J$ ({\color{red}Justify this}), hence $$|\mbf{J}^*\alpha_i|\leq e^{-t} |\alpha_i((\mbf{J}_*({\frac{\partial}{\partial x}}|_{(x,0)}))|,$$ and it is a $L^1$ form also.
 
 If we decompose $C_i$ into two parts $C_i^1=S^1\times[0, t]\cup C_i^2=S^1\times [t,\infty)$, then $\mbf{J}^*(d\alpha_i)$ being a $L^1$ form on  $C_i$ implies that
 $\int_{C_i^2} \mbf{J}^*(d\alpha_i)\ra 0$ as $t\ra\infty$. Furthermore by noting that
 $|\mbf{J}_*({\frac{\partial}{\partial x}}|_{(x,t)})|=e^{-t}|\mbf{J}_*({\frac{\partial}{\partial x}}|_{(x,0)})|$, $\int_{S^1\times \{t\}} \mbf{J}^*\alpha_i=\int_{S^1\times \{0\}} e^{-t}\mbf{J}^*\alpha_i$,
   the ordinary Stokes' lemma holds to get for each parabolic boundary $c_i$
 $$\int_\Sigma d(\chi_i {Str(\bf J)}^*\alpha_i)=\int_{c_i} {Str(\bf J)}^*\alpha_i=-\int_{C_i} d(\chi_i {Str(\bf J)}^*\alpha_i).$$
This implies that
$$\op{T}(\Sigma,\rho)=\frac{1}{2\pi}\int_{\Sigma'} {Str(\mbf{J})}^*\omega_{\op{D}^{\op{III}}_1}.$$
Also note that we can deform $\triangle$ so that the triangles wrapping around the parabolic boundary $c_i$ can be straightened to the cone point of $C_i$ to include
$C_i$ in $\cup_{\sigma\in\triangle}\sigma$. We still denote the deformed triangulation by $\triangle$.
Since $\Sigma'=\sum_{\sigma_i\in\triangle} \sigma_i$,
$|\int_{\sigma_i} Str({\bf J})^* \omega_{\op{D}^{\op{III}}_1}|=| \int_{Str({\bf J})(\sigma_i)} \omega_{\op{D}^{\op{III}}_1}|\leq \pi$ and there are $-4+4g+2q$ triangles,
we get
$$|\op{T}(\Sigma,\rho)|\leq \frac{1}{2\pi}(-4+4g+2q)\pi=|\chi(\Sigma)|.$$

\subsection{Milnor-Wood inequality of Toledo invariant for general $n$}
In the arguments of the previous section, we only used the fact that $\int_{c_i} {\bf J}^*\alpha_i=0$ for elliptic and hyperbolic boundary $c_i$, and the extendability of ${\bf J}$ to parabolic cones $C_i$. For general $n$, we need an estimate $|\int_{\sigma_i} Str({\bf J})^* \omega_{\op{D}^{\op{III}}_n}|=| \int_{Str({\bf J})(\sigma_i)} \omega_{\op{D}^{\op{III}}_n}|\leq n\pi$, which follows from the fact that the Gromov norm of $\kappa^b_G\in H^2_{c,b}(G,\br)$ is $\frac{{\text{rank of}}\ G}{2}$ \cite{Cl}.

First note that by \cite[Corollary 9.4]{Dj} , for any $a\in \op{Sp}(2n,\br)$, $a^2=\exp{X}$ for some $X\in \mathfrak g$. 

Let $\pi_1(\Sigma)=\langle c_1, c_2,\cdots, c_{q-1}, a_i, b_i \rangle$ where $c_i$ are boundary components, and $c_q=\Pi c_i \Pi [a_i,b_i]$. Take an index 2 subgroup $\Gamma< \pi_1(\Sigma)$
containing $c_i^2,\ i=1,\cdots, q-1$. Let $\Sigma_1=\widetilde \Sigma/\Gamma$ and $p:\Sigma_1\ra\Sigma$ be a covering map of degree 2.  Then for each boundary $c_i\ i=1,\cdots, q-1$, there exists a corresponding boundary $B_i$ of $\Sigma_1$ such that $p(B_i)=c_i^2$.  Now it is possible that
$p^{-1}(c_q)$ might be disjoint union of two circles $B_q, B_{q+1}$ which map to $c_q$ homeomorphically.  In this case, take a double cover $\Sigma_2$ of $\Sigma_1$, which corresponds to an index two subgroup of $\pi_1(\Sigma_1)$ containing $B_q^2, B_{q+1}^2$.
Then there exist two boundary components of $\Sigma_2$, $C_q,C_{q+1}$ which project to $B_q^2,
B_{q+1}^2$. Then the covering map $f$ from $\Sigma_2$ to $\Sigma$ has the property that each boundary component of $\Sigma_2$ projects to $c_i^{2k}$ for some $i=1,\cdots,q$.

Now for  this 4-fold covering map $f:\Sigma_2\ra\Sigma$, 
consider the induced representation $\rho_2=\rho\circ f_*:\pi_1(\Sigma_2)\ra \op{Sp}(2n,\br)$. Then for any boundary component $b$ of $\Sigma_2$, $\rho_2(b)=\rho(c_i^{2k})$ for some $c_i$, hence $\rho_2(b)=\exp(2\pi B)$ for some $B$ in the Lie algebra of $\op{Sp}(2n,\br)$. This allows us to define a complex structure
$\mbf{J}(x)=\exp(-xB)J\exp(xB)$ along any boundary of $\Sigma_2$. Now
if prove the Milnor-Wood inequality for $\rho_2$, then we prove the Milnor-Wood inequality for $\rho$ since both Toledo invariant and the Euler number for $\Sigma_2$ is $4$ times those of $\Sigma$. Hence we may assume that $\rho(c_i)=\exp(2\pi B_i)$ for some $B_i$.

Now we mention the classification of isometries for a symmetric space $X$ of noncompact type with $G=\op{Iso}^0(X)$. Let $\ell(\phi)=\inf_{x\in X} d_X(x,\phi(x))$ for $\phi\in G$. We say $\phi$ is (\cite[1.9.1]{Eb} )
\begin{enumerate}
\item  axial if $\ell(\phi)>0$ and realized in $X$.
\item elliptic if $\ell(\phi)=0$ and realized in $X$, i.e., it has a fixed point in $X$.
\item parabolic if $\ell(\phi)$ is not realized in $X$.
\end{enumerate}

If $\phi$ is parabolic, it has a fixed point at $X(\infty)$ (\cite[Prop. 4.1.1]{Eb} ). Hence it stabilizes a horosphere $H$ based at a fixed point of $\phi$. 

Let $Ad:G\ra \op{End}(\mathfrak g)$ be an adjoint representation and $G=KAN$ an Iwasawa decomposition. For any element $g\in G$, $g$ has a unique Jordan decomposition $g=ehu$ where $Ad(e)$ is diagonalizable with modulus 1 (complex) eigenvalues (conjugate to an element in $K$), $Ad(h)$ is diagonalizable with positive eigenvalues (conjugate to an element in $A$), and $Ad(u)$ is unipotent, i.e. $Ad(u)-I$ is nilpotent (conjugate to an element in $N$).
Then $\phi\in G$ is axial if and only if $Ad(\phi)$ is nonelliptic and semisimple (i.e. diagonalizable) and $\phi\in G$ is parabolic if and only if $Ad(\phi)$ is not semisimple (i.e. not diagonalizable). See \cite[Prop. 2.19.18]{Eb} .

In our case, $G=\op{PSp}(2n,\br)=\op{Sp}(2n,\br)/\pm I$. In view of Definition \ref{defhyperell}, an elliptic element is the one conjugate to $U(n)$ (hence, diagonalizable with modulus 1 eigenvalues), an axial element is the one which is diagonalizable, and  some  eigenvalues of which are not of modulus 1. Hence 
by Proposition \ref{dechyperellu}, an axial element $L\in \op{Sp}(2n,\br)$ has an orthogonal symplectic decomposition into hyperbolic parts and elliptic parts.
But in this case, $\int_{c_i}{\mbf{J}}^*\alpha_i=0$. 

If $L\in\op{PSp}(2n,\br)$ is parabolic, then it has a symplectic decomposition $L=L_s\oplus L_u$ such that the nondiagonalizable part $L_u$ of $L$ is unipotent, hence 
 $L_u$ is an element of a horospherical subgroup $N_x$ for some point $x\in X(\infty)$, see \cite[Prop. 2.19.18 (5)]{Eb} . But $N_x$ has a property that  for $g\in N_x$
 $$\lim_{t\ra \infty} e^{-tX} g e^{tX}=id,$$ where $X\in \mathfrak p$ is a unit vector whose infinite end point is $x$ in the Cartan decomposition $\mathfrak g=\mathfrak t\oplus\mathfrak p$ at $p\in X$. This implies that for $g\in N_x$
 $$\lim_{t\ra\infty} d(e^{tX} p, g e^{tX}p)=0,$$ which means that any two geodesic rays starting from $p$ and $gp$ pointing forwards $x$, get closer exponentially fast.
 This shows that  the estimate (\ref{fast}) is valid for parabolic boundary in general dimension.
 
Hence it suffices to attach cones
to parabolic boundaries of $\Sigma$ and argue as in the previous section, to show that
$$\op{T}(\Sigma,\rho)=\frac{1}{2\pi}\int_{\Sigma'} {Str(\mbf{J})}^*\omega_{\op{D}^{\op{III}}_1}.$$
Then using $| \int_{Str({\bf J})(\sigma_i)} \omega_{\op{D}^{\op{III}}_n}|\leq n\pi$, we are done.

%Since it is difficult to classify the isometries for $n>1$, we use the continuity of Toledo invariant to prove the Milnor-Wood inequality in general case $n\geq 1$. Let $tr:\op{Hom}(\pi_1(\Sigma),  \op{Sp}(2n,\mb{R}))\ra \br^{2qn}$ be a function associating $\{(|\lambda_1^i|,|\lambda_2^i|,\cdots,|\lambda_{2n}^i|)\}$ of the absolute values of the eigenvalues of $\rho(c_i)$ in a decreasing order to a representation $\rho$. Since the eigenvalues vary continuously on $\op{Sp}(2n,\br)$, $tr$ is a continuous function. Then $tr^{-1}( |\lambda_j^i|\neq 1)$ is an open and dense(?) set. As we saw in Proposition \ref{Hyp}, $|\op{T}(\Sigma,\rho)|\leq n|\chi(\Sigma)|$ for a representation $\rho$ such that $\rho(c_i)$ are all hyperbolic. By this, the general Milnor-Wood inequlity holds.

\vspace{5mm}
\noindent{\it Acknowledgements.} The third author would like to thank Dr. Xiangsheng Wang for helpful discussions on Atiyah-Patodi-Singer index theorem.

\end{document}